\documentclass[11pt, twoside,a4paper]{article}
\usepackage{flafter,amsmath,amssymb,latexsym,psfrag,graphicx,color,indentfirst,bm,amsthm}
\usepackage[T1]{fontenc}
\usepackage[utf8]{inputenc}
\usepackage[colorlinks,
linktocpage=true,
linkcolor=blue,
citecolor=green,
urlcolor=blue,
anchorcolor=black
]{hyperref}
\usepackage[a4paper,top=2.3cm,bottom=2.3cm,left=1.8cm,right=1.8cm, marginparwidth=1.75cm]{geometry}
\usepackage{float}
\usepackage[makeroom]{cancel}
\usepackage{mathtools}
\usepackage[numbers,sort&compress]{natbib}
\usepackage{appendix}
\usepackage{multirow}
\usepackage{caption}
\usepackage{subfigure}
\usepackage{graphicx}
\usepackage{dutchcal}
\usepackage{comment}
\allowdisplaybreaks
\usepackage{tikz}
\tikzset{mynode/.style={inner sep=2pt,fill,outer sep=0,circle}}
\usetikzlibrary{fadings}
\usepackage{amsmath,amssymb,braket,tikz}
\usetikzlibrary{tikzmark,calc}
\usepackage{ebgaramond}
\usepackage[sfdefault]{FiraSans}
\usepackage{microtype}

\newtheorem{theorem}{Theorem}[section]
\newtheorem{lemma}{Lemma}[section] 
\newtheorem{corollary}{Corollary}[section] 
\newtheorem{proposition}{Proposition}[section] 
\newtheorem{remark}{Remark}[section]  

\numberwithin{equation}{section}

\usepackage{mathrsfs}
\newsavebox\foobox
\newlength{\foodim}
\newcommand{\slantbox}[2][0]{\mbox{%
        \sbox{\foobox}{#2}%
        \foodim=#1\wd\foobox
        \hskip \wd\foobox
        \hskip -0.5\foodim
        \pdfsave
        \pdfsetmatrix{1 0 #1 1}%
        \llap{\usebox{\foobox}}%
        \pdfrestore
        \hskip 0.5\foodim
}}
\def\Laplace{\slantbox[-.45]{$\mathscr{L}$}}
\def\Fourier{\slantbox[-.45]{$\mathscr{F}$}}
\def\bsq#1{
\lq{#1}\rq}

\DeclareMathOperator*{\res}{Res}
\numberwithin{equation}{section}
\numberwithin{figure}{section}

\DeclareRobustCommand{\rchi}{{\mathpalette\irchi\relax}}
\newcommand{\irchi}[2]{\raisebox{\depth}{$#1\chi$}}
\newcommand\obullet[1]{\ThisStyle{\ensurestackMath{%
  \stackon[1pt]{\SavedStyle#1}{\SavedStyle\kern.6\LMpt\bullet}}}}
\newcommand\ocirc[1]{\ThisStyle{\ensurestackMath{%
  \stackon[1pt]{\SavedStyle#1}{\SavedStyle\kern.6\LMpt\circ}}}}
  
\title{Effective Medium Theory for Heat Generation Using Plasmonics:\\  {\Large {A Parabolic Transmission Problem Driven by the Maxwell System}}}

\author{Xinlin Cao\footnote{Department of Applied Mathematics, The Hong Kong Polytechnic University, Hong Kong SAR. Email: xinlin.cao@polyu.edu.hk.},\ Arpan Mukherjee\footnote{Joint Research Center of Applied Mathematics, Shenzhen MSU-BIT University, Shenzhen, People's Republic of China (arpanmath99@alumni.iitm.ac.in and arpan.mukherjee@oeaw.ac.at).} \ and Mourad Sini\footnote{Radon Institute (RICAM), Austrian Academy of Sciences, Altenbergerstrasse 69, A-4040, Linz, Austria (mourad.sini@oeaw.ac.at). This author is partially supported by the Austrian Science Fund (FWF): P36942.}}
\date{\today}

\begin{document}
\maketitle
\begin{abstract}    

The excitation of plasmonic nanoparticles by incident electromagnetic waves at frequencies near their subwavelength resonances induces localized heat generation in the surrounding medium. We develop a mathematical framework to rigorously quantify this heat generation in systems of arbitrarily distributed nanoparticles.
\begin{enumerate}
    \item For an arbitrary discrete distribution of $M$ nanoparticles within a bounded domain, the effective heat distribution is described by a coupled system: Volterra-type integral equations for the heat conduction and a Foldy-Lax-type system governing the self-consistent electric field intensities. These equations are parameterized by the particle geometries and the local electromagnetic field interactions. The effective heat generation is computed by solving these coupled systems, with the computational complexity scaling as $\mathcal{O}(M^2)$.

    \item In the case $M \gg 1$, under natural scaling regimes, the discrete system converges to a continuum model, yielding an effective parabolic equation for the heat distribution. The source term in this homogenized parabolic model is characterized by the solution of the homogenized Maxwell’s equations, incorporating an effective permittivity distribution derived from the Drude model under resonance conditions.
\end{enumerate}
\noindent
Our analysis utilizes advanced tools in potential theory, asymptotic analysis and homogenization. By leveraging layer potential representations, we rigorously characterize subwavelength plasmonic resonances and derive point-wise field approximations. The coupling between the Maxwell and heat equations is resolved by analyzing the spectral properties of the nanoparticles and their scaling limits.
\vspace{0.05in}

\noindent
This framework reduces the problem to two mathematical challenges: a control problem for the effective parabolic system and an internal phase-less inverse problem for the Maxwell system, thus providing a unified approach to modeling heat generation in nanoparticle clusters.

\bigskip

\bigskip


\end{abstract}     

\newpage

\tableofcontents

        \section{Introduction} 

    \subsection{Motivation} 

The interaction of electromagnetic waves with plasmonic nanoparticles has become a cornerstone of modern nanotechnology, enabling transformative advancements across various fields such as photothermal therapy, thermal imaging, energy harvesting, and sensing \cite{Challener, Han-al, Prost-al, mourad, new-1}. These nanoparticles, typically made of metals like gold, silver, or copper, exhibit a unique ability to concentrate electromagnetic fields at their surfaces due to plasmonic resonances. This ability results in highly localized heating when the nanoparticles absorb light, a phenomenon that has garnered significant attention due to its vast range of potential applications. The ability of absorbing media to convert light into heat is primarily governed by the Joule effect, in which energy from absorbed electromagnetic radiation is converted to thermal energy. In weakly absorbing media, this conversion is often inefficient and can limit the effectiveness of heat-based applications. To overcome this limitation, the engineering community has proposed enhancing electromagnetic-to-heat conversion by introducing nanoparticles into these media. By localizing and concentrating light at the nanoparticle surfaces, these materials facilitate a much more efficient heat generation process, thus opening the door to numerous potential applications in biomedical engineering, environmental monitoring, and advanced material science \cite{baffou, baffou-2, Challener, Han-al, Prost-al, Zograf-al-1, Zograf-al-2}.
\vspace{0.3cm}

\noindent 
Plasmonic nanoparticles exhibit several key optical properties that distinguish them from conventional materials. These include significant absorption, scattering, and field enhancement due to their ability to resonate at specific frequencies known as plasmonic resonances. These resonances are associated with the spectrum of the Neumann-Poincaré operator or the related magnetization operator in the resolvent model for electromagnetic wave propagation \cite{new-4, Arpan-Sini}. When nanoparticles are excited by electromagnetic waves near their plasmonic resonant frequencies, the resulting enhancement of the electric field near the nanoparticle surface can generate significant localized heat. The combination of this field enhancement and the absorption of electromagnetic energy results in the conversion of light into heat, often with high efficiency, even at subwavelength scales. The physics behind plasmonic resonances is rooted in the collective oscillation of free electrons in the metal at specific frequencies. These resonances depend on various factors, including the nanoparticle's size, shape, material composition, and surrounding medium. This fundamental understanding allows us to design nanoparticles that can be "tuned" to absorb light at specific wavelengths, thus optimizing heat generation for targeted applications such as cancer therapy, where precise heat delivery to tumor sites is crucial.
\vspace{0.3cm}

\noindent 
Building on these principles, prior works have estimated heat generation near a single nanoparticle embedded in a background medium \cite{ProfHabib, Arpan-Sini, AM5}. However, the challenge arises when trying to understand and quantify the behavior of nanoparticle ensembles or arrays, where the interactions between multiple particles and their collective response to electromagnetic fields can lead to complex and highly nonlinear behaviors. Despite the growing body of research, a comprehensive framework that couples the Maxwell equations, governing electromagnetic wave propagation, with the heat generation mechanisms driven by plasmonic excitations, remains under explored—especially in frequency regimes where plasmonic effects are most pronounced.
\vspace{0.3cm}

\noindent 
This work addresses this gap by formulating and analyzing a parabolic transmission problem that models the heat generation process in the context of effective medium theory (EMT). EMT provides a way to describe the collective behavior of systems of interacting particles by averaging out the microscopic details, offering an effective description that simplifies the overall system behavior. By coupling Maxwell's equations with heat generation, this study not only advances our understanding of thermo-plasmonic phenomena but also contributes to the broader field of homogenization theory, which seeks to understand the macroscopic behavior of complex materials from their microscopic properties. Specifically, we focus on two primary regimes:
\begin{enumerate}
      \item Discrete distribution of nanoparticles: In this case, nanoparticles are distributed arbitrarily within the medium. We develop a closed-form expression for the heat generated by such distributions, accounting for both the electromagnetic interactions and the heat diffusion processes. This model is referred to as the discrete effective model for light-to-heat conversion. It provides a powerful tool for predicting the heat generation in systems where nanoparticles are dispersed in a non-uniform manner, such as in biological tissues, with applications ranging from photo-thermal therapy to thermal imaging. This model, outlined in Theorem \ref{th1}, serves as a foundation for understanding nanoparticle interactions at the microscopic scale.
      \item Continuous distribution of nanoparticles: In this regime, nanoparticles are distributed periodically or in a regular, structured manner, which could mimic nanoparticle arrays or engineered meta-materials. The collective behavior of nanoparticles in such systems can be described by equivalent material properties, such as the effective electric permittivity and thermal conductivity. The continuous effective model for light-to-heat conversion, presented in Theorem \ref{non-periodic}, allows for the characterization of large-scale nanoparticle systems, such as those used in photonic crystals. These systems have vast potential in applications where large-area, uniform heating is needed, such as solar energy harvesting, and advanced coating technologies.
\end{enumerate}
\vspace{0.2cm}

\noindent 
With such characterizations, we transform the problem of heat generation, using nanoparticles, into a {\it{control problem}}, for the heat equation, using external sources. These sources can be generated as an {\it{inverse  problem}}, for the Maxwell system, of recovering a permittivity needed to generate the (internal but phaseless) data given by the intensity of the electric field inside the domain of interest. Finally, this needed permittivity profile can be generated by design of proper nanoparticles.
\vspace{0.3cm}

\noindent 
This approach provides a novel means of controlling heat generation in plasmonic nanoparticle systems. Furthermore, the results presented here, where we focus on homogeneous background media, lay the foundation for future work on heterogeneous media, where complex nanoparticle arrangements and variations in the surrounding material properties must be accounted for. This includes examining the impact of particle shape, orientation, and clustering effects on heat generation, as well as the role of environmental factors such as temperature gradients and material heterogeneity.

         \subsection{The Mathematical Models and Related Asymptotic Regimes}  

    \subsubsection{The Mathematical Models}   

The objective of this work is to present a detailed analysis of a mathematical model describing the photo-thermal effect in a system composed of \( M \) plasmonic nanoparticles, given by
\[
D := \bigcup_{j=1}^M D_j,
\]
where each \( D_j \) is defined as \( D_j = z_j + \delta B_j \) with \(\delta \ll 1\). Here, \( B_j \) is the $\mathcal{C}^2$-regular domain centered at the origin with volume \(\textit{Vol}(B_j) \sim 1\), and \( z_j \) represents the position of each nanoparticle.
\par\noindent
Let us define the volumetric heat capacity \( c_p \), thermal conductivity \( \gamma_p \), and electric permittivity \( \varepsilon_p \) of the nanoparticles. Similarly, let \( c_m \), \( \gamma_m \), and \( \varepsilon_m \) denote the corresponding properties of the homogeneous background medium \(\mathbb{R}^3 \setminus \overline{D}\), which are assumed to be constant and positive. We define the effective parameters over the entire space by
    \begin{align}
        \mathrm{c}_v = \mathrm{c}_p \chi(D) + \mathrm{c}_m \chi(\mathbb{R}^3 \setminus \overline{D}),\;
        \gamma = \gamma_p \chi(D) + \gamma_m \chi(\mathbb{R}^3 \setminus \overline{D}),\; \text{and}\;
        \varepsilon = \varepsilon_p \chi(D) + \varepsilon_m \chi(\mathbb{R}^3 \setminus \overline{D}),
    \end{align}
where \( \chi \) denotes the characteristic function of a given domain. We assume that the nanoparticles are non-magnetic, permitting us to set the magnetic permeability to a uniform constant across \(\mathbb{R}^3\), i.e., \(\mu = 1\) throughout. The homogeneous electromagnetic scattering problem for the total electric field \( E := E^{\textit{sc}} + E^{\textit{in}} \) is formulated as follows:
    \begin{align}\label{Maxwell-model}
        \begin{cases}
            \textit{curl} \, \textit{curl} \; E - k^2 \varepsilon E = 0\; \text{in }\ \mathbb{R}^3, \\
            E^{\textit{sc}} := E - E^{\text{in}} \text{ satisfies the Silver-Müller radiation condition (S-M.R.C.):} &\\
            \lim_{|x|\to +\infty} \left(\text{curl}\, E^{\textit{sc}}(x) \times x - i k \sqrt{\varepsilon} |x| E^{\textit{sc}}\right) = 0.
        \end{cases}
    \end{align}
We also consider $E^\textit{in}$ to be the incident plane wave satisfying 
$E^\textit{in} = E_0^\textit{in} e^{i k^w\; \theta\cdot\mathrm{x}}$, with the direction of wave propagation $\vartheta\in \mathbb{S}$ (unit sphere in $\mathbb{R}^3$), $E_0^\textit{in}\in \mathbb{S}$ is the polarization vector satisfying $\theta\cdot E_0^\textit{in} = 0$ and $k^w= k \sqrt{\varepsilon\mu}$ is the wave number with the incidence frequency $k$. The electric permittivity \(\varepsilon_p\) is modeled by the Lorentz model, given by
    \begin{align}
         \varepsilon_p(k, \zeta) = \varepsilon_\infty \Big(\varepsilon_0 + \frac{k_p^2}{k_0^2 - k^2 - i \zeta k}\Big),
    \end{align}
where \(k_p\) is the plasma frequency, \(k_0\) the undamped resonance frequency, \(\zeta\) the damping parameter, and \(\varepsilon_\infty\) the electric permittivity of the vacuum. The parameter \(\varepsilon_0\) describes the contribution of bound electrons to polarizability. Considering the background medium is non-dispersive, the parameters \(c_m\), \(\gamma_m\), and \(\varepsilon_m\) are assumed to be independent of the incident frequency \(k\).
\par\noindent
The photo-thermal effect is modeled by the following transmission parabolic problem with a source term coupled to the Maxwell system , see \cite{ProfHabib, baffou, Triki-Vauthrin} for instance,
    \begin{align}\label{heat}
         \begin{cases}
             \mathrm{c}_v \frac{\partial u}{\partial t} - \nabla \cdot (\gamma \nabla u) = \frac{k}{2\pi} \Im(\varepsilon) |E|^2 f(t), & \text{in } (\mathbb{R}^3 \setminus \partial D) \times (0, T),\\
             u\big|_{+} = u\big|_{-}, & \text{on } \partial D_i, \\
             \gamma_m^{-1} \partial_\nu u\big|_{+} = \gamma_p^{-1} \partial_\nu u\big|_{-}, & \text{on } \partial D_i, \\
             u(x, 0) = 0, & x \in \mathbb{R}^3
        \end{cases}
    \end{align}
where \(T \in \mathbb{R}\) is the final measurement time, and \(\partial_\nu \big|_{\pm}\) denotes the limit expression:
\(\partial_\nu u\big|_{\pm}(x, t) = \lim_{h \to 0} \nabla u(x \pm h \nu_x, t) \cdot \nu_x,\), where \(\nu\) represents the outward normal vector to \(\partial D\).
\vspace{0.2cm}

\noindent 
The term $\frac{k}{2\pi} \Im(\varepsilon) |E|^2$ models the amount of energy that is deposited on the nanoparticle via a laser excitation, via the pointing vector, see \cite{baffou}. The electric field $E$ is related to the time-harmonic regime. Ideally, this would be in the time-domain. However, if the used incident sources, creating the electric field, is chosen having a bandwidth near the plasmonic resonances, described later, then the modes near these resonant frequencies would be the dominating modes. Therefore, using incident electric field, through the laser excitation, with such bandwidths, it is enough to consider the time-harmonic modes.
\vspace{0.2cm}

\noindent 
The term \(f\)  models the time-modulation of the laser excitation that allows the attenuation, in time, of excitation. This can be a pulse distribution, i.e. $f:=\delta_0$ which is the Dirac distribution supported at the origin, \cite{Triki-Vauthrin}. However, here we take a smoother modulation given by a function with initial conditions satisfying \(\partial_t^{r+1} f(0) \neq 0\) and \(\partial_t^n f(0) = 0\) for all \(n \in \{0, 1, \ldots, r\}\). An example can be constructed by taking \(\varphi \in C^\infty_c(\mathbb{R})\) such that \(\varphi = 1\) on \([0, l/2]\) and \(\text{supp}(\varphi) \subseteq [0, l]\) with $l$ small if needed. Define $f(t) :=t^r \varphi(t)$.
\vspace{0.2cm}

\noindent 
Note that \(\Im(\varepsilon) = 0\) in \((\mathbb{R}^3 \setminus \overline{D})\). This last condition means that the background is not absorbing. Therefore, the nanoparticle, which is absorbing, is the only responsible for converting the exciting electric field into heat.  Notice that, we could also take absorbing background but in this case, one needs to use dielectric nanoparticles and not plasmonic ones, which means, based on the Lorentz model above, that we need to use incident frequencies close to the undamped resonance frequency $k_0$ and not the plasma frequency $k_p$. Indeed, for such frequencies, we have \(\Im(\varepsilon) \gg 1\) which then dominates the one of the background and hence enhances more the amount of the electric to heat conversion. The question of comparing the effects due to plasmonic and the ones due to dielectric nanoparticles will not be discussed in more details in this work, but the interested reader can see a related discussion in \cite{Arpan-Sini}.

\begin{remark}   

To describe the electromagnetic as well as the heat properties of the nano-particles, we can also use the Drude model which is an approximation, as the undamped resonant frequency \(k_0\)  is small, in the Lorentz model, reducing it to
    \begin{align}\label{Drude-model}
          \varepsilon_p(k, \zeta) = \varepsilon_\infty \Big(\varepsilon_0 - \frac{k_p^2}{k^2 + i \zeta k}\Big).
    \end{align}
This model provides a reasonable approximation of the optical properties in many metals over a wide frequency range. For example, for gold, the parameters \(\varepsilon_0 = 9.84\) eV, \(k_p = 9.096\) eV, and \(\zeta = 0.072\) eV yield a dielectric constant that aligns closely with experimental values of frequency from 0.8 eV to 4 eV \cite{drude3}. In practice, We consider metallic nanoparticles in a biological medium, such as tissue. For instance, the thermal conductivity of gold is \(\gamma_p = 318 \, \text{W/(m K)}\), whereas that of human hand tissue is \(\gamma_m = 0.963 \, \text{W/(m K)}\). Additionally, for human hand tissue, we have \(\gamma_m \, c_m = 1.592 \, \text{W}^2 \text{s/m}^4 \text{K}^2\), see for instance \cite{biotissue}.

\end{remark}    


    \subsubsection{The Related Asymptotic Regimes}    
    
To streamline our mathematical analysis based on the experimental values of the heat and electromagnetic parameters discussed in the previous remark, we adopt the following asymptotic regime:
\begin{equation}\label{scales}
    \zeta \sim \delta^h, \quad \gamma_p \sim \delta^{-\beta}, \quad \text{and} \quad c_p \sim 1, \quad \text{with} \quad \delta \ll 1,
\end{equation}
where \( h \) and \( \beta \) are positive real constants. Observe that for moderate incident frequencies $k$, we deduce the above scale of $\zeta$ and (\ref{Drude-model}) that
\begin{align}\label{scales2}
    \Im(\varepsilon_p) \sim \delta^h.
\end{align}
Additionally, we assume that the parameters associated with the homogeneous background medium, \( c_m \), \( \gamma_m \), and \( \varepsilon_m \), remain of order 1. 
\vspace{0.2cm}

\noindent
\textit{\textbf{Definition.}} Let \( D \) represent the union of the nanoparticles, defined by \( D := \bigcup\limits_{j=1}^M D_j \). We introduce the following parameters:

\noindent 1. The parameter \( \delta \), representing the maximum diameter among all distributed nanoparticles, given by
   \[
   \delta := \max_{1 \leq j \leq M} \operatorname{diam}(D_j).
   \]
\noindent 2. The parameter \( d \), representing the minimum distance between any two distributed nanoparticles, defined as
   \[
   d := \min_{1 \leq i,j \leq M} \operatorname{dist}(D_i, D_j), \quad i \neq j.
   \]
\par\noindent
In this work, we focus on the following regimes to model clusters of nanoparticles distributed in a three-dimensional bounded domain:
\begin{align} \label{regimes}
    M \sim d^{-3} \quad \text{and} \quad d \sim \delta^{\lambda}, \quad \lambda\geq 0, \quad \delta \ll 1.
\end{align}
The appropriate choices of $\lambda$ will be specified later.
\vspace{0.2cm}

\noindent
We, next review some basic results on the Fourier-Laplace transform, as outlined in \cite{lubich, sayas}.


    \subsection{Function Spaces and Preliminaries on Fourier-Laplace Transforms}\label{lap-fou} 

We begin by considering \( F(\bm{s}): \mathbb{C}_+ := \{\bm{s} \in \mathbb{C} : \Re \bm{s} > 0\} \to X \), where \( X \) is a Hilbert space, as an analytic function that is assumed to be polynomially bounded. Specifically, for \( j \in \mathbb{R} \), and for all \( \sigma \), there exists a constant \( C_F(\sigma) < \infty \) such that  
\begin{equation}\label{1.1}
    \Vert F(\bm{s}) \Vert_X \leq C_F(\Re \bm s) |\bm{s}|^j, \quad \text{for} \ \Re\bm{s} \geq \sigma > 0.
\end{equation}
Now, for \( \bm{s} \in \mathbb{C}_* := \mathbb{C} \setminus (-\infty, 0] \), we have \( \bm{s}^{\frac{1}{2}} := |\bm{s}|^{\frac{1}{2}} e^{\frac{i}{2} \text{Arg}(\bm{s})} \in \mathbb{C}_+ \). Let us also introduce the notations
\begin{equation}\label{1.13}
    \omega := \Re \bm{s}^{\frac{1}{2}} = \Re \overline{\bm{s}}^{\frac{1}{2}}, \quad \text{and} \quad \underline{\omega} := \min\{1, \Re \bm{s}^{\frac{1}{2}}\}.
\end{equation}
This setup arises from the symbol \( \bm{s} \) after applying the Fourier-Laplace transform to the heat equation. Here, \( F(\bm{s}) \) is well-defined and analytic in \( \mathbb{C}_* \), and satisfies the following condition:
\begin{equation}\label{1.2}
    \Vert F(\bm{s}) \Vert_X \leq D_F(\Re \bm{s}^{\frac{1}{2}}) |\bm{s}|^j, \quad \text{for} \ \Re\bm{s}^{\frac{1}{2}} = \omega > 0,
\end{equation}
where \( D_F \) shares a similar boundedness property as \( C_F \). Noting that
\begin{equation}\nonumber
    \min\{1, \Re \bm{s}^{\frac{1}{2}}\} \geq \min\{1, \Re \bm{s}\}, \quad \text{for} \ \bm{s} \in \mathbb{C}_+,
\end{equation}
we observe that if \( F(\bm{s}) \) satisfies the bound in (\ref{1.1}), it also satisfies the bound in (\ref{1.2}) with an appropriate choice of \( C_F \) (see, for instance, \cite{sayas2}). These bounds ensure that \( F \) is the Laplace transform of a distribution of finite order of differentiation, with support on the positive real line.
\\
We now define some necessary function spaces. Let \( \mathcal{D}'(\mathbb{R}; X) \) represent the space of \( X \)-valued distributions, and \( \mathcal{S}'(\mathbb{R}; X) \) denote the space of tempered distributions on the real line for a Banach space \( X \). We then introduce the space
\[
    \mathcal{L}'(\mathbb{R}, X) := \left\{ f \in \mathcal{D}'(\mathbb{R}, X) : e^{-\sigma t} f \in \mathcal{S}'(\mathbb{R}_+, X) \ \text{for some} \ \sigma > 0 \right\}.
\]
This definition allows us to consider the Laplace transform of a function \( f \in \mathcal{L}'(\mathbb{R}, X) \), such that \( f \in L^1(\mathbb{R}, X) \), given by
\[
    \hat{f}(x, \bm{s}) = \mathcal{L}[f](\bm{s}) := \int_{-\infty}^{+\infty} f(x, t) \exp(-\bm{s} t) \, dt, \quad \text{for a.e.} \ \bm{s} \in \mathbb{C}_\sigma := \{ \bm{s} \in \mathbb{C} : \Re \bm{s} = \sigma > 0 \}.
\]
For a Sobolev space \( X \) and \( r \in \mathbb{R} \), we further define the following anisotropic Hilbert space:
\[
    H_\sigma^r(\mathbb{R}; X) := \left\{ f \in \mathcal{L}'(\mathbb{R}, X) : \int_{-\infty+i\sigma}^{+\infty+i\sigma} |\bm{s}|^{2r} \Vert \mathcal{L}[f](\bm{s}) \Vert_X^2 \, d\bm{s} < \infty \right\},
\]
with the norm
\[
    \Vert f \Vert_{H_\sigma^r(\mathbb{R}; X)} = \left( \int_{-\infty+i\sigma}^{+\infty+i\sigma} |\bm{s}|^{2r} \Vert \mathcal{L}[f](\bm{s}) \Vert_X^2 \, d\bm{s} \right)^{\frac{1}{2}}.
\]
By Plancherel’s theorem, we know that \( \Vert \mathcal{L}[f](\bm{s}) \Vert_X = \Vert e^{-\sigma t} f(t) \Vert_X^2 \), which allows us to relate the norms in \( H_\sigma^p(\mathbb{R}; X) \) with the weighted norm of \( \mathcal{L}[f](\bm{s}) \) in \( X \). This connection is useful for deriving bounds on time-dependent functions. For instance,
\[
    \int_{-\infty+i\sigma}^{+\infty+i\sigma} \Vert \mathcal{L}[f](\bm{s}) \Vert_X^2 \, d\bm{s} = \int_{-\infty}^{+\infty} e^{-2\sigma t} \Vert f(t) \Vert_X^2 \, dt.
\]
We now recall the following lemma, which will be essential for establishing the boundedness of certain linear operators. The lemma is as follows:

\begin{lemma}\label{lubich}\cite[Lemma 2.1]{lubich}
If \( F(\bm{s}) \) is bounded as in equation (\ref{1.1}) within the half-plane \( \Re \bm{s} \geq \sigma > 0 \), then the operator \( F(\partial_t) \) extends by density to a bounded linear operator
\[
    F(\partial_t) : H_0^{r+j}(0,T; X) \to H^r_0(0,T; Y),
\]
for any \( r \in \mathbb{R} \). Here, \( H^r(\mathbb{R}; X) \) denotes the Sobolev space of order \( r \) for \( X \)-valued functions on \( \mathbb{R} \), and on finite intervals \( (0, T) \), we define
\[
    H_0^r(0, T; X) := \left\{ g|_{(0,T)} : g \in H^r(\mathbb{R}; X), \ \text{with} \ g \equiv 0 \ \text{on} \ (-\infty, 0) \right\}, \quad r \in \mathbb{R}.
\]
\end{lemma}
\noindent
Additionally, note that for \( q \in \mathbb{Z}_+ \) and \( r > \frac{1}{2} \), there exists a continuous embedding
\[
    H_0^{q+r}(0, T; X) \subset \mathcal{C}^q(0, T; X),
\]
which ensures that functions in \( H_0^{q+r}(0, T; X) \) exhibit sufficient regularity to belong to the space of \( k \)-times continuously differentiable functions over \( (0, T) \).
\\
Next, we recall the decomposition, see \cite{raveski} for more details,
\begin{align}\label{decomposition-introduction}
\big(\mathrm{L}^{2}(D)\big)^3 = \mathbb{H}_{0}(\textit{div}\ 0,D) \oplus\mathbb{H}_{0}(\textit{curl}\ 0,D)\oplus \nabla \mathbb{H}_{\textit{arm}},
\end{align}
where 
\begin{align}
    \begin{cases}
        \mathbb{H}_{0}(\textit{div}\ 0,D) = \Big\{ u\in \mathbb{H}(\textit{div},D): \textit{div}\;u =0\;\text{in}\; D \; \text{and}\;   u\cdot \nu = 0 \; \text{on}\; \partial D\Big\}, \\
        \mathbb{H}_{0}(\textit{curl}\ 0,D) = \Big\{ u\in \mathbb{H}(\textit{curl},D): \textit{curl}\;u =0\;\text{in}\; D \; \text{and}\; u\times \nu = 0 \; \text{on}\; \partial D\Big\}, \; \text{and} \\
        \nabla \mathbb{H}_{\textit{arm}} = \Big\{ u \in \big(\mathbb{L}^{2}(D)\big)^3: \exists \  \varphi \ \text{s.t.} \ u = \nabla \varphi,\; \varphi\in \mathbb{H}^1(D)\; \text{and} \ \Delta \varphi = 0 \Big\}.
    \end{cases}
\end{align}
For a given vector field \( f \), we define the Magnetization operator and the Newtonian potential operator as follows
\begin{align}\nonumber
    \mathbb{M}^{(k)}_{D}\big[f\big](\mathrm{x}) := \nabla \int_{D} \nabla \mathbcal{G}^{(k)}(\mathrm{x}, \mathrm{y}) \cdot f(\mathrm{y}) \, d\mathrm{y} \quad \text{and} \quad \mathbb{N}^{(k)}_{D}\big[f\big](\mathrm{x}) := \int_{D} \mathbcal{G}^{(k)}(\mathrm{x}, \mathrm{y}) f(\mathrm{y}) \, d\mathrm{y},
\end{align}
respectively, where \(\mathbcal{G}^{(k)}\) is the Green's function associated with the Helmholtz operator. Furthermore, these operators are similarly defined for the case \( k = 0 \), corresponding to the Laplace operator. Then, it is well known, see for instance \cite{friedmanI}, that the  Magnetization operator $\mathbb{M}^{(0)}_{D}: \nabla \mathbb{H}_{\textit{arm}}\rightarrow \nabla \mathbb{H}_{\textit{arm}}$ induces a complete orthonormal basis namely $\big(\lambda^{(3)}_{\mathrm{n}},\mathrm{e}^{(3)}_{\mathrm{n}}\big)_{\mathrm{n} \in \mathbb{N}}$. Related to the decomposition (\ref{decomposition-introduction}), we define $\overset{3}{\mathbb{P}}$ to be the natural projector as $\overset{3}{\mathbb{P}}:= \mathbb{L}^2 \to \nabla \mathbb{H}_{\textit{arm}}.$
\vspace{0.1cm}

\noindent
Throughout this paper, we use the notation $\mathcal{L}(\mathrm{X}; \mathrm{Z})$ to refer to the set of linear bounded operators, defined from $\mathrm{X}$ to $\mathrm{Z}$. Additionally, we define $\mathcal{L}(\mathrm{X})$ to be the same as $\mathcal{L}(\mathrm{X}; \mathrm{X})$. Furthermore, we use the standard Sobolev space of order $r$ on $D$, which we denote as $\mathrm{H}^r(D)$. In this manuscript, we use the notation $'\lesssim'$ to denote $'\le'$ with its right-hand side multiplied by a generic positive constant.


    \subsection{Statement of the Results }     

    \subsubsection{The Discrete Effective Models: Theorem \ref{th1}}   

We start by introducing the free-space fundamental solution, denoted \(\Phi^{(m)}(\mathbf{x}, t; \mathbf{y}, \tau)\), for the heat operator. This solution satisfies the distributional equation
\(
(\kappa_m \partial_t - \Delta) \Phi^{(m)}(\mathbf{x}, t) = \delta_0(x,t)\; \text{in} \; \mathbb{R}^3,
\)
where \(\kappa_m = \frac{c_m}{\gamma_m}\) represents the diffusion constant in a homogeneous medium. In three-dimensional space, the fundamental solution is expressed as
    \begin{align}\label{green}
           \Phi^{(m)}(\mathbf{x}, t; \mathbf{y}, \tau) = \left( \frac{\kappa_m}{4\pi(t - \tau)} \right)^{\frac{3}{2}} \exp\left( -\frac{\kappa_m |\mathbf{x} - \mathbf{y}|^2}{4(t - \tau)} \right) \chi_{(0, \infty)}(t - \tau),
    \end{align}
where \(\chi_{(0, \infty)}(\cdot)\) is the indicator function of \((0, \infty)\).
\\
We also express the operators \(\mathbb{M}_{D}^{(\kappa)}\) and \(\mathbb{N}_{D}^{(\kappa)}\) using the dyadic Green's function \(\bm{\Upsilon}^{(k)}(x,y)\), given by
    \begin{align}\nonumber
        \bm{\Upsilon}^{(k)}(x,y) = \underset{x}{\text{Hess}}\,\mathcal{G}^{(k)}(x,y) + k^2 \mathcal{G}^{(k)}(x,y) \mathbb{I},
    \end{align}
where \(\mathcal{G}^{(k)}(x,y)\) is the Green's function of the Helmholtz operator. 
\\
In this work, we focus on the sub-wavelength regime, specifically where \( k\delta \ll 1 \), relative to the size of \( B \). To explore this, we model the electric permittivity, \( \varepsilon_p \), using the Drude model. We select the incident frequency \( k \) and damping parameter \( \zeta \) as follows:
\begin{align}\label{frequency}
   k - k_{n_0} \sim \delta^h\; \text{and}\; \zeta - \zeta_{n_0}\sim \delta^h,
\end{align}
with
\begin{align}\nonumber
    k_{n_0} = \Bigg[\frac{k_p^2\varepsilon_\infty \lambda_{n_0}^{(3)}\big[1-\lambda_{n_0}^{(3)}\big(\Re(\varepsilon_m) - \varepsilon_\infty \varepsilon_0\big)\big]}{\big|1 - \lambda_{n_0}^{(3)}(\varepsilon_m - \varepsilon_\infty \varepsilon_0)\big|^2}\Bigg]^\frac{1}{2}\; \text{and}\; \zeta_{n_0}= \frac{\Im(\varepsilon_m)\lambda_{n_0}^{(3)}}{1 - \lambda_{n_0}^{(3)}(\varepsilon_m - \varepsilon_\infty \varepsilon_0)} k_{n_0}.
\end{align}
Under this chosen frequency, the nanoparticle exhibits plasmonic behavior, resonating at the plasmonic frequency. With this choice of the incident frequency, we observe the following property: 
\begin{align}\label{condition3D}
\big|1 + \eta \lambda^{(3)}_{n}\big| \sim 
\begin{cases}
\delta^h & \text{for } n = n_0, \\
1 & \text{for } n \neq n_0,      
\end{cases}   
\end{align}
where \(\eta := \varepsilon_p - \varepsilon_m\) represents the electromagnetic contrast parameter, assumed to be of order 1.
\bigskip     

\noindent We now proceed to state the first main result of this work.

\begin{theorem}\label{th1}       

Consider the heat transfer problem (\ref{heat}) associated with a cluster of plasmonic nanoparticles, denoted \( D_i = \delta B_i +z_i (\delta\ll 1) \), for \( i = 1, 2, \ldots, M \), each belonging to the class $\mathcal{C}^2$. Let the source term, \( J(x, t) := \frac{k \cdot \Im(\varepsilon)}{2 \pi} |E|^2 f(t)\)\footnote{Given that $|E|$ is defined in the Maxwell model (\ref{Maxwell-model}), we know that $E \in L^4_{\textit{loc}}(\mathbb{R}^3)$ (as shown in \cite{mourad}). Since $\Im(\varepsilon)$ is compactly supported, then it suffices to consider $f \in  H^{r-\frac{1}{2}}_{0, \sigma}(0, T)$ so that $J(x,t)$ will be in the function space $H^{r - \frac{1}{2}}_{0, \sigma} \Big(0, T; \mathbb{L}^2(\mathbb{R}^3)\Big).$}  with $f \in  H^{r-\frac{1}{2}}_{0, \sigma}(0, T)$ and \( \Im(\varepsilon) = 0 \) in \( \mathbb{R}^3 \setminus \overline{D} \). In particular, we assume that $f:\mathbb{R}\to \mathbb{R}$ is a causal function and of class $\mathcal{C}^6(\mathbb{R})$. Then, there exists a unique solution \(u(x, t)\) in \( H^r_{0, \sigma} \Big(0, T; \mathbb{H}^1(\mathbb{R}^3)\Big) \) that satisfies equation (\ref{heat}). Furthermore, considering the asymptotic regimes (\ref{scales})-(\ref{scales2})-(\ref{regimes}) for the corresponding heat and electromagnetic parameters, and choosing the incidence frequency $k$ and damping parameter $\zeta$ as defined in (\ref{frequency}) with $\frac{9}{5}<h<2$, under the conditions (\ref{condition3D}) and
\begin{align}\label{condition-heat-1}
    b \ \max_{1 \leq i \leq M} \sum_{j \neq i} d_{ij}^{-2} < 1,
\end{align}
where \( b = \max_{1 \leq j \leq M} b_j \) with \( b_j := \frac{\overline{\alpha}_j}{\kappa_m} \text{vol}(B_j) \delta^{3-\beta} \), the heat generated by the nanoparticle cluster admits the following asymptotic expansion
\begin{align}\label{assymptotic-expansion-us}
    u^\textit{sc}(x, t) = -\sum_{i=1}^M \frac{\alpha_i}{\kappa_m} \text{Vol}(D_i) \int_0^t \Phi^{(m)}(x, t; z_i, \tau)\ \sigma^{(i)}(\tau) \, d\tau + \mathcal{O}(M \delta^{4 - h}) \quad \text{as} \quad \delta \to 0,
\end{align}
for \( (x, t) \in \mathbb{R}^3 \setminus K \times (0, T) \), with \( \overline{D} \subset \subset K \). Here, \( \sigma^{(i)}\) for \( i = 1, 2, \ldots, M \), uniquely satisfies the following linear algebraic system
\begin{align}\label{intro-1}
    \sigma^{(i)} + \sum_{\substack{j = 1 \\ j \neq i}}^M \frac{\alpha_j}{\kappa_m} \text{vol}(D_j) \int_0^t \Phi^{(m)}(z_i, t; z_j, \tau) \frac{\partial}{\partial \tau} \sigma^{(j)}(\tau) \, d\tau = \frac{\gamma_m}{\overline{\gamma}_{p_i}} \frac{k \cdot \overline{\varepsilon}_p}{2 \pi \kappa_m} \, \delta^{\beta - h} \, f(t) \, \mathbcal{P}_{B}\cdot\mathbcal{Q}_i\cdot\overline{\mathbcal{Q}^\textit{Tr}_i},
\end{align}
where \( {\mathbcal{Q}}_i \) is defined as the solution to the linear algebraic system associated with the electromagnetic scattering problem (\ref{Maxwell-model})
\begin{align}\label{in-1}
    \mathbcal{Q}_i - \eta \sum_{j \neq i}^M \bm{\Upsilon}^{(k)}(z_i, z_j)\cdot \mathbcal{P}_{D_j} \cdot \mathbcal{Q}_j = \mathrm{E}^\textit{in}(z_i),
\end{align}
and \( \mathbcal{P}_{D_i} \) is the polarization matrix defined by
\begin{align}\label{in-2}
    \mathbcal{P}_{D_i} =  \frac{\delta^3}{1 + \eta \lambda_{n_0}^{(3)}}\sum_{m=1}^{m_{n_0}} \int_{B_i} e^{(3)}_{m,n_0}(B) \otimes \int_{B_i} e^{(3)}_{m,n_0}(B) + \mathcal{O}(\delta^3),
\end{align}
with \( \Tilde{e}^{(3)}_{m,n_0} \) representing the scalled eigenfunctions associated with the space \( \nabla \mathbb{H}_{\textit{arm}} \) on the domain \( B_i \) and \(\displaystyle e^{(3)}_{m,n_0}(B) = \frac{1}{\Vert \Tilde{e}_{m,n_0}^{(3)} \Vert_{L^2(B_i)}} \Tilde{e}_{m,n_0}^{(3)} \) representing the corresponding normalized functions. In (\ref{in-2}), the dominant term arises due to the properties (\ref{frequency})–(\ref{condition3D}) and using the same properties, we can deduce that $\mathbcal{P}_{D_i} = \delta^{3-h}\mathbcal{P}_B$, where, we have $\mathbcal{P}_B$ as
\begin{align}
    \mathbcal{P}_B := C_B \sum_{m=1}^{m_{n_0}} \int_{B_i} e^{(3)}_{m,n_0}(B) \otimes \int_{B_i} e^{(3)}_{m,n_0}(B),\; \text{with}\ C_B = \Bigg(\lambda_{n_0}^{(3)}\cdot \dfrac{k_p^2k_{n_0}\sqrt{5}}{k_{n_0}^4 + (\zeta_{n_0}k_{n_0})^2}\Bigg)^{-1}.
\end{align}
\\
Additionally, \( \alpha_i := \gamma_{p_i} - \gamma_m \) denotes the contrast between the inner and outer heat conductivity coefficients, while \( \kappa_m := \frac{c_m}{\gamma_m} \) represents the diffusion constant of the surrounding homogeneous medium. The system of equations in (\ref{intro-1}) is invertible in \( H^1(0, T) \) under the condition (\ref{condition-heat-1}), while the algebraic system (\ref{in-1}) is invertible under the condition
\begin{align}\label{in-3}
    |\eta| \, \max_{i = 1, 2, \ldots, M} \Vert \mathbcal{P}_{D_i} \Vert\ d^{-3} < 1.
\end{align}

\end{theorem}    


    \subsubsection{The Continuous Effective Models: Theorem \ref{non-periodic}} 

\noindent
\textbf{\textit{Assumption 1.}}\label{as11}
Let be given a bounded and smooth domain \( \mathbf{\Omega} \). We partition \( \mathbf{\Omega} \) into approximately \( [d^{-3}] \) subdomains, denoted by \( \Omega_i \), for \( i = 1, 2, \ldots , [d^{-3}] \) (where \([x]\) denotes the integer \( n \) satisfying \( n \leq x < n + 1 \)). These subdomains are distributed periodically and disjointly across \( \mathbf{\Omega} \), with each subdomain \( \Omega_i \) containing exactly one inclusion \( \mathrm{D}_i \) with $\textit{Vol}(\Omega_i) = d^3$.  Additionally, We assume that the shape, and the heat and electromagnetic properties, of the each nanoparticles $D_j$ for $j=1,2,\ldots,M$ is same, which implies that $b_i = b_j$ for $i,j=1,2,\ldots, M.$ Let us define $\overline{b}:= \overline{b}_j$ be the scaled value of $b_j$ i.e. $\overline{b}=\frac{\overline{\alpha}_j}{\kappa_m} \text{vol}(B_j)$ for each $j=1,2,\ldots,M.$ Finally, we take
\begin{equation}\label{scales-h-beta}
    h=\beta\; \text{and}\; d\sim \delta^{1-\frac{\beta}{3}}.
\end{equation}


\begin{figure}[H]
\begin{center}
    \begin{tikzpicture}[scale=0.8]

\draw (0,0) ellipse (4cm and 2cm);
\node[scale=1.5] at (0,-2.5) {$\mathbf{\Omega}$};

\draw[<-,red, thick] (0.1, 0.15) -- (2, 2.5);
\node[above, scale=1] at (2,2.5) {$\mathrm{D}_j$};
\draw[<-,red, thick] (-0.3, 0.25) -- (-0.3, 2.5);
\node[above, scale=1] at (-0.3,2.5) {$\Omega_j$};
\draw[<-,red, thick] (-1.1, 0.17) -- (-2.5, 2.5);
\node[above, scale=1] at (-2.5,2.5) {$\mathrm{D}_i$};
\draw[<-,red, thick] (-0.7, 0.25) -- (-0.7, 2.5);
\node[above, scale=1] at (-0.7,2.5) {$\Omega_i$};

\node (formula) []{};
\draw[-latex,red] ($(formula.north west)+(4.5,0)$) arc
    [
        start angle=160,
        end angle=20,
        x radius=0.5cm,
        y radius =0.5cm
    ] ;

\clip (0,0) ellipse (4cm and 2cm);

\def\xellip(#1){4*sqrt(1 - (#1/2)^2)}

\foreach \y in {-1.75,-1.25,...,1.75} {
    \draw ({\xellip(\y)},\y) -- ({-\xellip(\y)},\y);
}

\coordinate (A) at (-1,0.95); 
\coordinate (B) at (-2,0.95); 
\coordinate (C) at (0,0.95); 
\coordinate (D) at (1,0.95); 
\coordinate (E) at (2,0.95); 
\shade[ball color=white] (A) circle (0.2);
\shade[ball color=white] (B) circle (0.2);
\shade[ball color=white] (C) circle (0.2);
\shade[ball color=white] (D) circle (0.2);
\shade[ball color=white] (E) circle (0.2);
\coordinate (a) at (-1,-0.05); 
\coordinate (b) at (-2,-0.05); 
\coordinate (c) at (0,-0.05); 
\coordinate (d) at (1,-0.05); 
\coordinate (e) at (2,-0.05); 
\coordinate (f) at (3,-0.05); 
\coordinate (g) at (-3,-0.05); 
\shade[ball color=white] (g) circle (0.2);
\shade[ball color=white] (f) circle (0.2);
\shade[ball color=white] (a) circle (0.2);
\shade[ball color=white] (b) circle (0.2);
\shade[ball color=white] (c) circle (0.2);
\shade[ball color=white] (d) circle (0.2);
\shade[ball color=white] (e) circle (0.2);
\coordinate (m) at (-1,0.45); 
\coordinate (n) at (-2,0.45); 
\coordinate (o) at (0,0.45); 
\coordinate (p) at (1,0.45); 
\coordinate (q) at (2,0.45); 
\coordinate (r) at (3,0.45); 
\coordinate (s) at (-3,0.45); 
\shade[ball color=white] (m) circle (0.2);
\shade[ball color=white] (n) circle (0.2);
\shade[ball color=white] (o) circle (0.2);
\shade[ball color=white] (p) circle (0.2);
\shade[ball color=white] (q) circle (0.2);
\shade[ball color=white] (r) circle (0.2);
\shade[ball color=white] (s) circle (0.2);
\coordinate (M) at (-1,-0.55); 
\coordinate (N) at (-2,-0.55); 
\coordinate (O) at (0,-0.55); 
\coordinate (P) at (1,-0.55); 
\coordinate (Q) at (2,-0.55); 
\coordinate (R) at (3,-0.55); 
\coordinate (S) at (-3,-0.55); 
\shade[ball color=white] (M) circle (0.2);
\shade[ball color=white] (N) circle (0.2);
\shade[ball color=white] (O) circle (0.2);
\shade[ball color=white] (P) circle (0.2);
\shade[ball color=white] (Q) circle (0.2);
\shade[ball color=white] (R) circle (0.2);
\shade[ball color=white] (S) circle (0.2);
\coordinate (T) at (-1,-1.05); 
\coordinate (U) at (-2,-1.05); 
\coordinate (V) at (0,-1.05); 
\coordinate (W) at (1,-1.05); 
\coordinate (X) at (2,-1.05); 
\shade[ball color=white] (T) circle (0.2);
\shade[ball color=white] (U) circle (0.2);
\shade[ball color=white] (V) circle (0.2);
\shade[ball color=white] (W) circle (0.2);
\shade[ball color=white] (X) circle (0.2);
\coordinate (t) at (-1,-1.55); 
\coordinate (v) at (0,-1.55); 
\coordinate (w) at (1,-1.55); 
\shade[ball color=white] (t) circle (0.2);
\shade[ball color=white] (v) circle (0.2);
\shade[ball color=white] (w) circle (0.2);
\coordinate (1) at (-1,1.45); 
\coordinate (2) at (0,1.45); 
\coordinate (3) at (1,1.45); 
\shade[ball color=white] (1) circle (0.2);
\shade[ball color=white] (2) circle (0.2);
\shade[ball color=white] (3) circle (0.2);

\node[above, scale=0.3] at (-1,1.4)  {$\bullet$};
\node[above, scale=0.3] at (0,1.4)  {$\bullet$};
\node[above, scale=0.3] at (1,1.4)  {$\bullet$};
\node[above, scale=0.3] at (-1,0.9)  {$\bullet$};
\node[above, scale=0.3] at (0,0.9)  {$\bullet$};
\node[above, scale=0.3] at (1,0.9)  {$\bullet$};
\node[above, scale=0.3] at (2,0.9)  {$\bullet$};
\node[above, scale=0.3] at (-2,0.9)  {$\bullet$};

\node[above, scale=0.3] at (-2,0.4)  {$\bullet$};
\node[above, scale=0.3] at (0,.4)  {$\bullet$};
\node[above, scale=0.3] at (1,.4)  {$\bullet$};
\node[above, scale=0.3] at (-1,0.4)  {$\bullet$};
\node[above, scale=0.3] at (3,0.4)  {$\bullet$};
\node[above, scale=0.3] at (1,0.4)  {$\bullet$};
\node[above, scale=0.3] at (2,.4)  {$\bullet$};
\node[above, scale=0.3] at (-3,0.4)  {$\bullet$};

\node[above, scale=0.3] at (-2,-0.1)  {$\bullet$};
\node[above, scale=0.3] at (0,-0.1)  {$\bullet$};
\node[above, scale=0.3] at (1,-0.1)  {$\bullet$};
\node[above, scale=0.3] at (-1,-0.1)  {$\bullet$};
\node[above, scale=0.3] at (3,-0.1)  {$\bullet$};
\node[above, scale=0.3] at (1,-0.1)  {$\bullet$};
\node[above, scale=0.3] at (2,-0.1)  {$\bullet$};
\node[above, scale=0.3] at (-3,-0.1)  {$\bullet$};

\node[above, scale=0.3] at (-2,-0.6)  {$\bullet$};
\node[above, scale=0.3] at (0,-0.6)  {$\bullet$};
\node[above, scale=0.3] at (1,-0.6)  {$\bullet$};
\node[above, scale=0.3] at (-1,-0.6)  {$\bullet$};
\node[above, scale=0.3] at (3,-0.6)  {$\bullet$};
\node[above, scale=0.3] at (1,-0.6)  {$\bullet$};
\node[above, scale=0.3] at (2,-0.6)  {$\bullet$};
\node[above, scale=0.3] at (-3,-0.6)  {$\bullet$};

\node[above, scale=0.3] at (-2,-1.1)  {$\bullet$};
\node[above, scale=0.3] at (0,-1.1)  {$\bullet$};
\node[above, scale=0.3] at (1,-1.1)  {$\bullet$};
\node[above, scale=0.3] at (-1,-1.1)  {$\bullet$};
\node[above, scale=0.3] at (2,-1.1)  {$\bullet$};
\node[above, scale=0.3] at (1,-1.6)  {$\bullet$};
\node[above, scale=0.3] at (0,-1.6)  {$\bullet$};
\node[above, scale=0.3] at (-1,-1.6)  {$\bullet$};

\foreach \x in {-3.5,-2.5,...,3.5} {
    \draw (\x,-2) -- (\x,2);
}
\foreach \y in {-1.75,-1.25,...,1.75} {
    \draw (-3.5,\y) -- (3.5,\y);
}

\end{tikzpicture}
\begin{tikzpicture}[scale=2] 

\draw[line width=1.5pt] (0,0) -- (1,0) -- (1,1) -- (0,1) -- cycle; 
\draw[line width=1.5pt] (0,1) -- (1,1) -- (1.5,1.5) -- (0.5,1.5) -- cycle; 
\draw[line width=1.5pt] (1,0) -- (1,1) -- (1.5,1.5); 
\draw[line width=1.5pt] (0,0)--(0,1)--(0.5,1.5); 

\draw[dotted, blue, line width=1.2pt] (0.5,1.5) -- (0.5,0.5)--(0,0);

\draw[dotted, blue, line width=1.2pt] (1.5,1.5) -- (1.5,0.5)--(1,0);

\draw[line width=1.5pt] (1,0) -- (2,0) -- (2,1) -- (1,1) -- cycle; 
\draw[line width=1.5pt] (1,1) -- (2,1) -- (2.5,1.5) -- (1.5,1.5) -- cycle; 
\draw[line width=1.5pt] (2,0) -- (2.5,0.5) -- (2.5,1.5) -- (2,1) -- cycle; 

\draw[<->](0,-0.3)--(1,-0.3) node[pos=0.5, sloped, above] {$\delta$};
\draw[dotted, blue, line width=1.2pt] (0.5,0.5) -- (2.5,0.5);

\coordinate (A) at (0.75,0.75); 
\coordinate (B) at (1.75,0.75); 
\shade[ball color=white] (A) circle (0.2);
\shade[ball color=white] (B) circle (0.2);

\node[above, scale=1] at (0.75,0.65) {$z_i$};
\node[above, scale=0.3] at (0.7,0.7) {$\bullet$};

\node[above, scale=1] at (1.75,0.63) {$z_j$};
\node[above, scale=0.3] at (1.7,0.7) {$\bullet$};

\node[below,scale=1] at (0.5,0.25) {$\Omega_i$};
\node[below, scale=1] at (1.5,0.25) {$\Omega_j$};

\node[below, scale=1] at (0.75,0.55) {$\mathrm{D}_i$};
\node[below, scale=1] at (1.75,0.55) {$\mathrm{D}_j$};

\coordinate (P1) at (0.95,0.75); 
\coordinate (P2) at (1.55,0.75); 

\draw[red, line width=0.2pt] (P1) -- (P2) node[midway, above, scale = 1] {$d$};
\end{tikzpicture}
\end{center}
\caption{A schematic illustration for the global distribution of the particles in $\mathbf{\Omega}$ and local relation between any two particles.}
\end{figure}


\noindent 
We then state the second main result of this work, related to the continuous effective medium theory.

\begin{theorem}\label{non-periodic}    

Let \( \bm{\Omega} \) be a bounded, \( \mathcal{C}^2 \)-regular domain in \( \mathbb{R}^3 \) with unit volume, and assume that the plasmonic nanoparticles are distributed inside $\bm{\Omega}$ according to the \textit{Assumption} \textcolor{blue}{1}. Then, under the assumptions of Theorem \ref{th1}, for \((x, t) \in \mathbb{R}^3 \setminus \overline{\bm{\Omega}} \times (0, T)\), we have the following asymptotic expansion
\begin{align}
    u(x, t) - \mathbcal{W}(x, t) = \mathcal{O}(\delta^{\frac{2}{7}(3-\beta)}, \quad \text{as} \quad \delta \to 0,
\end{align}
where \( \mathbcal{W}(x, t) \) satisfies the following parabolic model:
\begin{align}
    \begin{cases}
        ((\kappa_m + \overline{b} \, \rchi_{\bm{\Omega}})\partial_t - \Delta) \mathbcal{W}  = \overline{b} \, \rchi_{\bm{\Omega}} \, \mathbcal{F}(x, t) & \text{if}\; (x, t) \in \mathbb{R}^3 \times (0, T), \\
        \mathbcal{W}(x, 0) = 0 & \text{for}\; x \in \mathbb{R}^3, \\
        |\mathbcal{W}(x, t)| \leq C_0 \, e^{A |x|^2} & \text{as} \; |x| \to +\infty,
    \end{cases}\label{parabolic-model}
\end{align}
for some positive constant \( C_0 \) and \( A \le \frac{1}{4T} \). Here, we define \( \displaystyle\mathbcal{F}(x, t) := \overline{a} \, f(t) \, \mathbcal{A}_B\cdot\mathbcal{P}_{B}^{-1}\cdot \mathbcal{A}_B \cdot \bm{E}_f\cdot \overline{\bm{E}}^\textit{Tr}_f \), where \( \overline{a} := \frac{\gamma_m}{\overline{\gamma}_{p_i}} \frac{k \cdot \overline{\varepsilon}_{p_i}}{2 \pi \kappa_m} \) and $\bm{\textit{E}}_f$ satisfies the following effective model:
 \begin{align}\label{effective-maxwell-model}
        \begin{cases}
            \textit{curl} \, \textit{curl} \; \bm{\textit{E}}_f - k^2 \varepsilon_\textit{ef} \bm{\textit{E}}_f = 0\; \text{in }\ \mathbb{R}^3, \\
            \bm{\textit{E}}_f^{\textit{sc}} := \bm{\textit{E}}_f - E^{\textit{in}} \text{ satisfies the Silver-Müller radiation condition (S-M.R.C.)} &\\
            \text{with}\; \varepsilon_\textit{ef} := \varepsilon_m + \mathbcal{A}_B\rchi(\bm{\Omega}),
        \end{cases}                        
    \end{align}
where $\mathbcal{A}_B$ is the effective polarization matrix given by $\displaystyle\mathbcal{A}_B :=\Big(\mathbb{I} - \nabla\int_{B}\nabla\mathbcal{G}^{(0)}(0,y) \cdot \mathbcal{P}_B dy\Big)^{-1}\mathbcal{P}_B.$  

\end{theorem}    

    \subsection{Discussion about the Results} 

According to the results in Theorem \ref{th1} and Theorem \ref{non-periodic}, we reduce the problem of heat generation using plasmonic nanoparticles to the following two sub problems, namely a control problem for the parabolic model and an internal phaseless inverse problem for the Maxwell model.  
\begin{enumerate}
      \item Solve the control problems of estimating the sources needed to achieve a needed amount of heat in $\bm{\Omega}$. This question is related to the discrete and continuous heat models (\ref{intro-1}) and (\ref{parabolic-model}) respectively. For the discrete problem, it amounts to estimate the needed source term in (\ref{in-1})  such that the solution vector $(\sigma_i)^M_{i=1}$ plugged in (\ref{assymptotic-expansion-us}) generates $u^{sc}$ that matches with the needed heat. For the continuous problem, it amounts to estimate the needed source term in (\ref{parabolic-model}) so that the solution of that problem matches with the needed heat.

     \item As a second step, we proceed to generate these sources, by solving inverse problems of estimating the permittivity sequence $\eta \mathbcal{P}_{D_i}, i=1, ..., M$ or the permittivity function $\epsilon_{ef}$ from the internal but phaseless data $(\mathbcal{P}_{B}\cdot\mathbcal{Q}_i\cdot\overline{\mathbcal{Q}^\textit{Tr}_i})^M_{i=1}$ or  $\mathbcal{A}_B\cdot\mathbcal{P}_{B}^{-1}\cdot \mathbcal{A}_B \cdot \bm{E}_f\cdot \overline{\bm{E}}^\textit{Tr}_f$, respectively. As discussed in \cite{cao-sini}, for a unit ball $B$, the polarization matrices $\mathbcal{P}_B$ and $\mathbcal{A}_B$ are proportional to the identity matrix. Therefore, for this case, the source function in (\ref{intro-1}) is proportional to $f(\cdot) \vert\mathbcal{Q}_i\vert^2$ while the source function $\mathbcal{F}(\cdot, \cdot)$ in (\ref{parabolic-model}) is proportional to $f(\cdot) \vert \bm{\textit{E}}_f(\cdot)\vert^2$.  These internal inverse problems are related to the models (\ref{in-2}) and (\ref{effective-maxwell-model}) respectively.
\end{enumerate}
These control and inverse problems, for both the discrete and continuous problems described above, will be analyzed and computationally addressed in future works.
\bigskip

\noindent  
We add the following two comments about the distribution of the nanoparticles in $\bm{\Omega}$.
\begin{enumerate}
    \item The periodicity in the distribution of the nanoparticle clusters within \( \mathbf{\Omega} \) is not a necessary condition. This non-periodic approach provides greater flexibility in modeling, especially given the random or irregular placement often needed in practical scenarios. For example, the Maxwell effective model describing the electromagnetic waves generated by nanoparticles distributed on latices inspired by van-der-Waals heterostructures, which are globally periodic but locally arbitrary distributed, was derived in \cite{cao-ghandriche-sini-VDW}.       
    A similar departure from periodicity, with globally arbitrary distributions, has been addressed in \cite{AM4} for a different problem setting, related to acoustic propagation in the presence of bubbles, where we have characterized the effective medium without requiring a strictly periodic arrangement. This framework allows for a more realistic distributions accommodating random or clustered arrangements of the nanoparticles. 

    \item The condition (\ref{scales-h-beta}) can be relaxed as well. It allows ensuring non-trivial effective coefficients for both the heat and electromagnetic models. It is assumed to simplify the exposition. However, characterizing the model for different choices of $h$ and $\beta$ might be of great interest. In particular, we can generate a moderate heat model but a highly contrasting electric permittivity allowing for a generation of giant electromagnetic fields (see \cite{cao-sini} for an attempt in this direction using dielectric nanoparticles). Such a scenario can be useful, in particular, for inverse problems using heat potentials using nanoparticles as contrast agents. Namely, the generation of local giant electromagnetic field can help in linearizing the boundary (electromagnetic impedance or heat potential) maps. Such an approach has been already tested in \cite{A-M} for the Calderon problem in acoustics. 
\end{enumerate}
\noindent
The analysis needed to prove Theorem \ref{th1} and Theorem \ref{non-periodic} is based on the layer potential and volume potential operators for both the parabolic and Maxwell models. For the parabolic model, we start with the Fourier-Laplace transform to derive a first a-priori estimate of the heat potentials in the weighted-in-time space $H^r_{0,\sigma}\Big(0,T; \mathbb H^1(\mathbb R^3)\Big)$. With such a transformation, we reduce the well-posedness to the Fourier-Laplace domain where the $s-$parameter lives in the upper-half space, as $\sigma>0$. The advantage of these spaces is that we avoid the spectrum of the Laplace operator and hence derive the coercivity estimate uniform with respect to $s$. The price to pay is that the derived estimate is not enough to our purpose of deriving the full asymptotic expansions for the heat-generation model. This is the only place where the Fourier-Laplace transform has been used. The rest of the proofs are done in the time-domain. To improve this first a-priori estimate, we use the Helmholtz decomposition and the spectral properties of the Newtonian-like and Magnetization-like operators appearing in the related heat-Lippmann-Schwinger equation. In parallel, while dealing with the Maxwell model, we use the Maxwell-Lippmann-Schwinger system fo equations given by the Newtonian as well as the Magnetization operators. We derive the needed a-priori estimates for the electric field, particularly in the $\mathbb L^4$-space, using the spectral decomposition of these two operators. Then injecting them into this Lippmann-Schwinger system of equations we derive the point-approximation expansions for both the heat-Maxwell system modeling the heat generation. The challenge here is twofold. First, we need to take into account the contrasting scales of both the heat conduction $\gamma_p$ and the permittivity distribution $\epsilon_p$ to sort out . Second, we need to take into account the cluster character of the plasmonic nanoparticle to sort out the Voltera-type system and algebraic system satisfied by the projected, heat and electric, fields on the spectral-subspaces.  Gathering all these steps, we derive the proof of Theorem \ref{th1}. The proof of Theorem \ref{non-periodic} is done by discretizing the Lippmann-Schwinger systems fo equations modeling (\ref{parabolic-model}) and (\ref{effective-maxwell-model}) respectively to derive the related Voltera-integral equations and Algebraic systems respectively. Then, we derive the natural effective heat-conduction coefficient and the effective permittivity by matching these last Voltera and algebraic systems with the ones derived in Theorem \ref{th1}.  The discretization steps are justified by deriving the needed regularity estimates for both the heat and Maxwell systems allowing for the control of the error terms and taking into account the contribution of the large cluster of the nanoparticles. 
\vspace{0.2cm}

\noindent 
Let us stress here that we do not use the standard homogenization techniques. Rather, we base all our analysis on the Foldy-Lax point-approximation framework. The advantage of this framework is that we can derive the related Foldy-Lax (Voltera-type or algebraic) systems for an arbitrary distributions of the nanoparticles. From this system, we can already fix the kind of continuous mathematical model we can end with, depending on how the nanoparticles are distributed in volumetric domains or lower-dimensional hyper-surfaces.


    \section{Proof of Theorem \ref{th1}}    

   \subsection{Constructing the Integral Equation for the Parabolic Problem} 

We consider the following initial value problem (IVP) with a source term \(\mathrm{J}(x,t)\), defined as \(\mathrm{J}(x,t) := \frac{k}{2\pi} \boldsymbol{\Im}(\varepsilon) |\mathrm{E}|^{2} f(t)\), as follows
    \begin{align}
        \begin{cases}
              \frac{c_m}{\gamma_m} \frac{\partial v}{\partial t} - \Delta v = \frac{1}{\gamma_m} \mathrm{J}(x,t) & \text{in} \ \mathbb{R}^3 \times (0,T), \\
              v(x,0) = 0 & \text{for} \ x \in \mathbb{R}^3.
        \end{cases}
    \end{align}
Next, we define \(w = u - v\), where \(w\) is the temperature difference between the medium before and after the injection of the nanoparticle. The function \(w\) satisfies the following initial value problem in a distributional sense
    \begin{align}
        \begin{cases}
               \frac{\partial w}{\partial t} - \Delta w = \nabla \cdot \alpha \nabla u - \alpha^\dagger \frac{\partial u}{\partial t} & \text{in} \ \mathbb{R}^3 \times (0,T), \\
               w(x,0) = 0,
        \end{cases}
    \end{align}
where \(\alpha^\dagger := c_p - c_m\) and \(\alpha := \gamma_p - \gamma_m\), representing the contrasts in heat capacity and thermal conductivity between the two regions.

Let $w$ be a distribution in $\mathbb{R}^3$, satisfying $(\kappa_m\partial_t-\Delta)w(x,t) = f,$ with $f$ having a compact support. Then, due to the fact $f$ is integrable, we deduce
\begin{align}
    w(x,t) = \frac{1}{\kappa_m}\int_0^t\int_D \Phi^{(m)}(\mathrm{x}, t; \mathrm{y}, \tau)\ f(y,\tau)dyd\tau.
\end{align}
Consequently, using the fundamental solution (\ref{green}) and applying Green's identity, we derive the following Lippmann-Schwinger equation for \(w\)
    \begin{align}\label{con-ls}
        w + \alpha^\dagger \frac{1}{\kappa_m} \int_0^t \int_D \Phi^{(m)}(\mathrm{x}, t; \mathrm{y}, \tau) \frac{\partial u}{\partial \tau} \, d\mathrm{y} \, d\tau - \alpha \frac{1}{\kappa_m} \nabla_\mathrm{x} \cdot \int_0^t \int_D \Phi^{(m)}(\mathrm{x}, t; \mathrm{y}, \tau) \nabla u \, d\mathrm{y} \, d\tau = 0.
    \end{align}
Next, we use Green's identity for the term involving \(\nabla \cdot \alpha \nabla u\) to further simplify the expression
    \begin{align}\label{con-ls2}
        \int_0^t \int_D \nabla_\mathrm{x} \Phi^{(m)}(\mathrm{x}, t; \mathrm{y}, \tau) \cdot \nabla u(\mathrm{y}, \tau) \, d\mathrm{y} \, d\tau 
        &\nonumber = -\int_0^t \int_{\partial D} \Phi^{(m)}(\mathrm{x}, t; \mathrm{y}, \tau) \partial_\nu u(\mathrm{y}, \tau) \, d\mathrm{y} d\tau 
        \\&+ \int_0^t \int_D \Phi^{(m)}(\mathrm{x}, t; \mathrm{y}, \tau) \Delta u \, d\mathrm{y} d\tau.
    \end{align}
Substituting the heat equation \(\Delta u = \frac{c_p}{\gamma_p} \frac{\partial u}{\partial t} - \frac{1}{\gamma_p} \mathrm{J}(x, t)\), the right hand side of the above expression becomes
    \begin{align}\label{con-ls1}
          -\int_0^t \int_{\partial D} \Phi^{(m)}(\mathrm{x}, t; \mathrm{y}, \tau) \partial_\nu u \, d\mathrm{y} d\tau + \frac{c_p}{\gamma_p} \int_0^t \int_D \Phi^{(m)}(\mathrm{x}, t; \mathrm{y}, \tau) \frac{\partial u}{\partial t} \, d\mathrm{y} d\tau - \frac{1}{\gamma_p} \int_0^t \int_D \Phi^{(m)}(\mathrm{x}, t; \mathrm{y}, \tau) \mathrm{J}(y,\tau) \, d\mathrm{y} d\tau.
    \end{align}
Consequently, combining (\ref{con-ls1}), (\ref{con-ls2}) and plugging those in (\ref{con-ls}), the Lippmann-Schwinger equation for \(u\) becomes
    \begin{align}
         u + \left( \alpha^\dagger - \alpha \frac{c_p}{\gamma_p} \right) \frac{1}{\kappa_m} \int_0^t \int_D \Phi^{(m)}(\mathrm{x}, t; \mathrm{y}, \tau) \frac{\partial u}{\partial \tau} \, d\mathrm{y} d\tau &+ \alpha \frac{1}{\kappa_m} \int_0^t \int_{\partial D} \Phi^{(m)}(\mathrm{x}, t; \mathrm{y}, \tau) \partial_\nu u \, d\mathrm{y} d\tau 
         \nonumber= v(x,t) 
         \\ &- \frac{\alpha}{\gamma_p \kappa_m} \int_0^t \int_D \Phi^{(m)}(\mathrm{x}, t; \mathrm{y}, \tau) \mathrm{J}(y,\tau) \, d\mathrm{y} d\tau.
    \end{align}
Finally, recalling that \(v(x,t)\) solves the following initial value problem
\[
\begin{cases}
\kappa_m \frac{\partial v}{\partial t} - \Delta v = \mathrm{J}(x,t) & \text{in} \ \mathbb{R}^3 \times (0,T), \\
v(x,0) = 0 & \text{for}\ x\in \mathbb R^3,
\end{cases}
\]
we express \(v(x,t)\) as
\[
v(x,t) = \frac{1}{\kappa_m} \int_0^t \int_D \Phi^{(m)}(\mathrm{x}, t; \mathrm{y}, \tau) \mathrm{J}(y,\tau) \, d\mathrm{y} d\tau.
\]
Thus, the Lippmann-Schwinger equation simplifies to
\begin{align}\label{lippmann}
u + \left( \alpha^\dagger - \alpha \frac{c_p}{\gamma_p} \right) \frac{1}{\kappa_m} \int_0^t \int_D \Phi^{(m)}(\mathrm{x}, t; \mathrm{y}, \tau) \frac{\partial u}{\partial \tau} \, d\mathrm{y} d\tau 
&\nonumber+ \alpha \frac{1}{\kappa_m} \int_0^t \int_{\partial D} \Phi^{(m)}(\mathrm{x}, t; \mathrm{y}, \tau) \partial_\nu u \, d\mathrm{y} d\tau 
\\ &= \frac{\gamma_m}{\gamma_p} \frac{1}{\kappa_m} \int_0^t \int_D \Phi^{(m)}(\mathrm{x}, t; \mathrm{y}, \tau) \mathrm{J}(\mathrm{y}, \tau) \, d\mathrm{y} d\tau.
\end{align}
Finally, we introduce the parameter 
\begin{equation}
    \vartheta = \alpha^\dagger - \alpha \frac{c_v}{\gamma_p} \sim 1,\; \text{implying}\; \alpha \sim \delta^{-\beta},
\end{equation}
where \(\alpha^\dagger\) and \(\alpha\) represent the contrasts in heat coefficients between the inner and outer media.


\subsection{Well-posedness and Regularity of the Parabolic Problem}\label{apriori}          

To establish the main result of this section, we begin by analyzing the problem (\ref{heat}) in the Fourier-Laplace domain. Utilizing fundamental technical results related to the Fourier-Laplace transform as discussed in Section \ref{lap-fou}, we subsequently return to the time domain to complete the proof.

\subsubsection{\texorpdfstring{A-Priori Estimates in $H^r_{0,\sigma}\Big(0,T; \mathbb H^1(\mathbb R^3)\Big)$}{A-priori Estimates}}
\begin{lemma}\label{lemma-apriori-heat}
Let us consider $J(x,t) = \frac{k\cdot \Im(\varepsilon)}{2\pi} \big|E\big|^2f(t)$ with $ f\in H^{r-\frac{1}{2}}_{0,\sigma}(0,T)$ and $\Im(\varepsilon) = 0 \in \mathbb R^3\setminus \overline{D}.$ Then, we have $u\in H^r_{0,\sigma}\Big(0,T; \mathbb H^1(\mathbb R^3)\Big)$ solving (\ref{heat}) which satisfies the following estimate for $k\le r-\frac{1}{2}$
\begin{align}\nonumber
    \Vert \partial_t^ku(\cdot,t)\Vert_{\mathbb L^2(D)} \lesssim \delta^{\frac{5}{2}-h},
\end{align}
and in the case of multiple inclusion, we have
\begin{align}\nonumber
    \Vert \partial_t^ku^{(i)}(\cdot,t)\Vert_{\mathbb L^2(D_i)} \lesssim \delta^{\frac{1}{2}+\beta-h}.
\end{align}
In addition to that, the solution admits the following bounds
\begin{align}\nonumber
    \Vert \partial_t^k\nabla u(\cdot,t)\Vert_{\mathbb L^2(D)} \lesssim \delta^{\frac{3}{2}+\beta-h}\ \text{and}\ \Vert \partial_t^k\partial_\nu u(\cdot,t)\Vert_{\mathbb L^2(\partial D)} \lesssim \delta^{1+\beta-h}.
\end{align}
\end{lemma}

\begin{proof}
To begin, we consider the following transmission Helmholtz-type equation, which results from applying the Fourier-Laplace transform to (\ref{heat})
\begin{align}\label{laplace}
    \begin{cases}
        \hat{u} \in H^1(\mathbb{R}^3), \quad \mathrm{c}_v \bm{\mathrm{s}} \; \Tilde{u}(\mathrm{x}, \bm{\mathrm{s}}) - \nabla \cdot \gamma \nabla \Tilde{u}(\mathrm{x}, \bm{\mathrm{s}}) = \Tilde{\mathrm{J}}(\mathrm{x}, \bm{\mathrm{s}}) \; \text{in} \ \mathbb{R}^3, \\
        \left[ u \right] = 0, \quad \left[ \partial_\nu u \right] = 0.
    \end{cases}
\end{align}
Next, we develop a variational approach for solving the above problem (\ref{laplace}) and apply the Lax-Milgram Lemma. The problem (\ref{laplace}) can be reformulated as the following variational form
\begin{align}
\begin{cases}
    \hat{u} \in H^1(\mathbb{R}^3), \\
    \gamma \big\langle \nabla \hat{u}, \nabla v \big\rangle_{L^2(\mathbb{R}^3)} + \bm{\mathrm{s}} \big\langle \hat{u}, v \big\rangle_{L^2(\mathbb{R}^3 )} = \big\langle \hat{\mathrm{J}}, v \big\rangle_{L^2(\bm{D})}, \quad \text{for all} \ v \in H^1(\mathbb{R}^3).
\end{cases}
\end{align}
To verify the coercivity of the bilinear form, we set $v = \bm{s}^{\frac{1}{2}} \Tilde{u}$. Given the relation (\ref{1.13}) and the condition $\Re(\bm{s}^{\frac{1}{2}}) > 0$, we conclude that the bilinear form is coercive, guaranteeing the existence of a unique solution. Specifically, we have
\begin{align}\label{bi}
    \Re \left( \overline{\bm{s}}^{\frac{1}{2}} \left( \| \mathrm{c}_v \nabla \Tilde{u}(\cdot, \bm{\mathrm{s}}) \|^2_{L^2(\mathbb{R}^3 )} + \bm{s} \| \gamma \Tilde{u}(\cdot, \bm{\mathrm{s}}) \|^2_{L^2(\mathbb{R}^3)} \right) \right) = \Re(\bm{s}^{\frac{1}{2}}) \left( \| \mathrm{c}_v \nabla \Tilde{u}(\cdot, \bm{\mathrm{s}}) \|^2_{L^2(\mathbb{R}^3)} + |\bm{s}| \| \gamma \Tilde{u}(\cdot, \bm{\mathrm{s}}) \|^2_{L^2(\mathbb{R}^3 )} \right) \geq 0. \nonumber
\end{align}
Finally, noting that $\omega = \Re(\bm{s}^{\frac{1}{2}})$, and considering the support of $\Tilde{\mathrm{J}}$ in $D$, along with the application of the Cauchy-Schwartz inequality, we find that $\Tilde{u}(\cdot, \bm{\mathrm{s}})$ satisfies the inequality
\begin{align}
\omega |\bm{s}| \| \mathrm{c}_v \Tilde{u}(\cdot, \bm{\mathrm{s}}) \|^2_{L^2(\mathbb{R}^3 )} + \omega \left\| \gamma \nabla \Tilde{u}(\cdot, \bm{\mathrm{s}}) \right\|^2_{L^2(\mathbb{R}^3 )} \leq |\bm{s}|^{\frac{1}{2}} \| \Tilde{\mathrm{J}}(\cdot, \bm{\mathrm{s}}) \|_{L^2(D)} \| \Tilde{u}(\cdot, \bm{\mathrm{s}}) \|_{L^2(\mathbb{R}^3 )}.
\end{align}
Given the scaling properties of \( \mathrm{c}_v \) and \( \gamma \) from (\ref{scales}), we further derive that
\begin{align}\label{ineq}
    \| \Tilde{u}(\cdot,\bm{\mathrm{s}}) \|_{\mathbb{L}^2(D)} \lesssim \frac{1}{\omega|\bm{s}|^{\frac{1}{2}}} \| \Tilde{J}(\cdot,\bm{\mathrm{s}}) \|_{\mathrm{L}^2(D)}\quad \text{and} \quad \| \nabla \Tilde{u}(\cdot,\bm{\mathrm{s}}) \|_{\mathbb{L}^2(D)} \lesssim \delta^\beta\frac{|\bm{s}|^{\frac{1}{2}}}{\omega} \| \Tilde{J}(\cdot,\bm{\mathrm{s}}) \|_{\mathrm{L}^2(D)}.
\end{align}
Next, we define the inverse Laplace transform of \( \hat{u}(\mathrm{x}, \cdot) \) for \( \Re(\bm{\mathrm{s}}) = \sigma > 0 \) as
\begin{align}\label{invlap}
    u(\mathrm{x}, \mathrm{t}) := \frac{1}{2\pi i} \int_{\sigma - i\infty}^{\sigma + i\infty} e^{\bm{\mathrm{s}} \mathrm{t}} \hat{u}(\mathrm{x}, \bm{\mathrm{s}})\, d\bm{\mathrm{s}} = \frac{1}{2\pi} \int_{-\infty}^{\infty} e^{(\sigma + i\omega) \mathrm{t}} \Tilde{u}(\mathrm{x}, \sigma + i\omega)\, d\omega.
\end{align}
Given the bounds on \( \bm{\mathrm{s}} \) in (\ref{ineq}), \( u(\mathrm{x}, \mathrm{t}) \) is well-defined. Moreover, using contour integration techniques, it can be shown that \( u(\mathrm{x}, \mathrm{t}) \) is independent of \( \sigma \), as detailed in \cite[pp. 39]{sayas}. Consequently, we derive the following bound
\begin{align}
    \| u \|^2_{H_{0,\sigma}^r(0,T; L^2(D))} 
    &\nonumber\lesssim \frac{1}{\omega^2} \int_{-\infty + i\sigma}^{+\infty + i\sigma} |\bm{s}|^{2r} \| \hat{u}(\bm{s}) \|_{L^2(D)}^2 \, d\bm{s}
    \\ \nonumber&\lesssim \frac{1}{\omega^2} \int_{-\infty + i\sigma}^{+\infty + i\sigma} |\bm{s}|^{2(r - \frac{1}{2})} \| \hat{u}(\bm{s}) \|_{L^2(D)}^2 \, d\bm{s}
    \\ &= \frac{1}{\omega^2} \| J \|^2_{H_{0,\sigma}^{r - \frac{1}{2}}(0,T; L^2(D))}.
\end{align}
Similarly, we obtain
\begin{align}
    \| \nabla u \|^2_{H_{0,\sigma}^r(0,T; L^2(D))} \lesssim \delta^4 \frac{1}{\omega^2} \| J \|^2_{H_{0,\sigma}^{r + \frac{1}{2}}(0,T; L^2(D))}.
\end{align}
We also note that \( \Im(\varepsilon) = 0 \) in \( \mathbb{R}^3 \setminus \overline{D} \), and for a plasmonic nanoparticle, \( \Im(\varepsilon_p) \sim \delta^h \) in \( D \), while \( \| \mathrm{E} \|^2_{\mathbb{L}^4(D)} \sim \delta^{\frac{3}{2} - 2h} \) for \( h \in (0,2) \). This leads to the conclusion
\begin{align}\label{h1}
    \| \Tilde{\mathrm{J}}(\cdot, \bm{\mathrm{s}}) \|_{\mathbb{L}^2(D)} 
    \nonumber&= \left( \int_D |\Tilde{\mathrm{J}}(x, \bm{\mathrm{s}})|^2 dx \right)^{\frac{1}{2}}
    \\ &\lesssim \Im(\varepsilon_p) \frac{k}{2\pi} \| \mathrm{E} \|^2_{\mathbb{L}^4(D)} \lesssim \delta^{\frac{3}{2} - h}.
\end{align}
Thus, we derive the estimates
\begin{align}
    \| u \|_{H_{0,\sigma}^r(0,T; \mathbb{L}^2(D))} \sim \delta^{\frac{3}{2} - h} \quad \text{and} \quad \| \nabla u \|_{H_{0,\sigma}^r(0,T; \mathbb{L}^2(D))} \sim \delta^{\frac{3}{2}+\beta - h}.
\end{align}
Consequently, for \( r > 1 \), we derive the following estimates
\[
\| u(\cdot, \mathrm{t}) \|_{\mathbb{L}^2(D)} \nonumber \lesssim \| u \|_{\mathrm{H}_{0,\sigma}^r(0,\mathrm{T}; \mathbb{L}^2(D))} \lesssim \delta^{\frac{3}{2}-h}, \quad \text{and} \quad
\|\nabla u(\cdot,\mathrm{t})\|_{\mathbb{L}^2(D)} \lesssim \| \nabla u \|_{\mathrm{H}_{0,\sigma}^r(0,\mathrm{T}; \mathbb{L}^2(D))} \lesssim \delta^{\frac{3}{2}+\beta-h}.
\]
Furthermore, we can infer the following estimates for $k\le p-\frac{1}{2}$:
\begin{equation}\label{esti22}
\| \partial_\mathrm{t}^\mathrm{k} u(\cdot, \mathrm{t}) \|_{\mathbb{L}^2(D)} \lesssim \delta^{\frac{3}{2}-h}, \quad \mathrm{t} \in [0,\mathrm{T}],
\end{equation}
and
\begin{equation}
 \| \partial_\mathrm{t}^\mathrm{k} \nabla u(\cdot, \mathrm{t}) \|_{\mathbb{L}^2(D)} \lesssim \delta^{\frac{3}{2}+\beta-h}, \quad \mathrm{t} \in [0,\mathrm{T}].   
\end{equation}
We now consider the elliptic problem described by equation \eqref{laplace} in the domain \( D \), given by
\begin{align}\label{esti222}
    \frac{\mathrm{c}_p}{\gamma_p} \partial_t u(x,t) - \Delta u(x,t) = \frac{1}{\gamma_p} J(x,t), \quad \text{in} \; D \times (0,T).
\end{align}
Utilizing the initial estimate \( \| \partial_t u(x,t) \|_{\mathbb{L}^2(D)} \lesssim \delta^{\frac{3}{2}-h} \) from equation \eqref{esti222}, and combining it with the estimate \( \| J(\cdot, t) \|_{\mathbb{L}^2(D)} \lesssim \delta^{\frac{3}{2}-h} \) and the scaling property \( \gamma_p \sim \delta^{-\beta} \), we can deduce that
\[
\| \Delta u(\cdot, t) \|_{\mathbb{L}^2(D)} \lesssim \delta^{\frac{3}{2}+\beta-h}.
\]
By applying scaling to the boundary \( \partial B \), using the trace theorem, and then scaling back to \( \partial D \), we obtain the following estimate
\begin{equation}
 \| \partial_\mathrm{t}^\mathrm{k} \partial_\nu u(\cdot, \mathrm{t}) \|_{\mathbb{L}^2(\partial D)} \lesssim \left( \delta^{\frac{1}{2}} \| \partial_\mathrm{t}^\mathrm{k} \Delta u(\cdot, \mathrm{t}) \|_{\mathbb{L}^2(D)} + \delta^{-\frac{1}{2}} \| \partial_\mathrm{t}^\mathrm{k} \nabla u(\cdot, \mathrm{t}) \|_{\mathbb{L}^2(D)} \right) \lesssim \delta^{1+\beta-h},\ \text{for a.e.}\  \mathrm{t} \in [0, \mathrm{T}],\ k\le r-\frac{1}{2}.   
\end{equation}
The estimate for \( \| \partial_\mathrm{t}^\mathrm{k} u(\cdot, \mathrm{t}) \|_{\mathbb{L}^2(D)} \) obtained above is insufficient to control the error terms that will arise in the subsequent section. Therefore, we require a more refined estimate to improve the accuracy of our results. The following steps are taken to achieve this. In order to proceed to further steps, we need the following proposition.
\begin{proposition}\label{prop2}
    The heat volume potential $\mathbf{\bm{V}}^t_{\bm{D}}$ can be extended to a bounded operator from $\mathrm{H}^{q}_{0,\sigma}\Big(0,\mathrm{T}; \mathrm{L}^2(\mathbf{D})\Big)$ to $\mathrm{H}^{q+\frac{1}{2}(1-r)}_{0,\sigma}\Big(0,\mathrm{T}; \mathrm{H}^r(\mathbb R^3)\Big)$,\; \text{for}\ r=0,1,2.
\end{proposition}
\begin{proof}
We begin by considering the Newtonian heat potential operator defined as
\[
    \bm{\mathrm{V}}^t_\mathbf{D}[\mathrm{g}](x,t) := \int_0^t \int_{\mathbf{D}} \Phi^{(m)}(x,t;y,\tau) \, \mathrm{g}(y,\tau) \, dy \, d\tau.
\]
We now note that the Laplace-Fourier transform $\mathrm{G}^{(\bm{s})}$ of the fundamental solution $\Phi^{(m)}(x,t;y,\tau)$ for the heat equation is actually the fundamental solution to the differential equation $-\Delta u + \bm{s} u = 0$, which is expressed as
\[
    \mathrm{G}^{(\bm{s})}(x,y) = \frac{1}{4\pi|x-y|} e^{-\sqrt{\bm{s}} \, |x-y|}.
\]
Based on this, we define the corresponding Newtonian potential as
\[
    \bm{\mathrm{V}}_{\bm{s},\mathbf{D}}[\hat{g}](x) := \int_{D} \mathrm{G}^{(\bm{s})}(x,y) \, \hat{g}(y) \, dy.
\]
In fact, it is known that for \(\hat{g}(y) \in L^2(\bm{D})\), extended by zero outside \(\bm{D}\), the Newtonian potential \(z := \bm{\mathrm{V}}_{\bm{s}, \mathbf{D}}[\hat{g}]\) satisfies the equation
\[
    -\Delta z + \bm{s} z = \hat{g} \quad \text{in} \; \mathbb{R}^3.
\]
Multiplying this equation by \(\overline{\bm{s}}^{1/2} \overline{z}\) and integrating over \(\mathbb{R}^3\), we obtain
\[
    \overline{\bm{s}}^{1/2} \|\nabla z\|_{L^2(\mathbb{R}^3)}^2 + s \overline{\bm{s}}^{1/2} \|z\|_{L^2(\mathbb{R}^3)}^2 = \overline{\bm{s}}^{1/2} \int_{\bm{D}} \hat{g} \, \overline{z} \, dx.
\]
By taking the real part of the above expression, we derive the following estimate
\begin{equation}\label{v1}
    \|z\|_{L^2(\bm{D})} \leq \frac{1}{\omega |\bm{s}|^{1/2}} \|\hat{g}\|_{L^2(\bm{D})}.
\end{equation}
Additionally, we have the following bound for the \(H^2(\mathbb{R}^3)\)-semi-norm of \(z\)
\[
    |z|_{H^2(\mathbb{R}^3)}^2 \leq \mathrm{C} \|\Delta z\|_{L^2(\mathbb{R}^3)}^2 \leq \mathrm{C} \left( |\bm{s}|^2 \|z\|_{L^2(\mathbb{R}^3)}^2 + \|\hat{g}\|_{L^2(\mathbb{R}^3)}^2 \right),
\]
where \(|\cdot|_{H^2(\mathbb{R}^3)}\) represents the semi-norm in the \(H^2(\mathbb{R}^3)\)-space. Therefore, using the previous estimate and assuming \(\frac{|\bm{s}|}{\sigma} \geq 1\), we obtain the bound
\begin{align}\label{v2}
        |z|_{H^2(\mathbb{R}^3)}^2 \leq \mathrm{C} \frac{|\bm{s}|}{\omega^2} \|\hat{g}\|_{L^2(\bm{D})}^2.
\end{align}
Similarly, we have the following estimate in the \(H^1_{\bm{s}}(\mathbb{R}^3)\)-norm
\begin{align}\label{v3}
        \|z\|_{H^1_{\bm{s}}(\mathbb{R}^3)} \leq \frac{1}{\omega} \|\hat{g}\|_{L^2(\bm{D})},
\end{align}
where \(\|z\|_{H^1_{\bm{s}}(\mathbb{R}^3)}^2 := \|\nabla z\|_{L^2(\mathbb{R}^3)}^2 + |\bm{s}|\|z\|_{L^2(\mathbb{R}^3)}^2\).
Using the estimate above, we can define the corresponding time-domain operator-valued distributions via the inverse Laplace transform, leading to the heat potential operator \(\bm{\mathrm{V}}^t_{\bm{D}}\). This operator satisfies the associated heat equation
\[
    \kappa \partial_t u - \Delta u = g \quad \text{in} \; \mathrm{H}^q_{0,\sigma}(0,\mathrm{T}; \mathrm{L}^2(\bm{D})).
\]
Further details on distributional convolution can be found in \cite{sayas2} and \cite[Chapter 3]{sayas}. Consequently, by employing techniques from \cite{hduong,le-monk} and applying the estimates (\ref{v1}), (\ref{v2}), (\ref{v3}), the mapping properties follow from the inverse Fourier-Laplace transform and Lemma \ref{lubich}.
\\
\noindent
This completes the proof of Proposition \ref{prop2}.
\end{proof}
\noindent
Let us now state the following proposition.
\begin{proposition}\label{proposition-maxwell-1}      
Consider the electromagnetic scattering problem (\ref{Maxwell-model}) for cluster of plasmonic nanoparticles $D_i$ for $i=1,2,\ldots,M.$ Then, for $h,\lambda$ satisfying
    \begin{align}\nonumber
          3-3\lambda-h \ge 0 \quad \text{and} \quad \frac{9}{5}< h <2,
    \end{align}
we have the following a-priori estimates
    \begin{align}\nonumber
           \underset{i}{\max}\big\Vert \mathrm{E}_i\big\Vert_{\mathbb{L}^2(D_i)} \lesssim \delta^{\frac{3}{2}-h}.
    \end{align}
In addition to that, we have
    \begin{align}\nonumber
         \big\Vert E_i\big\Vert_{\mathbb{L}^\mathrm{4}(D_i)} \sim \delta^{\frac{3}{4}-\mathrm{h}}.
    \end{align}
\end{proposition}
\begin{proof}
    See Section \ref{a-priori} for the proof.
\end{proof}
\noindent
We now proceed to complete the proof to establish the first estimate stated in Lemma \ref{lemma-apriori-heat}.

\subsubsection*{Part 1: The single inclusion case}\label{part1}

We begin by examining the Lippmann-Schwinger equation (\ref{lippmann}) in the Laplace-Fourier domain, which is given by the following expression
\begin{align}\nonumber
    \hat{u}(x,\bm{s}) = - \vartheta \bm{s} \mathcal{V}_{\bm{s},D}[\hat{u}](x,\bm{s}) - \alpha \mathcal{S}_{\bm{s},\partial D}[\partial_\nu\hat{u}](x,\bm{s}) + \frac{\gamma_m}{\gamma_p} \mathcal{V}_{\bm{s},D}[\hat{\mathrm{J}}](x,\bm{s}),
\end{align}
where \(\mathcal{S}_{\bm{s},\partial D}\) denotes the single-layer potential associated with the Helmholtz equation, which holds for \(x \notin \partial D\) and for all frequencies \(\bm{s}\). Similarly, \(\mathcal{V}_{\bm{s},D}\) represents the Newtonian potential of the Helmholtz equation.
\\
Next, utilizing the scaling property, we derive the following expression
\begin{align}
    \Tilde{\hat{u}}(\xi,\bm{s}) = -\delta^2 \vartheta \bm{s} \mathcal{V}_{\bm{s},B}[\Tilde{\hat{u}}](\xi,\bm{s}) - \delta \alpha \mathcal{S}_{\bm{s},\partial B}[\widetilde{\partial_\nu\hat{u}}](\xi,\bm{s}) + \delta^2 \frac{\gamma_m}{\gamma_p} \mathcal{V}_{\bm{s},B}[\Tilde{\hat{\mathrm{J}}}](\xi,\bm{s}).
\end{align}
Taking the \(L^2(B)\)-norm of the above equation, we obtain
\begin{align}
    \Vert\Tilde{\hat{u}}(\cdot,\bm{s})\Vert_{\mathrm{L}^2(B)} = \delta^2 \vartheta |\bm{s}| \Vert\mathcal{V}_{\bm{s},B}[\Tilde{\hat{u}}]\Vert_{\mathrm{L}^2(B)} + \delta \alpha \Vert\mathcal{S}_{\bm{s},\partial B}[\widetilde{\partial_\nu\hat{u}}]\Vert_{\mathrm{L}^2(B)} + \delta^2 \frac{\gamma_m}{\gamma_p} \Vert\mathcal{V}_{\bm{s},B}[\Tilde{\hat{\mathrm{J}}}]\Vert_{\mathrm{L}^2(B)}.
\end{align}
By applying the bounded estimates \(\Vert\mathcal{S}_{\bm{s},\partial B}[\widetilde{\partial_\nu\hat{u}}]\Vert_{\mathrm{H}^1(\mathbb{R}^3)} \le \frac{|\bm{s}|^{1/2}}{\omega \underline{\omega}^2} \Vert\widetilde{\partial_\nu\hat{u}}\Vert_{\mathrm{H}^{-1/2}(\partial D)}\) (see \cite[Theorem 2.1]{sayas2}) and
\\
\(\Vert\mathcal{V}_{\bm{s},B}[g]\Vert_{L^2(\mathbb{R}^3)} \le \frac{1}{\omega |\bm{s}|^{1/2}} \Vert\hat{g}\Vert_{L^2(\mathbb{R}^3)}\) by Proposition \ref{prop2}, along with the continuous Sobolev embeddings \(H^1(\mathbb{R}^3) \hookrightarrow L^2(B)\) and \(L^2(\partial B) \hookrightarrow H^{-1/2}(\partial B)\), we deduce
\begin{align}
    \Vert\Tilde{\hat{u}}(\cdot,\bm{s})\Vert_{\mathrm{L}^2(B)} 
    &\le \delta^2 \vartheta \frac{|\bm{s}|^{1/2}}{\omega} \Vert\Tilde{\hat{u}}(\cdot,\bm{s})\Vert_{\mathrm{L}^2(B)} + \delta \alpha \frac{|\bm{s}|^{1/2}}{\omega \underline{\omega}^2} \Vert\widetilde{\partial_\nu\hat{u}}(\cdot,\bm{s})\Vert_{\mathrm{H}^{-1/2}(\partial B)} \nonumber \\
    &\quad + \delta^2 \frac{\gamma_m}{\gamma_p} \frac{1}{\omega |\bm{s}|^{1/2}} \Vert\Tilde{\hat{\mathrm{J}}}(\cdot,\bm{s})\Vert_{\mathrm{L}^2(B)}.
\end{align}
Using the initial estimates \(\Vert\Tilde{\hat{u}}(\cdot,\bm{s})\Vert_{\mathrm{L}^2(B)} \sim \delta^{-h}\), \(\Vert\widetilde{\partial_\nu\hat{u}}(\cdot,\bm{s})\Vert_{\mathrm{L}^2(\partial B)} \sim \delta^{\beta-h}\), $\alpha \sim \delta^{-\beta}$ and \(\Vert\Tilde{\hat{\mathrm{J}}}(\cdot,\bm{s})\Vert_{\mathrm{L}^2(B)} \sim \delta^{-h}\), we obtain
\begin{align}
    \Vert\Tilde{\hat{u}}(\cdot,\bm{s})\Vert_{\mathrm{L}^2(B)} 
    &\le \vartheta \frac{|\bm{s}|^{1/2}}{\omega} \delta^{2-h} + \frac{|\bm{s}|^{1/2}}{\omega \underline{\omega}^2} \delta^{1-h} + \frac{\gamma_m}{\gamma_p} \frac{1}{\omega |\bm{s}|^{1/2}} \delta^{2-h} \nonumber \\
    &\le \frac{|\bm{s}|^{1/2}}{\omega \underline{\omega}^2} \delta^{1-h}.
\end{align}
Thus, using Lemma \ref{important} and the preceding discussion, we deduce the following point-wise estimate
\begin{align}
    \Vert u(\cdot,\mathrm{t})\Vert_{\mathbb{L}^2(D)} \lesssim \Vert u\Vert_{\mathrm{H}_{0,\sigma}^p(0,\mathrm{T};\mathbb{L}^2(D))} \lesssim \delta^{\frac{5}{2}-h}.
\end{align}
Repeating the process for the time derivative, we derive the following estimate
\begin{equation}
    \| \partial_\mathrm{t}^\mathrm{k} u(\cdot, \mathrm{t}) \|_{\mathbb{L}^2(D)} \lesssim \delta^{\frac{5}{2}-h}, \quad \mathrm{t} \in [0,\mathrm{T}].
\end{equation}

\subsubsection*{Part 2: The multiple inclusion case}

The solution to equation (\ref{heat}) is expressed as the following integral equation for \( i = 1, 2, \ldots, M \)
\begin{align}\nonumber
    u^{(i)}(x,t) + \sum_{i=1}^M \vartheta_i\ \bm{\mathbcal{V}^t}_{D_i}\big[\partial_t u\big](\mathrm{x},\mathrm{t}) + \sum_{i=1}^M\alpha_i\ \bm{\mathcal{S}^t}_{\partial D_i}\big[\partial_\nu u^{(i)}\big](\mathrm{x},\mathrm{t}) = \sum_{i=1}^M \frac{\gamma_m}{\gamma_{p_i}} \bm{\mathbcal{V}^t}_{D_i}\big[\mathrm{J}\big](\mathrm{x},\mathrm{t}),\; (\mathrm{x},\mathrm{t}) \in D_i \times (0,\mathrm{T}).
\end{align}    
Subsequently, we have
\begin{align}\nonumber
    u^{(i)}(x,t)\nonumber &+ \vartheta_i\ \bm{\mathbcal{V}^{(ii)}}\big[\partial_t u^{(i)}\big](\mathrm{x},\mathrm{t}) + \alpha_i\ \bm{\mathcal{S}^{(ii)}}\big[\partial_\nu u^{(i)}\big](\mathrm{x},\mathrm{t}) + \sum\limits_{\substack{j=1 \\ j\neq i}}^M \vartheta_j\ \bm{\mathbcal{V}^{(ij)}}\big[\partial_t u^{(j)}\big](\mathrm{x},\mathrm{t}) \\ \nonumber&+ \sum\limits_{\substack{j=1 \\ j\neq i}}^M\alpha_j\ \bm{\mathcal{S}^{(ij)}}\big[\partial_\nu u^{(j)}\big](\mathrm{x},\mathrm{t}) = \frac{\gamma_m}{\gamma_{p_i}}\bm{\mathbcal{V}^t}_{D_i}\big[\mathrm{J}\big](\mathrm{x},\mathrm{t})  + \sum\limits_{\substack{j=1 \\ j\neq i}}^M \frac{\gamma_m}{\gamma_{p_j}}\bm{\mathbcal{V}^{(ij)}}\big[\mathrm{J}\big](\mathrm{x},\mathrm{t}),\; (\mathrm{x},\mathrm{t}) \in D_i \times(0,\mathrm{T}).
\end{align} 
In the Fourier-Laplace domain, this integral equation becomes
\begin{align}
    \hat{u}^{(i)}(x,\bm{s}) \nonumber &= - \vartheta_i\ \bm{s}\ \mathcal{V}^{(ii)}_{\bm{s},D_i}\big[\hat{u}^{(i)}\big](x,\bm{s}) - \alpha_i\ \mathcal{S}^{(ii)}_{\bm{s},\partial  D_i}\big[\partial_\nu\hat{u}^{(i)}\big](x,\bm{s}) + \frac{\gamma_m}{\gamma_{p_i}}\ \mathcal{V}^{(ii)}_{\bm{s},D_i}\big[\hat{\mathrm{J}}\big](x,\bm{s}) \\ 
    &- \sum\limits_{\substack{j=1 \\ j\neq i}}^M\vartheta_j\ \bm{s}\ \mathcal{V}^{(ij)}_{\bm{s},D_j}\big[\hat{u}^{(j)}\big](x,\bm{s})
    - \sum\limits_{\substack{j=1 \\ j\neq i}}^M\alpha_j\ \mathcal{S}^{(ij)}_{\bm{s},\partial D_i}\big[\partial_\nu\hat{u}^{(j)}\big](x,\bm{s}) + \sum\limits_{\substack{j=1 \\ j\neq i}}^M\frac{\gamma_m}{\gamma_{p_j}}\ \mathcal{V}^{(ij)}_{\bm{s},D_i}\big[\hat{\mathrm{J}}\big](x,\bm{s}).
\end{align}
Next, taking the \( L^2(D_i) \)-norm of both sides results in
\begin{align}\label{l}
    \Vert\hat{u}^{(i)}(x,\bm{s})\Vert_{L^2(D_i)} 
    \nonumber&= \vartheta_i\ |\bm{s}|\ \Vert\mathcal{V}^{(ii)}_{\bm{s},D_i}\big[\hat{u}^{(i)}\big](x,\bm{s})\Vert_{L^2(D_i)} 
    + \alpha_i\ \Vert\mathcal{S}^{(ii)}_{\bm{s},\partial D_i}\big[\partial_\nu\hat{u}^{(i)}\big](x,\bm{s})\Vert_{L^2(D_i)} 
    + \frac{\gamma_m}{\gamma_{p_i}}\ \Vert\mathcal{V}^{(ii)}_{\bm{s},D_i}\big[\hat{\mathrm{J}}\big](x,\bm{s})\Vert_{L^2(D_i)} \\ 
    &+\nonumber \sum\limits_{\substack{j=1 \\ j\neq i}}^M\vartheta_j\ |\bm{s}|\ \Vert\mathcal{V}^{(ij)}_{\bm{s},D_j}\big[\hat{u}^{(j)}\big](x,\bm{s})\Vert_{L^2(D_i)}
    + \sum\limits_{\substack{j=1 \\ j\neq i}}^M\alpha_j\ \Vert\mathcal{S}^{(ij)}_{\bm{s},\partial D_i}\big[\partial_\nu\hat{u}^{(j)}\big](x,\bm{s})\Vert_{L^2(D_i)}
    \\ &+ \sum\limits_{\substack{j=1 \\ j\neq i}}^M\frac{\gamma_m}{\gamma_{p_j}}\ \Vert\mathcal{V}^{(ij)}_{\bm{s},D_i}\big[\hat{\mathrm{J}}\big](x,\bm{s})\Vert_{L^2(D_i)}.
\end{align}
As the first three terms of this expression can be approximated similarly to the single inclusion case, we now focus on estimating the last three terms. First, we have
\begin{align}\nonumber
    \Vert\mathcal{V}^{(ij)}_{\bm{s},D_j}\big[\hat{u}^{(j)}\big](x,\bm{s})\Vert_{L^2(D_i)} 
    \nonumber&= \Bigg( \int_{D_i}\Big|\int_{D_j} G^s(x,y)\hat{u}^{(j)}(y,\bm{s})\,dy\Big|^2\,dx \Bigg)^{\frac{1}{2}}
    \\ &\lesssim \frac{1}{|z_i-z_j|}\,|D_i\times D_j|^{\frac{1}{2}} \,\Vert\hat{u}^{(j)}(\cdot,\bm{s})\Vert_{L^2(D_j)},
\end{align}
which leads to the estimate
\begin{align}\label{l1}
    \sum\limits_{\substack{j=1 \\ j\neq i}}^M \vartheta_j |\bm{s}| \Vert\mathcal{V}^{(ij)}_{\bm{s},D_j}\big[\hat{u}^{(j)}\big](x,\bm{s})\Vert_{L^2(D_i)} 
    \lesssim \delta^{\frac{9}{2}-h} |\bm{s}| \sum\limits_{\substack{j=1 \\ j\neq i}}^M d_{ij}^{-1} \lesssim |\bm{s}| \delta^{\frac{3}{2}+\beta-h}.
\end{align}
Similarly, we obtain for $\alpha_j \sim \delta^{-\beta}$ that
\begin{align}\label{l2}
    \sum\limits_{\substack{j=1 \\ j\neq i}}^M\alpha_j \Vert\mathcal{S}^{(ij)}_{\bm{s},\partial D_j}\big[\partial_\nu\hat{u}^{(j)}\big](x,\bm{s})\Vert_{L^2(D_i)}
    \lesssim \delta^{\frac{7}{2}-h} \sum\limits_{\substack{j=1 \\ j\neq i}}^M d_{ij}^{-1} \lesssim \delta^{\frac{1}{2}+\beta-h},
\end{align}
and finally due to the fact $\gamma_{p_j} \sim \delta^{-\beta}$, we derive the following estimate
\begin{align}\label{l3}
    \sum\limits_{\substack{j=1 \\ j\neq i}}^M \frac{\gamma_m}{\gamma_{p_j}} \Vert\mathcal{V}^{(ij)}_{\bm{s},D_j}\big[\hat{\mathrm{J}}\big](x,\bm{s})\Vert_{L^2(D_i)} 
    \lesssim \delta^{\frac{3}{2}+2\beta-h}.
\end{align}
Combining the estimates from Section \ref{part1} and using equations (\ref{l1}), (\ref{l2}), and (\ref{l3}) in (\ref{l}), we conclude
\begin{align}\nonumber
    \Vert\hat{u}(\cdot,\bm{s})\Vert_{L^2(D_i)} \lesssim \frac{|\bm{s}|^{1/2}}{\omega \underline{\omega}^2} \delta^{\frac{1}{2}+\beta-h}.
\end{align}
Therefore, the following estimate holds
\begin{align}\nonumber
    \Vert u(\cdot,t)\Vert_{\mathbb{L}^2(D_i)} \lesssim \Vert u\Vert_{H_{0,\sigma}^r(0,T;\mathbb{L}^2(D_i))} \lesssim \delta^{\frac{1}{2}+\beta-h}.
\end{align}
Repeating the process after taking the derivative w.r.to time, we can infer the following estimate
\begin{equation}
\| \partial_\mathrm{t}^\mathrm{k} u(\cdot, \mathrm{t}) \|_{\mathbb{L}^2(D_i)} \lesssim \delta^{\frac{1}{2}+\beta-h}, \quad \mathrm{t} \in [0,\mathrm{T}],
\end{equation}
This completes the proof.
\end{proof}

    \subsection{Proof of Theorem \ref{th1}: The Single Inclusion Case} 

We begin by recalling the Lippmann-Schwinger equation as derived in (\ref{lippmann})
   \begin{align}\label{new-lip}
          u &+ \vartheta\frac{1}{\kappa_m}\int_0^t\int_{D} \Phi^{(m)}(\mathrm{x}-\mathrm{y};\mathrm{t}-\tau) \partial_t u(y,\tau) \, d\mathrm{y} \, d\tau 
          + \nonumber \alpha \frac{1}{\kappa_m}\int_0^t\int_{\partial D} \Phi^{(m)}(\mathrm{x}-\mathrm{y};\mathrm{t}-\tau) \partial_\nu u(y,\tau) \, d\mathrm{y} \, d\tau 
          \\ &= \frac{\gamma_m}{\gamma_p}\frac{1}{\kappa_m}\int_0^t\int_{D} \Phi^{(m)}(\mathrm{x}-\mathrm{y};\mathrm{t}-\tau) \mathrm{J}(y,\tau) \, d\mathrm{y} \, d\tau \; \text{in} \ D \times (0,T).
   \end{align}
The single-layer heat potential operator, $\bm{\mathcal{S}}^t_{\partial D}$, is defined as
    \begin{align}\nonumber
          \bm{\mathcal{S}}^t_{\partial D}[f](\mathrm{x}, \mathrm{t}) := \frac{1}{\kappa_m} \int_0^t \int_{\partial D} \Phi^{(m)}(\mathrm{x}, \mathrm{t}; \mathrm{y}, \tau) f(y, \tau) \, d\sigma_y \, d\tau,\; (\mathrm{x}, \mathrm{t}) \in D \times (0,T),
    \end{align}
and the corresponding volume heat potential operator, $\bm{\mathbcal{V}}^t_D$, is defined by
    \begin{align}\nonumber
          \bm{\mathbcal{V}}^t_D[f](\mathrm{x}, \mathrm{t}) := \frac{1}{\kappa_m} \int_0^t \int_{D} \Phi^{(m)}(\mathrm{x}, \mathrm{t}; \mathrm{y}, \tau) f(y, \tau) \, d\mathrm{y} \, d\tau,\; (\mathrm{x}, \mathrm{t}) \in D \times (0,T).
    \end{align}
We can now rewrite the Lippmann-Schwinger equation (\ref{new-lip}) as follows
    \begin{align}
          u(x,t) + \vartheta \bm{\mathbcal{V}}^t_D[\partial_t u](\mathrm{x},\mathrm{t}) + \alpha \bm{\mathcal{S}}^t_{\partial D}[\partial_\nu u](\mathrm{x},\mathrm{t}) = \frac{\gamma_m}{\gamma_p} \bm{\mathbcal{V}}^t_D[\mathrm{J}](\mathrm{x},\mathrm{t}),\; \text{where} \ (\mathrm{x},\mathrm{t}) \in D \times (0,T).
    \end{align}
Next, we consider the jump relation for the Neumann trace of the single-layer heat potential operator, which is given by $\partial_\nu^\pm \bm{\mathcal{S}}^t_{\partial D}[f] = \mp\frac{1}{2}f + \bm{\mathcal{K}}^t[f]$. Consequently, applying the Neumann trace to the Lippmann-Schwinger equation yields
    \begin{align} \label{boundary}
           (1 + \frac{\alpha}{2}) \partial_\nu u(x,t) + \vartheta \partial_\nu \bm{\mathbcal{V}}^t_D[\partial_t u](\mathrm{x},\mathrm{t}) + \alpha \bm{\mathcal{K}}^t_{\partial D}[\partial_\nu u](\mathrm{x},\mathrm{t}) = \frac{\gamma_m}{\gamma_p} \partial_\nu \bm{\mathbcal{V}}^t_D[\mathrm{J}](\mathrm{x},\mathrm{t}),\; \text{where} \ (\mathrm{x},\mathrm{t}) \in \partial D \times (0,T).
    \end{align}
Here, the adjoint of the Neumann–Poincaré operator, corresponding to the heat operator, is defined as:
    \begin{align}\nonumber
           \bm{\mathbcal{K}}^t_{\partial D}[f](x,t) := \frac{1}{\kappa_m} \int_0^t \int_{\partial D} \left(\frac{\kappa_m}{4\pi(t-\tau)}\right)^{\frac{3}{2}} \frac{\kappa_m \langle y-x \cdot \nu_x \rangle}{2(t-\tau)} e^{-\frac{\kappa_m |x-y|^2}{4(t-\tau)}} f(y,\tau) \, d\sigma_y \, d\tau.
    \end{align}
Integrating over the boundary $\partial D$, we obtain
    \begin{align}
           (1 + \frac{\alpha}{2}) \int_{\partial D} \partial_\nu u + \vartheta \int_{\partial D} \partial_\nu \bm{\mathbcal{V}}^t_D[\partial_t u] + \alpha \int_{\partial D} \bm{\mathcal{K}}^t_{\partial D}[\partial_\nu u] = \frac{\gamma_m}{\gamma_p} \int_{\partial D} \partial_\nu \bm{\mathbcal{V}}^t_D[\mathrm{J}],\nonumber
    \end{align}
which can be rewritten as
    \begin{align} \label{LP-2}
           (1 + \frac{\alpha}{2}) \int_{\partial D} \partial_\nu u + \vartheta \int_{\partial D} \partial_\nu \bm{\mathbcal{V}}^t_D[\partial_t u] + \alpha \int_{\partial D} \partial_\nu u \, \bm{\mathcal{K}}_{\partial D}[1] = \frac{\gamma_m}{\gamma_p} \int_{\partial D} \partial_\nu \bm{\mathbcal{V}}^t_D[\mathrm{J}],
    \end{align}
where the Neumann–Poincaré operator corresponding to the heat operator is defined as
    \begin{align}
           \nonumber\bm{\mathcal{K}}_{\partial D}[f](x,t) := \frac{1}{\kappa_m} \int_0^t \int_{\partial D} \left(\frac{\kappa_m}{4\pi(t-\tau)}\right)^{\frac{3}{2}} \frac{\kappa_m \langle x-y \cdot \nu_y \rangle}{2(t-\tau)} e^{-\frac{\kappa_m |x-y|^2}{4(t-\tau)}} f(y,\tau) \, d\sigma_y \, d\tau.
    \end{align}
Therefore,
     \begin{align}
           \bm{\mathbcal{K}}_{\partial D}\Big[1\Big](x,t) &= \int_{\partial D} \frac{\langle \mathrm{x} - \mathrm{y} \cdot \nu_\mathrm{y} \rangle}{4\pi} \frac{1}{|\mathrm{x} - \mathrm{y}|^3} \left[ \int_{0}^{t} \frac{\kappa_m^{3/2}}{2\sqrt{\pi}} \frac{|\mathrm{x} - \mathrm{y}|^3}{(t - \tau)^{3/2}} \frac{1}{2(t - \tau)} e^{-\frac{\kappa_m |\mathrm{x} - \mathrm{y}|^2}{4(\tau - \mathrm{s})}} \, d\tau \right] d\sigma_\mathrm{y}.
    \end{align}
Define
    \begin{equation}\label{defvarphi3D}
          \varphi(\mathrm{x}, \mathrm{y}, \mathrm{t}) := \int_{0}^{t} \frac{\kappa_m^{3/2}}{2\sqrt{\pi}} \frac{|\mathrm{x} - \mathrm{y}|^3}{(t - \tau)^{3/2}} \frac{1}{2(t - \tau)} e^{-\frac{\kappa_m |\mathrm{x} - \mathrm{y}|^2}{4(t - \tau)}} \, d\tau.
    \end{equation}
By performing the variable change \(\mathrm{m} := \frac{\sqrt{\kappa}_m |\mathrm{x} - \mathrm{y}|}{2 \sqrt{t - \tau}}\), we get $\tau = \mathrm{t} - \frac{\kappa_m |\mathrm{x} - \mathrm{y}|^2}{4 \mathrm{m}^2}$
and $ d\tau = \frac{1}{2} \kappa_m |\mathrm{x} - \mathrm{y}|^2 \mathrm{m}^{-3}.$ Thus, we obtain
    \begin{align}
           \varphi(\mathrm{x}, \mathrm{y}, \mathrm{t}) &= \frac{4}{\sqrt{\pi}} \int_{\frac{\sqrt{\kappa}_m |\mathrm{x} - \mathrm{y}|}{2 \sqrt{t}}}^{\infty} \mathrm{m}^2 \exp(-\mathrm{m}^2) \, d\mathrm{m}. \nonumber
    \end{align}
Considering the evaluation of the improper integral
    \begin{align}\nonumber
          \int_{\frac{|\mathrm{x} - \mathrm{y}|}{2 \sqrt{t}}}^{\infty} \mathrm{m}^2 \exp(-\mathrm{m}^2) \, d\mathrm{m} 
          &= \frac{\sqrt{\pi}}{4} + \frac{\sqrt{\kappa}_m |\mathrm{x} - \mathrm{y}|}{4 \sqrt{t}} e^{-\frac{\kappa_m |\mathrm{x} - \mathrm{y}|^2}{4t}} - \frac{\sqrt{\pi}}{4} \textbf{erf}\left(\frac{\sqrt{\kappa}_m |\mathrm{x} - \mathrm{y}|}{2 \sqrt{t}}\right) 
          \\ \nonumber&= \frac{\sqrt{\pi}}{4} + \frac{\sqrt{\kappa}_m |\mathrm{x} - \mathrm{y}|}{4 \sqrt{t}} \left(1 + \sum_{n=1}^\infty \frac{(-1)^n \left(\frac{\kappa_m |\mathrm{x} - \mathrm{y}|^2}{4t}\right)^n}{n!}\right) \\ \nonumber& \quad - \frac{\sqrt{\pi}}{4} \left(\frac{2}{\sqrt{\pi}} \frac{\sqrt{\kappa}_m |\mathrm{x} - \mathrm{y}|}{2 \sqrt{t}} + \sum_{n=1}^\infty \frac{(-1)^{2n+1} \left(\frac{\sqrt{\kappa}_m |\mathrm{x} - \mathrm{y}|}{2 \sqrt{t}}\right)^n}{(2n+1) n!}\right) 
          \\ \nonumber&= \frac{\sqrt{\pi}}{4} + \sum_{n=1}^\infty \frac{(-1)^n \left(\frac{\kappa_m |\mathrm{x} - \mathrm{y}|^2}{4t}\right)^n}{n!} - \frac{1}{2} \sum_{n=1}^\infty \frac{(-1)^n \left(\frac{\kappa_m |\mathrm{x} - \mathrm{y}|}{2 \sqrt{t}}\right)^{2n+1}}{(2n+1) n!} 
          \\&= \frac{\sqrt{\pi}}{4} + \mathcal{O}\left(\frac{|\mathrm{x} - \mathrm{y}|^3}{t^{3/2}}\right).
    \end{align}
For the regime where \( |\mathrm{x} - \mathrm{y}| \ll t \) and using the Maclaurin series for the error function \(\textbf{erf}(\mathrm{a}) = \frac{2}{\sqrt{\pi}} \sum_{n=0}^\infty \frac{(-1)^n \mathrm{a}^{2n+1}}{(2n+1) n!}\), we deduce
    \begin{align}
           \varphi(\mathrm{x}, \mathrm{y}, \mathrm{t}) 
           = \frac{4}{\sqrt{\pi}} \int_{\frac{|\mathrm{x} - \mathrm{y}|}{2 \sqrt{t}}}^{\infty} \mathrm{m}^2 \exp(-\mathrm{m}^2) \, d\mathrm{m} 
           = 1 + \mathcal{O}\left(\frac{|\mathrm{x} - \mathrm{y}|^3}{t^{3/2}}\right).
    \end{align}
Thus, we have
   \begin{align}
         \bm{\mathbcal{K}}_{\partial D}\Big[1\Big](x,t) 
         \nonumber&= \bm{\mathbcal{K}}_\text{Lap}\Big[1\Big](x,t) + \mathcal{O}\left(\int_{\partial D} \frac{\langle \mathrm{x} - \mathrm{y} \cdot \nu_\mathrm{y} \rangle}{4\pi |\mathrm{x} - \mathrm{y}|^3} \frac{|\mathrm{x} - \mathrm{y}|^3}{t^{3/2}} \, d\sigma_y\right) 
         \\ &= -\frac{1}{2} + \mathcal{O}\left(\int_{\partial D} \frac{\langle \mathrm{x} - \mathrm{y} \cdot \nu_\mathrm{y} \rangle}{4\pi |\mathrm{x} - \mathrm{y}|^3} \frac{|\mathrm{x} - \mathrm{y}|^3}{t^{3/2}} \, d\sigma_y\right),
   \end{align}
where the Neumann–Poincaré operator \(\mathbcal{K}_\text{Lap}\) corresponding to the Laplace operator is defined by
    \begin{align}\nonumber
         \mathbcal{K}_\text{Lap}\Big[f\Big](x) := \frac{1}{4\pi} \int_{\partial D} \frac{\langle x - y \cdot \nu_y \rangle}{|x - y|^3} \, f(y) \, d\sigma_y,
    \end{align}
where \( d\sigma_y \) represents the surface element, \(\nu_y\) is the inward normal vector on \(\partial D\), and the integral is interpreted in the principal value sense.
\\
By substituting the above expression into equation (\ref{LP-2}), we derive
    \begin{align}\label{expression}
           \int_{\partial D} \partial_\nu u + \vartheta \int_{\partial D} \partial_\nu \bm{\mathbcal{V}^t}_D\big[\partial_t u\big] + \alpha \mathcal{O}\left( \int_{\partial D} \int_{\partial D} \frac{\langle \mathrm{x} - \mathrm{y} \cdot \nu_\mathrm{y} \rangle}{4\pi |\mathrm{x} - \mathrm{y}|^3} \frac{|\mathrm{x} - \mathrm{y}|^3}{t^{3/2}} \partial_\nu u(y, \tau) \, d\sigma_y \right) = \frac{\gamma_m}{\gamma_p} \int_{\partial D} \partial_\nu \bm{\mathbcal{V}^t}_D\big[\mathrm{J}\big].
   \end{align}
\noindent 
Given that the volume potential \(\bm{\mathbcal{V}^t}_D[\mathrm{J}]\) satisfies the partial differential equation \(\partial_t u(x, t) - \Delta u(x, t) = \mathrm{J}(x, t)\) and applying Green's identity, we obtain
    \begin{align}\label{impr}
          \int_{\partial D} \partial_\nu u 
          \nonumber&= \frac{\gamma_m}{\gamma_p} \frac{1}{\kappa_m} \int_{\partial D} \partial_\nu \bm{\mathbcal{V}^t}_D\big[\mathrm{J}\big](x, t) \, d\sigma_x 
          - \underbrace{\mathcal{O}\left(\alpha \int_{\partial D} \int_{\partial D} \frac{\langle \mathrm{x} - \mathrm{y} \cdot \nu_\mathrm{y} \rangle}{4\pi |\mathrm{x} - \mathrm{y}|^3} \frac{|\mathrm{x} - \mathrm{y}|^3}{t^{3/2}} \partial_\nu u(y, \tau) \, d\sigma_y \right)}_{:= \textbf{err}^{(1)}} 
          \\ \nonumber&- \vartheta \underbrace{\int_{\partial D} \partial_\nu \bm{\mathbcal{V}^t}_D\big[\partial_t u\big](x, t) \, d\sigma_x}_{:= \textbf{err}^{(2)}} 
          \\&\nonumber= \frac{\gamma_m}{\gamma_p} \frac{1}{\kappa_m} \int_D \Delta \bm{\mathbcal{V}^t}_D\big[\mathrm{J}\big](x, t) \, dx + \textbf{err}^{(1)} + \textbf{err}^{(2)} \\
          &= \underbrace{\frac{\gamma_m}{\gamma_p} \frac{D \cdot \Im(\varepsilon_p)}{2\pi \kappa_m} f(t) \int_D |\mathrm{E}|^2 \, d\mathrm{y}}_{:= \text{Dominant Term}} + \underbrace{\frac{\gamma_m}{\gamma_p} \frac{1}{\kappa_m} \int_D \partial_t \bm{\mathbcal{V}^t}_D\big[\mathrm{J}\big](x, t) \, dx}_{:= \textbf{err}^{(3)}} + \textbf{err}^{(1)} + \textbf{err}^{(2)}.
    \end{align}
We now focus on estimating the term \(\textbf{err}^{(3)}\). First, since \(\mathrm{J}(x, t)\) satisfies \(\mathrm{J}(\cdot, 0) = 0\) identically, and referring to \cite[Theorem 5]{friedman}, the fundamental solution \(\Phi^{(m)}(x, t; y, \tau)\) satisfies \(\big(\partial_t - \Delta\big)\Phi^{(m)}(x, t; y, \tau) = \delta_0(x, y)\), then by differentiating \(\bm{\mathbcal{V}^t}_D\big[\mathrm{J}\big]\) with respect to time and using integration by parts, it follows that
$$
\partial_t \bm{\mathbcal{V}^t}_D\big[\mathrm{J}\big] = \bm{\mathbcal{V}^t}_D\big[\partial_t \mathrm{J}\big].
$$
Thus, we can rewrite the above expression as
    \begin{align}\label{abc}
          \bm{\mathbcal{V}^t}_D\big[\partial_t \mathrm{J}\big] &= \int_D \frac{1}{4\pi |x - y|} \partial_t \mathrm{J}(y, \tau) \, dy 
          + \int_D \frac{1}{4\pi |x - y|} \left(\varphi(x, y, t) - \partial_t \mathrm{J}(y, t)\right) \, dy.
    \end{align}
To estimate the second term in the above equation, we refer to the following lemma. The justification follows similarly to the proof of Lemma 3.2 in \cite{sini-haibing}, and thus we omit it here.
\begin{lemma}\label{singlelayer}    

We set $\varphi(\mathrm{x},\mathrm{y},\mathrm{t})$ as follows:
    \begin{equation}\nonumber
          \varphi(x,\mathrm{y},\mathrm{t}) := \displaystyle\int_{0}^{t} \frac{\kappa_m^\frac{3}{2}|x-y|}{2\sqrt{\pi}(t-\tau)^\frac{3}{2}}\textbf{exp}\big(-\dfrac{|\mathrm{x}-\mathrm{y}| ^2}{4(\mathrm{t}-\tau)}\big)f(y,\tau)d\tau.
    \end{equation}
Then we have
    \begin{equation}\nonumber
          \varphi(\mathrm{x},\mathrm{y},\mathrm{t})-f(y,\tau) = \mathcal{O}\bigg(|\mathrm{x}-\mathrm{y}| \ \Vert \partial_\mathrm{t}f(y,\cdot)\Vert_{\mathrm{L}^{2}(0,t)}\bigg),
    \end{equation}
for $\mathrm{x}, \mathrm{y}$ such that $|\mathrm{v}-\mathrm{y}| \ll \mathrm{t}$ and $\mathrm{t} \in (0,\mathrm{T}]$ uniformly with respect to $D$. 

\end{lemma}     
\noindent
Taking the integral with respect to \(D\) to the expression (\ref{abc}), we obtain
    \begin{align}\label{2.20}
          \int_{D} \bm{\mathbcal{V}^t}_D\big[\partial_t \mathrm{J}\big] &= \underbrace{\int_{D} \int_D \frac{1}{4\pi |x - y|} \mathrm{J}(y, \tau) \, dy \, dx}_{:= \textbf{err}^{(4)}} 
          + \underbrace{\int_{D} \int_D \frac{1}{4\pi |x - y|} \left(\varphi(x, y, t) - \mathrm{J}(y, t)\right) \, dy \, dx}_{:= \textbf{err}^{(5)}}.
    \end{align}
Then, we deduce
    \begin{align} \label{err4}
           \textbf{err}^{(4)} := \Big| \int_{D}\int_{D} \frac{1}{4\pi|\mathrm{x}-\mathrm{y}|}\partial_t J(y,\tau) dy \Big| 
           &\nonumber\lesssim \Vert \partial_tf(t)\Vert_{L^\infty}\Big(\int_{D\times D}\frac{1}{4\pi|\mathrm{x}-\mathrm{y}|^2}\Big)^\frac{1}{2}\ |D|^\frac{1}{2}\ \Im(\varepsilon_p)\ \big\Vert \mathrm{E}\big\Vert^2_{L^4(D)}
           \\ &\nonumber\lesssim \Big(|D|\sup\limits_{y\in \partial D}\int_{D}\dfrac{dx}{|\mathrm{x}-\mathrm{y}|^2}\Big)^\frac{1}{2} \ |D|^\frac{1}{2}\ \Im(\varepsilon_p)\ \big\Vert \mathrm{E}\big\Vert^2_{L^4(D)}
           \\ &\lesssim \delta^{\frac{9}{2}-h}.
    \end{align}
Again,
    \begin{align}\label{err5}
          \textbf{err}^{(4)} := \Big|\int_{D}\int_{D} \frac{1}{4\pi|\mathrm{x}-\mathrm{y}|}\Big(\varphi(\mathrm{x},\mathrm{y},\mathrm{t})-\partial_t J(y,t)\Big)dy\Big|
          &\nonumber\lesssim |D\times D|^\frac{1}{2}\ |D|^\frac{1}{2}\ \Vert \partial^2_t \mathrm{J}(\cdot,t)\Vert_{L^2\big(D\big)}
          \\ &\nonumber\lesssim |D\times D|^\frac{1}{2} \ |D|^\frac{1}{2}\ \Vert \partial^2_tf(t)\Vert_{L^\infty} \ \Im(\varepsilon_p)\big\Vert \mathrm{E}\big\Vert^2_{L^4(D)}
          \\ &\lesssim \delta^{6-h}.
    \end{align}
Therefore, combining the estimates (\ref{err4}), (\ref{err5}), and plugging these into (\ref{2.20}), we deduce that
    \begin{align}\label{err1}
           \textbf{err}^{(3)} := \Bigg|\frac{\gamma_m}{\gamma_p}\frac{1}{\kappa_m} \int_D \partial_t\bm{\mathbcal{V}^t}_D\big[\mathrm{J}\big](x,t)\ dx\Bigg| \lesssim \delta^{\frac{9}{2}+\beta-h}.
    \end{align}
Recall that if \(D\) is a bounded open subset of \(\mathbb{R}^n\) of class \(\mathcal{C}^{1,\alpha}\) with \(\alpha \in (0,1]\), then there exists a constant \(c_{D,\alpha} > 0\) such that
\[
|(\mathbf{y} - \mathbf{v}) \cdot \nu_{\mathbf{v}}| \leq c_{D,\alpha} |\mathbf{y} - \mathbf{v}|^{1+\alpha}.
\]
Consequently, using the a priori estimate for \(\|\partial_\nu u(y, \cdot)\|_{L^2(\partial D)}\sim \delta^{1+\beta-h}\) and noting that \(\alpha \sim \delta^{-\beta}\), we obtain
    \begin{align}\label{2}
          \textbf{err}^{(1)} := \mathcal{O}\left(\alpha \int_{\partial D} \int_{\partial D} \frac{\langle \mathbf{x} - \mathbf{y} \cdot \nu_\mathbf{y} \rangle}{4\pi |\mathbf{x} - \mathbf{y}|^3} \frac{|\mathbf{x} - \mathbf{y}|^3}{t^{3/2}} \partial_\nu u(y, \tau) \, d\sigma_y \right) = \mathcal{O}(\delta^{6-h}).
    \end{align}
Next, we focus on estimating the term \(\textbf{err}^{(2)}\). We have
    \begin{align}\label{2.27}
        \textbf{err}^{(2)} 
        &\nonumber:= \vartheta \int_{\partial D} \partial_\nu \bm{\mathbcal{V}^t}_D \left[\partial_t u\right](x, t) \, d\sigma_x 
        \\ &\lesssim \Vert 1 \Vert_{L^2(\partial D)} \Vert \partial_\nu \bm{\mathbcal{V}^t}_D \left[\partial_t u\right] \Vert_{L^2(\partial D)}
    \end{align}
We follow a similar approach as we did to estimate in (\ref{err1}) to deduce
    \begin{align}\nonumber
           \partial_{\nu_x} \bm{\mathbcal{V}^t}_D\left[\partial_t u\right](x,t) 
           &= \int_0^t \int_D \partial_{\nu_x} \Phi^{(m)}(x,t; y, \tau) \, \partial_t u(y, \tau) \, d\tau \\
           &= \int_D \frac{\langle \mathbf{y} - \mathbf{x} \cdot \nu_\mathbf{x} \rangle}{4\pi |\mathbf{x} - \mathbf{y}|^3} \left[ \int_0^t \frac{1}{2\sqrt{\pi}} \frac{|\mathbf{x} - \mathbf{y}|^3}{(t - \tau)^{3/2}} \frac{1}{2(t - \tau)} e^{-\frac{|\mathbf{x} - \mathbf{y}|^2}{4(\tau - \mathrm{s})}} \, d\tau \right] d\sigma_\mathbf{y}. \nonumber
    \end{align}
We define
    \begin{equation}
          \varphi(\mathbf{x}, \mathbf{y}, t) := \int_0^t \frac{1}{2\sqrt{\pi} (t - \tau)^{3/2}} \frac{|\mathbf{x} - \mathbf{y}|^3}{(t - \tau)} e^{-\frac{|\mathbf{y} - \mathbf{v}|^2}{4(\mathrm{s} - \tau)}} u(y, \tau) \, d\tau, \nonumber
    \end{equation}
from which we deduce
    \begin{align}\label{normalv}
          \partial_{\nu_x} \bm{\mathbcal{V}^t}_D \left[\partial_t u\right](x,t) 
          &= \int_D \frac{\langle \mathbf{y} - \mathbf{x} \cdot \nu_\mathbf{x} \rangle}{4\pi |\mathbf{x} - \mathbf{y}|^3} \partial_t u(y, \tau) \, dy 
          + \int_D \frac{\langle \mathbf{y} - \mathbf{x} \cdot \nu_\mathbf{x} \rangle}{4\pi |\mathbf{x} - \mathbf{y}|^3} \left(\varphi(\mathbf{x}, \mathbf{y}, t) - \partial_t u(y, t)\right) \, dy.
    \end{align}
Integrating with respect to \(\partial D\), we obtain
    \begin{align}\label{err2}\nonumber
          \Vert \partial_\nu \bm{\mathbcal{V}^t}_D \left[\partial_t u\right] \Vert_{L^2(\partial D)}
          &=\Bigg(\int_{\partial D} \Big|\partial_{\nu_x} \bm{\mathbcal{V}^t}_D \left[\partial_t u\right](x,t)\Big|^2 \Bigg)^\frac{1}{2}
          \\ \nonumber&= \underbrace{\Bigg(\int_{\partial D} \Big|\int_D \frac{\langle \mathbf{y} - \mathbf{x} \cdot \nu_\mathbf{x} \rangle}{4\pi |\mathbf{x} - \mathbf{y}|^3} \partial_t u(y, \tau) \, dy\Big|^2\Bigg)^\frac{1}{2}}_{:= \text{err}^{(6)}} 
          \\ &+ \underbrace{\Bigg(\int_{\partial D} \Big|\int_D \frac{\langle \mathbf{y} - \mathbf{x} \cdot \nu_\mathbf{x} \rangle}{4\pi |\mathbf{x} - \mathbf{y}|^3} \left(\varphi(\mathbf{x}, \mathbf{y}, t) - \partial_t u(y, t)\right) \, dy\Big|^2\Bigg)^\frac{1}{2}}_{:= \text{err}^{(7)}}.
    \end{align}
To estimate the second term in the equation above, we state the following lemma.
\begin{lemma}\label{lemma1}
    Define \(\varphi(\mathbf{x}, \mathbf{y}, t)\) as follows:
\begin{equation}\nonumber
    \varphi(\mathbf{x}, \mathbf{y}, t) := \int_0^t \frac{1}{2\sqrt{\pi} (t - \tau)^{3/2}} \frac{|\mathbf{x} - \mathbf{y}|^3}{(t - \tau)} e^{-\frac{|\mathbf{y} - \mathbf{v}|^2}{4(\mathrm{s} - \tau)}} u(y, \tau) \, d\tau.
\end{equation}
Then
\begin{equation}\nonumber
    \varphi(\mathbf{x}, \mathbf{y}, t) - u(y, t) = \mathcal{O}\left(|\mathbf{x} - \mathbf{y}| \, \| \partial_t u(y, \cdot) \|_{L^2(0, t)}\right),
\end{equation}
for \(\mathbf{x}, \mathbf{y}\) such that \(|\mathbf{x} - \mathbf{y}| \ll t\) and \(t \in (0, T]\), uniformly with respect to \(D\).
\end{lemma}
\begin{proof}
    We start by considering the change of variable \(\mathrm{m} := \frac{|\mathrm{x} - \mathrm{y}|}{2\sqrt{t - \tau}}\). Then it follows that \(\tau = t - \frac{|\mathrm{x} - \mathrm{y}|^2}{4\mathrm{m}^2}\). Differentiating \(\tau\) with respect to \(m\), we obtain:
\[ 
d\tau = \frac{|\mathrm{x} - \mathrm{y}|^2}{2} \mathrm{m}^{-3} \, dm.
\] 
Therefore, we deduce
\begin{align*}
    \varphi(\mathrm{x},\mathrm{y},\mathrm{t}) = \frac{4}{\sqrt{\pi}}\displaystyle\int_{\frac{|\mathrm{x}-\mathrm{y}|}{2\sqrt{t}}}^{\infty}\mathrm{m}^2e^{-\mathrm{m}^2}\ u\big(y,\mathrm{t} - \frac{|\mathrm{x}-\mathrm{y}|^2}{4\mathrm{m}^2}\big)d\mathrm{m},
\end{align*} 
which, we rewrite as
\begin{align*}
    \varphi(\mathrm{x},\mathrm{y},\mathrm{t}) = \frac{4}{\sqrt{\pi}}\displaystyle\int_{\frac{|\mathrm{x}-\mathrm{y}|}{2\sqrt{t}}}^{\infty}\mathrm{m}^2e^{-\mathrm{m}^2}\ \Big(u\big(y,\mathrm{t} - \frac{|\mathrm{x}-\mathrm{y}|^2}{4\mathrm{m}^2}\big)-u(y,t)\Big) dm+ \frac{4}{\sqrt{\pi}}u(y,t)\displaystyle\int_{\frac{|\mathrm{x}-\mathrm{y}|}{2\sqrt{t}}}^{\infty}\mathrm{m}^2e^{-\mathrm{m}^2}\ dm
\end{align*}
We now use the fact that $\displaystyle \int_0^\infty\mathrm{m}^2\textbf{exp}(-\mathrm{m}^2)d\mathrm{m}= \frac{\sqrt{\pi}}{4}$ to obtain
\begin{align}\label{esti}
    \varphi(\mathrm{x},\mathrm{y},\mathrm{t}) - u(y,t) = \frac{4}{\sqrt{\pi}}\displaystyle\int_{\frac{|\mathrm{x}-\mathrm{y}|}{2\sqrt{t}}}^{\infty}\mathrm{m}^2e^{-\mathrm{m}^2}\ \Big(u\big(y,\mathrm{t} - \frac{|\mathrm{x}-\mathrm{y}|^2}{4\mathrm{m}^2}\big)-u(y,t)\Big) dm - \frac{4}{\sqrt{\pi}}u(y,t)\displaystyle\int_0^{\frac{|\mathrm{x}-\mathrm{y}|}{2\sqrt{t}}}\mathrm{m}^2e^{-\mathrm{m}^2}\ dm
\end{align}
Considering $ u(y,0) = 0,$ we have $\displaystyle\partial_t u = \int_0^t \partial_t(y,s)ds$ and then by using Cauchy-Schwartz inequality, we deduce the following estimate
\begin{align*}
    u(y,t) = \mathcal{O}\Big(t^\frac{1}{2} \Vert \partial_t u(y,\cdot)\Vert_{L^2(0,t)}\Big)
\end{align*}
We now observe that $\mathrm{m}\ge \frac{\sqrt{\alpha}|\mathrm{y}-\mathrm{v}|}{2\sqrt{t}}$ which implies $t-\frac{|\mathrm{x}-\mathrm{y}|^2}{4\mathrm{m}^2} \ge 0$. Consequently, we derive the similar estimate as before
\begin{align*}
    u\big(y,\mathrm{t} - \frac{|\mathrm{x}-\mathrm{y}|^2}{4\mathrm{m}^2}\big)-u(y,t) 
    &\nonumber= \int_t^{\frac{|\mathrm{x}-\mathrm{y}|^2}{4m^2}}\partial_t u(y,s)ds
    \\ &\nonumber = \mathcal{O}\Big( \frac{|\mathrm{x}-\mathrm{y}|^2}{4\mathrm{m}^2}\Vert \partial_t u(y,\cdot)\Vert_{L^2(0,t)}\Big)
\end{align*}
Now, plugging the above two estimate in (\ref{esti}) to obtain
\begin{align}\label{e3}
   \varphi(\mathrm{x},\mathrm{y},\mathrm{t}) - u(y,t) 
   &= \nonumber\mathcal{O}\bigg(\frac{4}{\sqrt{\pi}}\displaystyle \int_{\frac{|\mathrm{x}-\mathrm{y}|}{2\sqrt{t}}}^{\infty}\mathrm{m}^2\bm{\mathrm{e}}^{-\mathrm{m}^2}d\mathrm{m}\frac{|\mathrm{x}-\mathrm{y}|^2}{4m^2}\Vert \partial_t u(y, \cdot )\Vert_{\mathrm{L}^2(0,t)}\bigg)
   \\ &+ \mathcal{O} \bigg(\frac{4}{\sqrt{\pi}}\displaystyle\int_{0}^{\frac{|\mathrm{x}-\mathrm{y}|}{2\sqrt{t}}}\mathrm{m}^2\bm{\mathrm{e}}^{-\mathrm{m}^2}d\mathrm{m}\;t^{\frac{1}{2}}\Vert \partial_t u(y, \cdot )\Vert_{\mathrm{L}^2(0,t)}\bigg). 
\end{align}
Therefore, considering the regime $|\mathrm{x}-\mathrm{y}|\ll \sqrt{t}$, and as we have error function's Maclaurin series as: $\textbf{erf}(\mathrm{a}) = \frac{2}{\sqrt{\pi}}\displaystyle\sum_{\mathrm{n=0}}^\infty \frac{(-1)^\mathrm{n}\mathrm{a}^{2\mathrm{n}+1} }{(2\mathrm{n}+1)\mathrm{n}!},$ we deduce 
\begin{align}\label{e1}
    \int_{0}^{\frac{|\mathrm{x}-\mathrm{y}|}{2\sqrt{t}}}\mathrm{m}^2\bm{\mathrm{e}}^{-\mathrm{m}^2}d\mathrm{m} 
    &\nonumber= \frac{\sqrt{\pi}}{4}\ \textbf{erf}\Big(\frac{|\mathrm{x}-\mathrm{y}|}{2\sqrt{t}}\Big) - \frac{|\mathrm{x}-\mathrm{y}|}{4\sqrt{t}}\ e^-{\dfrac{|\mathrm{x}-\mathrm{y}|^2}{4t}}
    \\ &= \mathcal{O}\Big(\frac{1}{t^\frac{3}{2}}|\mathrm{x}-\mathrm{y}|^3\Big).
\end{align}
Similarly, we have
\begin{align}\label{e2}
    \int_{\frac{|\mathrm{x}-\mathrm{y}|}{2\sqrt{t}}}^\infty\bm{\mathrm{e}}^{-\mathrm{m}^2}d\mathrm{m} = \frac{\sqrt{\pi}}{2}-\frac{\sqrt{\pi}}{2}\textbf{erf}\Big(\frac{|\mathrm{x}-\mathrm{y}|}{2\sqrt{t}}\Big).
\end{align}
Then, if we plug the previous two estimates (\ref{e1}) and (\ref{e2}) in (\ref{e3}), we deduce
\begin{equation}
     \varphi(\mathrm{x},\mathrm{y},\mathrm{t})-u(y,t) = \mathcal{O}\bigg(|\mathrm{x}-\mathrm{y}| \ \Vert \partial_\mathrm{t}u(y,\cdot)\Vert_{\mathrm{L}^2(0,t)}\bigg),
\end{equation}
for $\mathrm{x}, \mathrm{y}$ such that $|\mathrm{x}-\mathrm{y}| \ll \mathrm{t}$ and $\mathrm{t} \in (0,\mathrm{T}]$ uniformly with respect to $D$. The proof is complete.
\end{proof}
\noindent
Given that \(D\) is a bounded open subset of \(\mathbb{R}^n\) of class \(\mathcal{C}^{1,\alpha}\) with \(\alpha \in (0,1]\), there exists a constant \(c_{D,\alpha} > 0\) such that
\[
|(\mathbf{y} - \mathbf{v}) \cdot \nu_{\mathbf{v}}| \leq c_{D,\alpha} |\mathbf{y} - \mathbf{v}|^{1+\alpha}.
\]
Thus,
\begin{align}\label{err6}
    \Bigg(\int_{\partial D} \Big|\int_D \frac{\langle \mathbf{y} - \mathbf{x} \cdot \nu_\mathbf{x} \rangle}{4\pi |\mathbf{x} - \mathbf{y}|^3} \partial_t u(y, \tau) \, dy\Big|^2\Bigg)^\frac{1}{2} 
    &\lesssim \nonumber\left( \int_{\partial D \times D} \frac{|\langle \mathbf{y} - \mathbf{x} \cdot \nu_\mathbf{x} \rangle|^2}{4\pi |\mathbf{x} - \mathbf{y}|^6} \right)^{1/2} |\partial D|^{1/2} \| \partial_t u(y, \cdot) \|_{L^2(D)} \\
    &\lesssim \nonumber\left( \int_{\partial D \times D} \frac{1}{|\mathbf{x} - \mathbf{y}|^2} \, dx \, dy \right)^{1/2} |\partial D|^{1/2} \| \partial_t u(y, \cdot) \|_{L^2(D)} \\
    &\lesssim \nonumber \delta^{3/2} \left( \int_{\partial B \times B} \frac{1}{|\xi - \eta|^2} \, d\xi \, d\eta \right)^{1/2} |\partial D|^{1/2} \| \partial_t u(y, \cdot) \|_{L^2(D)} \\
    &\lesssim \delta^{5-h}.
\end{align}
Similarly,
\begin{align}\label{err7}
    \Bigg(\int_{\partial D} \Big|\int_D \frac{\langle \mathbf{y} - \mathbf{x} \cdot \nu_\mathbf{x} \rangle}{4\pi |\mathbf{x} - \mathbf{y}|^3} \left(\varphi(\mathbf{x}, \mathbf{y}, t) - \partial_t u(y, t)\right) \, dy\Big|^2\Bigg)^\frac{1}{2}
    &\lesssim \nonumber\left( \int_{\partial D \times D} \frac{|\langle \mathbf{y} - \mathbf{x} \cdot \nu_\mathbf{x} \rangle|^2}{4\pi |\mathbf{x} - \mathbf{y}|^4} \right)^{1/2} |\partial D|^{1/2} \| \partial^2_t u(y, t) \|_{L^2(D \times (0, t))} \\
    &\lesssim \nonumber |\partial D \times D|^{1/2} |\partial D|^{1/2} \int_0^t \| \partial^2_t u(y, \cdot) \|_{L^2(D)} \\
    &\lesssim \delta^{6-h}.
\end{align}
Thus, using the estimates (\ref{err6}), (\ref{err7}), and substituting them into (\ref{2.27}), we obtain
\begin{align}\label{err3}
    \textbf{err}^{(2)} := \vartheta \int_{\partial D} \partial_\nu \bm{\mathbcal{V}^t}_D \left[\partial_t u\right](x, t) \, d\sigma_x \lesssim \delta^{6-h}.
\end{align}
Therefore, combining (\ref{err1}), (\ref{2}), and (\ref{err3}) in (\ref{impr}), we obtain
\begin{align}\label{normalderi}
    \int_{\partial D} \partial_\nu u &= \frac{\gamma_m}{\gamma_p} \frac{1}{\kappa_m} \frac{D \cdot \Im(\varepsilon_p)}{2\pi} \int_D |\mathbf{E}|^2 \, d\mathbf{y} + \mathcal{O}(\delta^{6-h}).
\end{align}
This estimate will be used in the following section.


\subsubsection{End of the Proof of Theorem \ref{th1}: The Asymptotic Expansions}      

We now present the asymptotic expansion of the solution to (\ref{heat}) as $\delta \to 0$, based on the integral representation 
\begin{align}\label{single-integral}
          u(x,t) &+ \vartheta \frac{1}{\kappa_m}\int_0^t\int_{D}\Phi^{(m)}(x,t;y,\tau)\ \partial_t u(y,\tau)\ dyd\tau 
          + \nonumber\alpha\frac{1}{\kappa_m} \int_0^t\int_{\partial D}\Phi^{(m)}(x,t;y,\tau)\ \partial_\nu u(y,\tau)\ d\sigma_yd\tau 
          \\ &= \frac{\gamma_m}{\gamma_p}\frac{k\cdot\Im(\varepsilon_p)}{2\pi}\ \frac{1}{\kappa_m}\int_0^t\int_{D}\Phi^{(m)}(x,t;y,\tau)|\mathrm{E}|^2(y)\ f(\tau) dyd\tau.
    \end{align}
Observe that for any fixed $(x,t) \in \mathbb{R}^3 \setminus \overline{D} \times (0,T)$, the function $\Phi^{(m)}(x,t;y,\tau)$ is sufficiently smooth with respect to $(y,\tau) \in \partial D \times (0,t)$. Consequently, the application of Taylor's series expansion yields
$$\big| \Phi^{(m)}(x,t;y,\tau) - \Phi^{(m)}(x,t;z,\tau)\big| \lesssim \delta,$$ 
for $z \in D$, $y \in \partial D$, and $x \in \mathbb{R}^3 \setminus \overline{D}.$
\begin{align}\label{mainformula-single}
    u(x,t) = -\alpha\frac{1}{\kappa_m}\int_0^t\Phi^{(m)}(x,t;z,\tau)\Big(\int_{\partial D}\partial_\nu u(y,\tau)d\sigma_y\Big)d\tau + \text{error},
\end{align}
where $\text{error}:= \text{err}^{(1)} + \text{err}^{(2)} + \text{err}^{(3)}$. Using the same notation as in \cite{sini-haibing}, consider a nanoparticle occupying a domain $D = \delta \mathrm{B} + \mathrm{z}$, where $\mathrm{B}$ is centered at the origin and $|\mathrm{B}| \sim 1.$ Moreover, we use the following notation for defining functions $\varphi$ and $\psi$ on $\partial D \times (0,T)$ and $\partial \mathrm{B} \times (0,T)$, respectively:
\begin{equation}
\hat{\varphi}(\eta, \Tilde{\tau}) = \varphi^{\Lambda}(\eta, \Tilde{\tau}) := \varphi(\delta\eta + \mathrm{z}, \delta^2 \Tilde{\tau}), \quad \quad \quad  \check{\psi}(\mathrm{x}, t) = \psi^\vee(\mathrm{x}, t) := \psi\left(\frac{\mathrm{x}-\mathrm{z}}{\delta}, \frac{t}{\delta^2}\right)
\end{equation}
for $(\mathrm{x},t) \in \partial D \times (0,T)$ and $(\eta,\Tilde{\tau}) \in \partial B \times (0,T)$, respectively.
\\
Now, we have
\begin{align}\label{serr1}
    \text{err}^{(1)} &\nonumber:= \Big|\alpha\int_0^t\int_{\partial D}\Big(\Phi^{(m)}(x,t;y,\tau) - \Phi^{(m)}(x,t;z,\tau)\Big)\partial_\nu u(y,\tau)\ d\sigma_yd\tau \Big|
    \\ \nonumber&\lesssim \mathcal{O}\Big(\alpha \delta \Vert 1 \Vert_{L^2(\partial D \times (0,T))}\Vert \partial_\nu u(y,t)\Vert_{L^2(\partial D \times(0,T))}\Big) 
    \\ &\lesssim \mathcal{O}\Big(\alpha \delta \Vert 1 \Vert_{L^2(\partial D \times (0,T))} \int_0^t\underbrace{\Vert \partial_\nu u(\cdot,t)\Vert_{L^2(\partial D)}}_{\delta^{1+\beta-h}}\Big) \lesssim \delta^{4-h}.
\end{align}
Similarly, we rewrite the second term as
\begin{align}\label{serr2}
   \text{err}^{(2)} &\nonumber:= \Big| \vartheta\bm{\mathbcal{V}^t}_D\Big[\partial_t u\Big](x,t)\Big| 
   \\ &\nonumber\lesssim\Big|\int_{D} \frac{1}{4\pi|\mathrm{x}-\mathrm{y}|}\partial_t u(y,\tau) dy\Big| + \Big|\int_{D} \frac{1}{4\pi|\mathrm{x}-\mathrm{y}|}\Big(\varphi(\mathrm{x},\mathrm{y},\mathrm{t}) - \partial_t u(y,t)\Big)dy\Big|
   \\  &\lesssim |D|^\frac{1}{2} \Vert \partial_t u(\cdot,t) \Vert_{L^2(D)} + |D|^\frac{1}{2} \Vert \partial^2_t u(\cdot,t) \Vert_{L^2(D)} \lesssim \delta^{4-h}.
\end{align}
For the third term, we have
\begin{align}\label{serr3}
   \text{err}^{(3)} &\nonumber:= \Big| \frac{\gamma_m}{\gamma_p}\frac{\omega\cdot\Im(\varepsilon_p)}{2\pi}\bm{\mathbcal{V}^t}_D\Big[\mathrm{J}\Big](x,t)\Big| 
   \\ &\nonumber\lesssim\delta^{\beta+h}\Bigg(\Big|\int_{D} \frac{1}{4\pi|\mathrm{x}-\mathrm{y}|}\mathrm{J}(y,\tau) dy\Bigg| + \Bigg|\int_{D} \frac{1}{4\pi|\mathrm{x}-\mathrm{y}|}\Big(\varphi(\mathrm{x},\mathrm{y},\mathrm{t}) - \mathrm{J}(y,t)\Big)dy\Big|\Bigg)
   \\  &\lesssim \delta^{\beta+h}\Bigg(|D|^\frac{1}{2}\ \Vert f(t)\Vert_{L^\infty} \big\Vert \mathrm{E} \big\Vert^2_{L^4(D)} + |D|^\frac{1}{2}\ \Vert \partial_t f(t)\Vert_{L^\infty} \big\Vert \mathrm{E} \big\Vert^2_{L^4(D)}\Bigg)
   \lesssim \delta^{3+\beta-h}.
\end{align}
Therefore, combining (\ref{serr1}), (\ref{serr2}), and (\ref{serr3}), we derive from (\ref{mainformula-single}) that
\begin{align}
    u(x,t) = -\alpha\frac{1}{\kappa_m}\int_0^t\Phi^{(m)}(x,t;z,\tau)\Big(\int_{\partial D}\partial_\nu u(y,\tau)d\sigma_y\Big)d\tau + \mathcal{O}(\delta^{4-h}).
\end{align}
Finally, due to the estimate (\ref{normalderi}), we obtain that
\begin{align}
     \nonumber u(x,t) = -\frac{k\cdot\Im(\varepsilon_p)}{2\pi}\frac{\gamma_m}{\kappa^2_m}\int_0^t\Phi^{(m)}(x,t;z,\tau)f(\tau)d\tau\int_D |\mathrm{E}|^2(y) \ d\mathrm{y} + \mathcal{O}(\delta^{4-h}).
\end{align}


\subsection{Proof of Theorem \ref{th1}: the Multiple Inclusion Case} 

Our aim in this section is to provide the asymptotic expansion of the solution to (\ref{heat}) as $\delta\to 0$ for the case when we distribute the multiple nano-particles in the region of interest. Let us now express the solution to (\ref{heat}) as the following integral equation for $i=1,2,\ldots,M$
\begin{align}
    u^{(i)}(x,t) + \sum_{i=1}^M \vartheta_i\ \bm{\mathbcal{V}^t}_{D_i}\big[\partial_t u\big](\mathrm{x},\mathrm{t}) + \sum_{i=1}^M\alpha_i\ \bm{\mathcal{S}^t}_{\partial D_i}\big[\partial_\nu u^{(i)}\big](\mathrm{x},\mathrm{t}) = \sum_{i=1}^M \frac{\gamma_m}{\gamma_{p_i}} \bm{\mathbcal{V}^t}_{D_i}\big[\mathrm{J}\big](\mathrm{x},\mathrm{t}),\; (\mathrm{x},\mathrm{t}) \in D_i \times(0,\mathrm{T}).
\end{align}    
Then, we have
\begin{align}
    u^{(i)}(x,t)\nonumber &+ \vartheta_i\ \bm{\mathbcal{V}^{(ii)}}\big[\partial_t u^{(i)}\big](\mathrm{x},\mathrm{t}) + \alpha_i\ \bm{\mathcal{S}^{(ii)}}\big[\partial_\nu u^{(i)}\big](\mathrm{x},\mathrm{t}) + \sum\limits_{\substack{j=1 \\ j\neq i}}^M \vartheta_j\ \bm{\mathbcal{V}^{(ij)}}\big[\partial_t u^{(j)}\big](\mathrm{x},\mathrm{t}) \\ &+ \sum\limits_{\substack{j=1 \\ j\neq i}}^M\alpha_j\ \bm{\mathcal{S}^{(ij)}}\big[\partial_\nu u^{(j)}\big](\mathrm{x},\mathrm{t}) = \frac{\gamma_m}{\gamma_{p_i}}\bm{\mathbcal{V}^t}_{D_i}\big[\mathrm{J}\big](\mathrm{x},\mathrm{t})  + \sum\limits_{\substack{j=1 \\ j\neq i}}^M \frac{\gamma_m}{\gamma_{p_j}}\bm{\mathbcal{V}^{(ij)}}\big[\mathrm{J}\big](\mathrm{x},\mathrm{t}),\; (\mathrm{x},\mathrm{t}) \in D_i \times(0,\mathrm{T}).
\end{align}    
Then, using the jump relation we arrive at the following boundary integral representation
\begin{align}
&\nonumber(1+\frac{\alpha_i}{2})\partial_\nu u^{(i)}(x,t) + \vartheta_i\ \partial_\nu\bm{\mathbcal{V}^{(ii)}}\big[\partial_t u^{(i)}\big](\mathrm{x},\mathrm{t}) + \alpha_i\ \bm{\mathcal{K}^{(ii)}}\big[\partial_\nu u^{(i)}\big](\mathrm{x},\mathrm{t}) + \sum\limits_{\substack{j=1 \\ j\neq i}}^M \vartheta_j\ \partial_\nu\bm{\mathbcal{V}^{(ij)}}\big[\partial_t u^{(j)}\big](\mathrm{x},\mathrm{t}) \\ &+ \sum\limits_{\substack{j=1 \\ j\neq i}}^M\alpha_j\ \partial_\nu\bm{\mathcal{S}^{(ij)}}\big[\partial_\nu u^{(j)}\big](\mathrm{x},\mathrm{t}) = \frac{\gamma_m}{\gamma_{p_i}}\ \partial_\nu \bm{\mathbcal{V}^{(ii)}}\big[\mathrm{J}\big](\mathrm{x},\mathrm{t}) + \sum\limits_{\substack{j=1 \\ j\neq i}}^M \frac{\gamma_m}{\gamma_{p_j}}\ \partial_\nu \bm{\mathbcal{V}^{(ij)}}\big[\mathrm{J}\big](\mathrm{x},\mathrm{t}),\; (\mathrm{x},\mathrm{t}) \in \partial D_i \times(0,\mathrm{T}).
\end{align}    
Let us now take integration with respect to $\partial D_i$ to obtain
    \begin{align}\label{imprexpression}
         \nonumber&(1+\frac{\alpha_i}{2})\int_{\partial D_i}\partial_\nu u^{(i)}(x,t) 
         + \underbrace{\vartheta_i\int_{\partial D_i}\partial_\nu\bm{\mathbcal{V}^{(ii)}}_D\big[\partial_t u^{(i)}\big](\mathrm{x},\mathrm{t})}_{:= \text{err}_n^{(1)}} 
         + \underbrace{\alpha_i\int_{\partial D_i}\bm{\mathcal{K}^{(ii)}}_{\partial D}\big[\partial_\nu u^{(i)}\big](\mathrm{x},\mathrm{t})}_{:= \text{err}_n^{(2)}} 
         \\ \nonumber&+ \underbrace{\sum\limits_{\substack{j=1 \\ j\neq i}}^M \vartheta_j\int_{\partial D_i}\partial_\nu\bm{\mathbcal{V}^{(ij)}}\big[\partial_t u^{(j)}\big](\mathrm{x},\mathrm{t})}_{:= \text{err}_n^{(3)}} 
         + \underbrace{\sum\limits_{\substack{j=1 \\ j\neq i}}^M \alpha_j\int_{\partial D_i}\partial_\nu\bm{\mathcal{S}^{(ij)}}\big[\partial_\nu u^{(j)}\big](\mathrm{x},\mathrm{t})}_{:= \textbf{Term}_n^{(1)}} 
         = \underbrace{\int_{\partial D_i}\frac{\gamma_m}{\gamma_{p_i}}\partial_\nu \bm{\mathbcal{V}^{(ii)}}\big[\mathrm{J}\big](\mathrm{x},\mathrm{t})}_{:= \textbf{Term}_n^{(2)}} 
         \\ &+ \underbrace{\sum\limits_{\substack{j=1 \\ j\neq i}}^M \int_{\partial D_i}\frac{\gamma_m}{\gamma_{p_j}}\partial_\nu \bm{\mathbcal{V}^{(ij)}}\big[\mathrm{J}\big](\mathrm{x},\mathrm{t})}_{:= \textbf{err}_n^{(4)}}.
    \end{align}  
Proceeding similarly as we derived the estimates (\ref{err2}), we can show that
\begin{align}\label{errn1}
    \text{err}_n^{(1)} \lesssim \delta^{6-h}.
\end{align}
Again, as discussed in the previous section and due to the estimate (\ref{2}), we have
\begin{align}\label{errn2}
    \text{err}_n^{(2)} = -\frac{\alpha_i}{2}\int_{\partial D_i}\partial_\nu u^{(i)} + \mathcal{O}(\delta^{6-h}).
\end{align}
We start with the following singularity properties related to the heat fundamental solution.
\begin{lemma}\cite[Lemma 4.2]{sini-haibing}
    For $x\ne y$ and $|x-y|\to 0,$ we have
    \begin{equation}
        \Big(\int_0^t|\partial_{\nu_x}\Phi(x,t;y,\tau)|^2d\tau\Big)^\frac{1}{2} = \mathcal{O}(|x-y|^{-3}),\;\text{and} \; \Big(\int_0^t|\nabla_x\Phi(x,t;y,\tau)|^2d\tau\Big)^\frac{1}{2} = \mathcal{O}(|x-y|^{-3}),\; t \in (0,T).
    \end{equation}
\end{lemma}
\noindent
Let us now proceed to estimate the term $\displaystyle\sum\limits_{\substack{j=1 \\ j\neq i}}^M \vartheta_j\int_{\partial D_i}\partial_\nu\bm{\mathbcal{V}^{(ij)}}\big[\partial_t u^{(j)}\big](\mathrm{x},\mathrm{t})$. In order to do that, we use divergence theorem and fact that Newtonian potential satisfies the corresponding heat operator to obtain
    \begin{align} \label{errn3}
          \text{err}_n^{(3)}
          \nonumber&:= \Bigg|\sum\limits_{\substack{j=1 \\ j\neq i}}^M \vartheta_j\int_{\partial D_i}\partial_\nu\bm{\mathbcal{V}^{(ij)}}\big[\partial_t u^{(j)}\big](\mathrm{x},\mathrm{t})d\sigma_x\Bigg|
          \\ \nonumber&\lesssim \sum\limits_{\substack{j=1 \\ j\neq i}}^M \vartheta_j\Vert 1\Vert_{L^2(\partial D_i)}\Vert\partial_\nu\bm{\mathbcal{V}^{(ij)}}\big[\partial_t u^{(j)}\big]\Vert_{L^2(\partial D_i)}
          \\ \nonumber& = \sum\limits_{\substack{j=1 \\ j\neq i}}^M \delta\vartheta_j\Bigg(\int_{\partial D_i}\Big(\int_0^t\int_{D_j}\partial_{\nu_x}\Phi(x,t;y,\tau)\ \partial_t u^{(j)}(\mathrm{y},\tau)dyd\tau\Big)^2d\sigma_x\Bigg)^\frac{1}{2}
          \\ \nonumber& \lesssim\Bigg|\sum\limits_{\substack{j=1 \\ j\neq i}}^M \frac{\vartheta_j}{\kappa_m}\delta \int_0^T \Vert \partial_t u^{(j)}(\cdot,t)\Vert_{L^2(D_j)}\Bigg(\int_{\partial D_i}\int_{ D_j} \Big(\int_0^t|\partial_{\nu_x}\Phi(x,t;y,\tau)|^2d\tau\Big)d\sigma_x\Bigg)^\frac{1}{2} 
          \\ \nonumber &\lesssim \delta\sum\limits_{\substack{j=1 \\ j\neq i}}^M \frac{\vartheta_j}{\kappa_m}\Vert \partial_t u^{(j)}(\cdot,t)\Vert_{L^2(D_j)} |\partial D_i\times D_j|^\frac{1}{2} \Big(\int_0^t|\partial_{\nu_x}\Phi(x,t;y,\tau)|^2d\tau\Big)^\frac{1}{2}
          \\  &\lesssim \delta^{4+\beta-h}\sum\limits_{\substack{j=1 \\ j\neq i}}^M d_{ij}^{-3}.
   \end{align}
Next, we use divergence theorem and the fact that the single layer potential $\mathcal{S}w$ satisfies the homogeneous heat equation. That is, $[\alpha\partial_t - \Delta_x](\mathcal{S}w)(x, t) = 0\ \forall (x, t) \in \mathbb R^3 \setminus \partial D\times(0,T)$. Then, we have 
\begin{align}\label{interactionterm}
    \textbf{Term}_n^{(1)} := \sum\limits_{\substack{j=1 \\ j\neq i}}^M \alpha_j\int_{\partial D_i}\partial_\nu\bm{\mathcal{S}^{(ij)}}\big[\partial_\nu u^{(j)}\big](\mathrm{x},\mathrm{t})\ dx
    &\nonumber= \sum\limits_{\substack{j=1 \\ j\neq i}}^M \alpha_j\int_{D_i}\Delta\bm{\mathcal{S}^{(ij)}}\big[\partial_\nu u^{(j)}\big](\mathrm{x},\mathrm{t})\ dx
    \\&= \sum\limits_{\substack{j=1 \\ j\neq i}}^M \alpha_j\int_{D_i} \bm{\mathcal{S}^{(ij)}}\big[\partial_t\partial_\nu u^{(j)}\big](\mathrm{x},\mathrm{t})\ dx.
\end{align}
Let us now define with $\displaystyle\sigma^{(j)} = \int_{\partial D_j}\partial_\nu u^{(j)}$ the following
\begin{align}\label{interactionterm1}
    \bm{\mathcal{S}^{(ij)}}\big[\partial_t\sigma^{(j)}\big](\mathrm{x},\mathrm{t}) = \frac{1}{\kappa_m}\int_{0}^t \Phi(z_i,t;z_j,\tau)\ \partial_t\sigma^{(j)}(\tau) d\tau + \mathcal{A}^{(ij)} + \mathcal{B}^{(ij)}\; \text{with}\ (x,t) \in \partial D_i \times(0,T),
\end{align}
where, $$\bm{\mathcal{A}^{(ij)}}\big[\sigma^{(j)}\big](\mathrm{x},\mathrm{t}) := \frac{1}{\kappa_m}\int_{0}^t \big(\Phi(x,t;z_j,\tau) - \Phi(z_i,t;z_j,\tau)\big)\ \partial_t\sigma^{(j)}(\tau) d\tau $$
and
$$\bm{\mathcal{B}^{(ij)}}\big[\sigma^{(j)}\big](\mathrm{x},\mathrm{t}) := \frac{1}{\kappa_m}\int_{0}^t \big(\Phi(x,t;y,\tau) - \Phi(x,t;z_j,\tau)\big)\ \partial_t\sigma^{(j)}(\tau) d\tau.$$
We now proceed to estimate term $\bm{\mathcal{A}^{(ij)}}\big[\sigma^{(j)}\big](\mathrm{x},\mathrm{t}).$ We write
    \begin{align}\nonumber
          \bm{\mathcal{A}^{(ij)}}\big[\sigma^{(j)}\big](\mathrm{x},\mathrm{t}) 
          &\nonumber= \underbrace{\int_{0}^t \big(x-z_i\big)\cdot \nabla_x \Phi(z_i,t;z_j,\tau)\ \partial_t\sigma^{(j)}(\tau) d\tau}_{:= \bm{\mathcal{A}^{(ij)}}_1 } 
          \\ &\nonumber+ \underbrace{\frac{1}{2} \sum\limits_{k,l=1}^3\int_{0}^t \big(x_k-z_{ik}\big)\big(x_l-z_{il}\big)\cdot \partial_{x_k}\partial_{x_l} \Phi(z^*_i,t;z_j,\tau)\ \partial_t\sigma^{(j)}(\tau) d\tau}_{:= \bm{\mathcal{A}^{(ij)}}_2 },
    \end{align}
where $z^*_i = z_i + \theta(x-z_i), 0<\theta<1.$ \\
Therefore, plugging the previous expression (\ref{interactionterm1}) in (\ref{interactionterm}), we derive that
    \begin{align}\label{important}
         \textbf{Term}_n^{(1)} := \nonumber&\sum\limits_{\substack{j=1 \\ j\neq i}}^M \alpha_j\int_{\partial D_i}\partial_\nu\bm{\mathcal{S}^{(ij)}}\big[\partial_\nu u^{(j)}\big](\mathrm{x},\mathrm{t}) 
        \\ \nonumber&=   \sum\limits_{\substack{j=1 \\ j\neq i}}^M \frac{\alpha_j}{\kappa_m} \text{vol}(D_i)\int_{0}^t \Phi(z_i,t;z_j,\tau)\ \partial_t\sigma^{(j)}(\tau) d\tau + \underbrace{\sum\limits_{\substack{j=1 \\ j\neq i}}^M \frac{\alpha_j}{\kappa_m}\int_{D_i}\mathcal{A}_1^{(ij)}}_{:= \textbf{err}_n^{(5)}} + \underbrace{\sum\limits_{\substack{j=1 \\ j\neq i}}^M \frac{\alpha_j}{\kappa_m}\int_{D_i}\mathcal{A}_2^{(ij)}}_{:= \textbf{err}_n^{(6)}}
        \\ &+ \underbrace{\sum\limits_{\substack{j=1 \\ j\neq i}}^M \frac{\alpha_j}{\kappa_m}\int_{D_i}\mathcal{B}^{(ij)}}_{:= \textbf{err}_n^{(7)}}.
     \end{align}
Keeping in mind the definition of $\sigma^{(j)}$, we start with estimating the term $\textbf{err}_n^{(4)}$.
\begin{align}
    \textbf{err}_n^{(5)} 
    \nonumber&:= \Big|\sum\limits_{\substack{j=1 \\ j\neq i}}^M \frac{\alpha_j}{\kappa_m}\int_{D_i}\mathcal{A}_1^{(ij)}\Big| 
    \\ \nonumber&\lesssim |D_i|^\frac{1}{2}\Vert \partial_t\partial_\nu u^{(j)}(\cdot,t)\Vert_{L^2(\partial D_j)}\Bigg(\int_{D_i}\int_{\partial D_j}\Big(\int_0^t |x-z_i|^2|\nabla_x \Phi(z_i,t;z_j,\tau)|^2 d\tau\Big)  d\sigma_x d\sigma_y\Bigg)^\frac{1}{2}
    \\ \nonumber&\lesssim \sum\limits_{\substack{j=1 \\ j\neq i}}^M \frac{\alpha_j}{\kappa_m}|D_i|^\frac{1}{2}\Vert \partial_t\partial_\nu u^{(j)}(\cdot,t)\Vert_{L^2(\partial D_j)}\ \delta\ |D_i\times \partial D_j|^\frac{1}{2}\ d_{ij}^{-3}
\end{align}
Consequently, we obtain that 
\begin{align}\label{errn5}
    \textbf{err}_n^{(5)} \lesssim \delta^{6-h}\sum\limits_{\substack{j=1 \\ j\neq i}}^M d_{ij}^{-3}.
\end{align}
In a similar way, we can show that
\begin{align}\label{errn6}
    \textbf{err}_n^{(6)} \lesssim \delta^{8-h}\sum\limits_{\substack{j=1 \\ j\neq i}}^M d_{ij}^{-4}\; \text{and}\; \textbf{err}_n^{(7)} \lesssim \delta^{6-h}\sum\limits_{\substack{j=1 \\ j\neq i}}^M d_{ij}^{-3}.
\end{align}
Therefore, with these estimates (\ref{errn5}), (\ref{errn6}), and plugging these into (\ref{important}), we deduce that
    \begin{align}\label{termn1}
        \sum\limits_{\substack{j=1 \\ j\neq i}}^M \frac{\alpha_j}{\kappa_m}\int_{\partial D_i}\partial_\nu\bm{\mathcal{S}^{(ij)}}\big[\partial_\nu u^{(j)}\big](\mathrm{x},\mathrm{t}) 
        =   \sum\limits_{\substack{j=1 \\ j\neq i}}^M \frac{\alpha_j}{\kappa_m} \text{vol}(D_i)\int_{0}^t \Phi(z_i,t;z_j,\tau)\ \partial_t\sigma^{(j)}(\tau) d\tau + \mathcal{O}\Big(\delta^{6-h}\sum\limits_{\substack{j=1 \\ j\neq i}}^M d_{ij}^{-3}\Big).
    \end{align}
\noindent
Following a similar derivation for the first term on the right-hand side of the expression in (\ref{important}) and using the estimate in (\ref{err1}), we deduce that
\begin{align}\label{termn2}
    \textbf{Term}_n^{(2)}
    \nonumber&:= \int_{\partial D_i}\frac{\gamma_m}{\gamma_{p_i}}\partial_\nu \bm{\mathbcal{V}^{(ii)}}\big[\mathrm{J}\big](\mathrm{x},\mathrm{t})
    \\ &= \frac{\gamma_m}{\gamma_{p_i}}\frac{1}{\kappa_m}\frac{k\cdot\Im(\varepsilon_p)}{2\pi}\ f(t) \int_{D_i} |\mathrm{E}|^2 \ d\mathrm{y} + \mathcal{O}(\delta^{\frac{13}{2}-h}).
\end{align}
Next, we turn our attention to estimate the following term
\begin{align}\label{errn4}
   \textbf{err}_n^{(4)} \nonumber&:= \Bigg|\sum\limits_{\substack{j=1 \\ j\neq i}}^M \int_{\partial D_i}\frac{\gamma_m}{\gamma_{p_j}}\partial_\nu \bm{\mathbcal{V}^{(ij)}}\big[\mathrm{J}\big](\mathrm{x},\mathrm{t})d\sigma_x\Bigg|
   \\ \nonumber& = \Bigg|\sum\limits_{\substack{j=1 \\ j\neq i}}^M \frac{\gamma_m}{\gamma_{p_j}}\int_{\partial D_i}\Big(\int_{D_j}\partial_{\nu_x}\Phi(x,t;y,\tau)\ J(\mathrm{y},\tau)dyd\tau\Big)d\sigma_x\Bigg|
     \\ \nonumber& \lesssim\Bigg|\sum\limits_{\substack{j=1 \\ j\neq i}}^M \frac{\gamma_m}{\gamma_{p_j}}|\partial D_i|^\frac{1}{2}\Vert f\Vert_{L^\infty}\Vert J\Vert_{L^2(D_j)}\Bigg(\int_{\partial D_i}\int_{ D_j} \Big(\int_0^t|\partial_{\nu_x}\Phi(x,t;y,\tau)|^2d\tau\Big)\ dyd\sigma_x\Bigg)^\frac{1}{2} 
     \\ \nonumber &\lesssim \sum\limits_{\substack{j=1 \\ j\neq i}}^M \frac{\gamma_m}{\gamma_{p_j}}|\partial D_i|^\frac{1}{2}\Vert f\Vert_{L^\infty}\Vert J\Vert_{L^2(D_j)} |\partial D_i\times D_j|^\frac{1}{2} \Big(\int_0^t|\partial_{\nu_x}\Phi(x,t;y,\tau)|^2d\tau\Big)^\frac{1}{2}
     \\ &\lesssim \delta^{5+\beta-h}\sum\limits_{\substack{j=1 \\ j\neq i}}^M d_{ij}^{-3}.
\end{align}
Finally, combining the estimates (\ref{errn1}), (\ref{errn1}), (\ref{errn1}), (\ref{termn1}), (\ref{termn2}), (\ref{errn4}), and plugging those into the expression we obtain that $\displaystyle\sigma^{(i)} = \int_{\partial D_i}\partial_\nu u^{(i)}$ for $i=1,2,\ldots,M$ satisfies the system of equations
    \begin{align}\label{time-algebraic}
          \sigma^{(i)} + \sum\limits_{\substack{j=1 \\ j\neq i}}^M \frac{\alpha_j}{\kappa_m} \text{vol}(D_i)\int_{0}^t \Phi(z_i,t;z_j,\tau)\ \frac{\partial}{\partial\tau}\sigma^{(j)}(\tau) d\tau 
          =  \frac{\gamma_m}{\gamma_{p_i}}\frac{1}{\kappa_m}\frac{k\cdot\Im(\varepsilon_p)}{2\pi}\ f(t) \int_{D_i} |\mathrm{E}|^2 \ d\mathrm{y} + \mathcal{O}\Big(\delta^{4+\beta-h}\sum\limits_{\substack{j=1 \\ j\neq i}}^M d_{ij}^{-3}\Big). 
    \end{align}
In the following section, we prove the unique solvability of the above algebraic system and establish an estimate for \(\sum\limits_{i=1}^M \Vert \sigma^{(i)}\Vert^2_{L^2(0,T)}\).


\subsubsection{\texorpdfstring{Unique Solvability of the Discrete Algebraic System in $H^1(0,T)$}{Unique Solvability of the Discrete Algebraic System}}  

Let us consider the algebraic system (\ref{time-algebraic})
   \begin{align}\label{algebraic}
         \sigma^{(i)} + \sum\limits_{\substack{j=1 \\ j\neq i}}^M b_j\int_{0}^t \Phi(z_i,t;z_j,\tau) \frac{\partial}{\partial \tau} \sigma^{(j)}(\tau)\ d\tau = \mathcal F_i(t).
   \end{align}
Since \(\sigma^{(i)}(\cdot, 0) = 0\) and using the smoothness of the fundamental solution \(\Phi(z_i, t; z_j, \tau)\), we apply integration by parts to rewrite the system as
   \begin{align}\label{algebraic1}
         \sigma^{(i)} + \sum\limits_{\substack{j=1 \\ j\neq i}}^M b_j \int_{0}^t \partial_t \Phi(z_i, t; z_j, \tau) \sigma^{(j)}(\tau)\ d\tau = \mathcal F_i(t).
   \end{align}
which we rewrite in operator form as
\begin{align}
    \mathbcal{A} \bm{\sigma}(x,t) = \mathbcal{F}(x,t),
\end{align}
where we aim to invert the operator \(\mathbcal{A}:= \bm{\mathrm{I}} + b_j \mathbcal{K}\) in \(L^2(0,T)\). Here, 
\(\bm{\sigma} := \big( \sigma^{(1)}, \sigma^{(2)}, \ldots, \sigma^{(M)}\big)^t\), \(\mathbcal{F} := \big( f_1, f_2, \ldots, f_M\big)^t\) and the operator \(\mathbcal{K}\) is defined as
\begin{equation}
    \mathbcal{K} \bm{\sigma}(t) := \int_0^t
    \begin{pmatrix}
        0 & \partial_t\Phi(z_1, t; z_2, \tau) &  \dots & \partial_t\Phi(z_1, t; z_M, \tau) \\
        \partial_t\Phi(z_2, t; z_1, \tau) & 0 &  \dots & \partial_t\Phi(z_2, t; z_M, \tau) \\
        \vdots & \ddots & \dots & \vdots\\
        \partial_t\Phi(z_M, t; z_1, \tau) &  \partial_t\Phi(z_M, t; z_2, \tau) & \dots & 0
    \end{pmatrix} \bm{\sigma}(\tau)\, d\tau.
\end{equation}
It is observed that the fundamental solution \(\Phi(z_i, t; z_j, \tau)\) is smooth for \(z_i \neq z_j\), where \(i \neq j\) and \(i, j = 1, 2, \ldots, M\). Thus, the matrix-valued kernel
\[
\mathbb{K}(t-\tau; z_i-z_j) := \begin{pmatrix}
        0 & \partial_t\Phi(z_1, t; z_2, \tau) &  \dots & \partial_t\Phi(z_1, t; z_M, \tau) \\
        \partial_t\Phi(z_2, t; z_1, \tau) & 0 &  \dots & \partial_t\Phi(z_2, t; z_M, \tau) \\
        \vdots & \ddots & \dots & \vdots\\
        \partial_t\Phi(z_M, t; z_1, \tau) &  \partial_t\Phi(z_M, t; z_2, \tau) & \dots & 0
    \end{pmatrix}
\]
belongs to the \(L^2(0,T)\) function space, implying that \(\mathbcal{K}\) is a Hilbert-Schmidt operator, and therefore compact in \(L^2(0,T)\). Compact operators have a discrete spectrum, which implies that the spectrum of \( \bm{I} + b_j \mathbcal{K} \) does not include zero for sufficiently small \( b_j \), ensuring the invertibility of the operator. 
Thus, for small \( b_j \), the operator \( \bm{I} + b_j\ \mathbcal{K} \) is indeed invertible. However, such a property is not enough. We need to estimate $\Big(\sigma^{(i)}\Big)_{i=1}^M$ in terms of $\mathbcal{F}$ and $\delta.$ For this, we state the following lemma.

\begin{lemma}          
Under the condition
    \begin{align}\label{condition}
        b\ \max\limits_{1\le i\le M}\sum\limits_{j\ne i} d_{ij}^{-2} < 1,
    \end{align}
where \(b = \max\limits_{1\le j\le M} b_j\), with \(b_j := \frac{\overline{\alpha}_j}{\kappa_m} \text{vol}(B_j)\delta^{3-\beta}\), the algebraic system (\ref{time-algebraic}) is uniquely solvable in \(H^1(0,T)\). Moreover, the following estimate holds
    \begin{align}\label{goodesti}
        \Big(\sum\limits_{i=1}^M \Vert \sigma^{(i)}\Vert^2_{H^1(0,T)}\Big)^\frac{1}{2} \le \Big(1- b\ \max\limits_{1\le i\le M}\sum\limits_{j\ne i} d_{ij}^{-2}\Big)^{-1} \Big(\sum\limits_{i=1}^M \Vert f_i(t)\Vert^2_{H^1(0,T)}\Big)^\frac{1}{2},
    \end{align}
where \(\displaystyle \mathcal F_i(t) := \frac{\gamma_m}{\gamma_{p_i}}\frac{1}{\kappa_m}\frac{k\cdot\Im(\varepsilon_p)}{2\pi} f(t) \int_{D_i} |\mathrm{E}|^2 \ d\mathrm{y}\), a function that is constant with respect to the spatial variable \(x\).
\end{lemma}      

\begin{proof}     
Consider the algebraic system (\ref{time-algebraic})
   \begin{align}
         \sigma^{(i)} + \sum\limits_{\substack{j=1 \\ j\neq i}}^M b_j \int_{0}^t \Phi(z_i,t;z_j,\tau) \frac{\partial}{\partial \tau} \sigma^{(j)}(\tau)\ d\tau = \mathcal F_i(t).
   \end{align}
We observe that, since \(\sigma^{(i)}(0) = 0\) and $f \in C^{k-1}(\mathbb{R})$, we transform the integral equation as follows
    \begin{align}
        \int_0^t \Big(\partial_\tau\sigma^{(i)}(\tau) + \sum\limits_{j\ne i}b_j \Phi(z_i,t;z_j,\tau) \partial_\tau\sigma^{(j)}(\tau)\Big)\,d\tau = \int_0^t \partial_\tau\mathcal{F}_i(\tau)\,d\tau,
    \end{align}
which further implies with the assumptions that $\partial_t\sigma^{(i)}(t) + \sum\limits_{j\ne i}b_j \Phi(z_i,t;z_j,\tau) \partial_t\sigma^{(j)}(t) -\partial_t\mathcal{F}_i(t) \ge 0$ 
    \begin{align}\nonumber
          \partial_t\sigma^{(i)}(t) + \sum\limits_{j\ne i}b_j \Phi(z_i,t;z_j,\tau) \partial_t\sigma^{(j)}(t) = \partial_t\mathcal{F}_i(t),\; \text{a.e for}\ t \in (0,T).
    \end{align}
We now set $q_i(t) := \partial_t\sigma^{(i)}(t)$ and $\mathcal{F}(t) := \partial_t\mathcal{F}_i(t)$ Therefore, we rewrite the above equations as follows
    \begin{align}
          q_i(t) + \sum\limits_{j\ne i}b_j \Phi(z_i,t;z_j,\tau) q_j(t) = \mathcal{F}_i(t).
    \end{align}
Multiplying both sides by \(q_i\), integrating over \([0,T]\), and summing over \(i = 1, \dots, M\), we obtain
   \begin{align}\label{inverse}
         \sum\limits_{i=1}^M \int_0^T q^2_i(t)dt + \sum\limits_{i=1}^M \sum\limits_{\substack{j=1 \\ j\neq i}}^M b_j \int_0^T \Phi(z_i,t;z_j,\tau)\ q_j(t)\ q_i(t) dt =  \sum\limits_{i=1}^M \int_0^T \mathcal{F}_i(t)\ q_i(t)\ dt.
   \end{align}
We now, observe that
    \begin{align}
         \int_0^T |\Phi(z_i,t;z_j,\tau)\ q_j(t)\ q_i(t)| dt \le \Big(\int_0^T |\Phi(z_i,t;z_j,\tau)|^2\Big)^\frac{1}{2} \Big(\int_0^T |q_j(t)\ q_i(t)|^2\Big)^\frac{1}{2}
    \end{align}
Now, consider the following proposition.
\begin{proposition}    
For \(x \neq y\) and \( |x - y| \to 0 \), we have
        \begin{align}\nonumber
              \Big(\int_0^T |\Phi(x, t; y, \tau)|^2\ dt\Big)^\frac{1}{2} = \mathcal{O}(|x - y|^{-2}), \quad t \in (0,T).
       \end{align}
\end{proposition}
\noindent
Then, based on the above Proposition and from (\ref{inverse}), we deduce
    \begin{align}
          \sum\limits_{i=1}^M \Vert q_i\Vert^2_{L^2(0,T)} - b \sum\limits_{i=1}^M \sum\limits_{\substack{j=1 \\ j\neq i}}^M d_{ij}^{-2} \Vert q_i\Vert^2_{L^2(0,T)}  \le \Big(\sum\limits_{i=1}^M \Vert q_i\Vert^2_{L^2(0,T)}\Big)^\frac{1}{2} \Big(\sum\limits_{i=1}^M \Vert \mathcal{F}_i(t)\Vert^2_{L^2(0,T)}\Big)^\frac{1}{2}, \nonumber
    \end{align}
which further implies that
    \begin{align}
         \Big(1 - b\ \max\limits_{1 \le i \le M} \sum\limits_{j \neq i} d_{ij}^{-2}\Big) \sum\limits_{i=1}^M \Vert q_i\Vert^2_{L^2(0,T)} \le \Big(\sum\limits_{i=1}^M \Vert q_i\Vert^2_{L^2(0,T)}\Big)^\frac{1}{2} \Big(\sum\limits_{i=1}^M \Vert \mathcal{F}_i(t)\Vert^2_{L^2(0,T)}\Big)^\frac{1}{2}.
   \end{align}   
Consequently, we deduce that
    \begin{align}
          \Big(1 - b\ \max\limits_{1 \le i \le M} \sum\limits_{j \neq i} d_{ij}^{-2}\Big) \Big(\sum\limits_{i=1}^M \Vert q_i\Vert^2_{L^2(0,T)}\Big)^\frac{1}{2} \le  \Big(\sum\limits_{i=1}^M \Vert \mathcal{F}_i(t)\Vert^2_{L^2(0,T)}\Big)^\frac{1}{2}.
    \end{align}
Thus, based on the above estimate, the unique solvability of (\ref{algebraic}) and the estimate (\ref{goodesti}) follow from condition (\ref{condition}) and the observation that \( q_i(t) := \partial_t \sigma^{(i)}(t) \). \\
The proof is complete.
\end{proof}   

\subsubsection{End of the proof of Theorem \ref{th1}: The Asymptotic Expansions}
In the following, we derive the asymptotic formula for the solution to (\ref{heat}) in the case of multiple inclusions. The proof is similar to the case of single inclusion, so we skip the details but outline it. we derive for $x\in \mathbb R^3\setminus D,\ z_i \in D_i$ and $t\in(0,T)$ that
\begin{align}
    u(x,t) = -\sum\limits_{i=1}^M\alpha_i\frac{1}{\kappa_m}\int_0^t\Phi^{(m)}(x,t;z_i,\tau)\sigma^{(i)}(\tau)d\tau + \text{error},
\end{align}
where, we set $\displaystyle\sigma^{(i)}:= \Big(\int_{\partial D}\partial_\nu u^{(i)}(y,\tau)d\sigma_y\Big).$ The "error" consists of the following terms and can be estimated with a similar procedure as (\ref{e1}), (\ref{e2}), and (\ref{e3}) as follows:
\begin{align}
\text{err}^{(1)} &\nonumber:= \Big|\sum\limits_{i=1}^M\alpha_i\int_0^t\int_{\partial D}\Big(\Phi^{(m)}(x,t;y,\tau)- \Phi^{(m)}(x,t;z_i,\tau)\Big)\partial_\nu u(y,\tau)\ d\sigma_yd\tau \Big| \lesssim M\delta^{4-h},\\
\text{err}^{(2)} &\nonumber:=\Big|\sum\limits_{i=1}^M \vartheta_i \ \bm{\mathbcal{V}^t}_D\Big[\partial_t u\Big](x,t)\Big| \lesssim M\delta^{2+\beta-h}, \ \text{and}\\
\text{err}^{(3)} &\nonumber:=\Big| \sum\limits_{i=1}^M\frac{1}{\gamma_{p_i}}\bm{\mathbcal{V}^t}_D\Big[\mathrm{J}\Big](x,t)\Big|\lesssim M\delta^{3+\beta-h}.
\end{align}
Therefore, combining above three estimates, we derive that
\begin{align}\label{mainformula}
    u(x,t) = -\sum\limits_{i=1}^M\alpha_i\frac{1}{\kappa_m}\int_0^t\Phi^{(m)}(x,t;z_i,\tau)\sigma_i(\tau)d\tau + \mathcal{O}\big(M\delta^{4-h}\big),
\end{align}
where $\sigma^{(i)}$ satisfies the following linear system of equations
    \begin{align}\nonumber
         \sigma^{(i)} + \sum\limits_{\substack{j=1 \\ j\neq i}}^M \frac{\alpha_j}{\kappa_m} \text{vol}(D_i)\int_{0}^t \Phi(z_i,t;z_j,\tau)\ \frac{\partial}{\partial\tau}\sigma^{(j)}(\tau) d\tau 
         =  \frac{\gamma_m}{\gamma_{p_i}}\frac{1}{\kappa_m}\frac{k\cdot\Im(\varepsilon_p)}{2\pi}\ f(t) \int_{D_i} |\mathrm{E}|^2 \ d\mathrm{y} 
         + \mathcal{O}\Big(\delta^{4+\beta-h}\sum\limits_{\substack{j=1 \\ j\neq i}}^M d_{ij}^{-3}\Big). 
    \end{align}
This completes the proof. \qed


    \section{Proof of Theorem \ref{non-periodic} 
    }
We begin by considering the following non-homogeneous first-order matrix differential equation, along with its initial conditions, while neglecting error terms of higher order:
    \begin{align}\label{matrix-al}
         \sigma^{(i)} + \sum_{\substack{j=1 \\ j \neq i}}^M \frac{\alpha_j}{\kappa_m} \text{vol}(D_i) \int_{0}^t \Phi^{(m)}(z_i, t; z_j, \tau) \frac{\partial}{\partial \tau} \sigma^{(j)}(\tau) \, d\tau = g_i(t), \quad \text{for}\; t \in (0, T),
    \end{align}
where \( \alpha_i := \gamma_{p_i} - \gamma_m \) denotes the contrast between the inner and outer heat coefficients, and \( \kappa_m := \frac{c_m}{\gamma_m} \) is the diffusion constant for the background medium. We define $ \displaystyle g_i(t)= \frac{\gamma_m}{\gamma_{p_i}} \frac{1}{\kappa_m} \frac{k \cdot \Im(\varepsilon_p)}{2 \pi} f(t) \int_{D_i} |\mathrm{E}_i|^2 \, d\mathrm{y}.$ 
\\
Based on the a priori estimates for \( E_i \) from Proposition \ref{proposition-maxwell-1}, along with the properties described in equations (\ref{frequency}) and (\ref{condition3D}), the expression for \( g_i(t) \) can be rewritten in a more tractable form as follows
\[
\int_{D_i} |\mathrm{E}_i|^2 
= \int_{D_i} |\overset{3}{\mathbb{P}}(\mathrm{E}_i)|^2 + \mathcal{O}(\delta^3)
= E_i(z_i) \cdot \left( \sum_{m=1}^{m_{n_0}} \int_{D_i} e^{(3)}_{m,n_0} \otimes \int_{D_i} e^{(3)}_{m,n_0} \right) \cdot E_i^\textit{Tr}(z_i) + \mathcal{O}(\delta^3),
\]
where we define \( \mathbb{K}_{D_i} \) as
\[
\mathbb{K}_{D_i} := \sum_{m=1}^{m_{n_0}} \int_{D_i} e^{(3)}_{m,n_0} \otimes \int_{D_i} e^{(3)}_{m,n_0}.
\]
Next, recalling the definition \(\displaystyle \mathbcal{Q}_i = \int_{D_i} \overset{3}{\mathbb{P}}(\mathrm{E}_i) \), and applying the spectral decomposition, we obtain the relation
\[
\int_{D_i} \overset{3}{\mathbb{P}}(\mathrm{E}_i) = E_i(z_i) \cdot \left( \sum_{m=1}^{m_{n_0}} \int_{D_i} e^{(3)}_{m,n_0} \otimes \int_{D_i} e^{(3)}_{m,n_0} \right) + \mathcal{O}(\delta^3),
\]
which leads to the following expression for \( E_i(z_i) \)
\[
E_i(z_i) = \mathbb{K}_{D_i}^{-1} \cdot \int_{D_i} \overset{3}{\mathbb{P}}(\mathrm{E}_i).
\]
Substituting this into the previous equation and neglecting higher-order terms, we derive
\[
\int_{D_i} |\overset{3}{\mathbb{P}}(\mathrm{E}_i)|^2 = \mathbb{K}_{D_i}^{-1} \mathbcal{Q}_i \cdot \overline{\mathbcal{Q}_i^\textit{Tr}}.
\]
After normalization, let \( e^{(3)}_{m,n_0}(B) := \frac{\Tilde{e}^{(3)}_{m,n_0}}{\| \Tilde{e}^{(3)}_{m,n_0} \|_{L^2(B)}} \), where \( \| \Tilde{e}^{(3)}_{m,n_0} \|_{L^2(B)} = \mathcal{O}(\delta^{-\frac{3}{2}}) \) represents the scaled eigenfunctions associated with the space \( \nabla \mathbb{H}_{\textit{arm}} \) on the domain \( B \). Using the property (\ref{condition3D}) and the definition of the polarization matrix \( \mathbcal{P}_{D_i} \)
\[
\mathbcal{P}_{D_i} = \frac{\delta^3}{1 + \eta \lambda_{n_0}^{(3)}} \sum_{m=1}^{m_{n_0}} \int_{B} e^{(3)}_{m,n_0}(B) \otimes \int_{B} e^{(3)}_{m,n_0}(B) + \mathcal{O}(\delta^3),
\]
we observe that
\[
\mathbb{K}_{D_i} = \delta^h \, \mathbcal{P}_{D_i}.
\]
Thus, from the previous expression, we obtain
\[
\int_{D_i} |\overset{3}{\mathbb{P}}(\mathrm{E}_i)|^2 = \delta^{3-2h} \mathbcal{P}_B\ \mathbcal{Q}_i \cdot \overline{\mathbcal{Q}_i^\textit{Tr}},
\]
where \( \mathbcal{P}_B \) is defined as \(\displaystyle\mathbcal{P}_B := \sum_{m=1}^{m_{n_0}} \int_{B} e^{(3)}_{m,n_0}(B) \otimes \int_{B} e^{(3)}_{m,n_0}(B).\)
\noindent
Thus, we can rewrite the algebraic system in equation (\ref{matrix-al}) in terms of the scaled variable \(\displaystyle \Tilde{\sigma}^{(i)} := \frac{1}{\text{vol}(D_i)} \int_{\partial D_i} \partial_\nu u^{(i)} \) for \( i = 1, 2, \ldots, M \):
\begin{align}\label{matrix}
    \Tilde{\sigma}^{(i)} + \sum_{\substack{j = 1 \\ j \neq i}}^M \frac{\alpha_j}{\kappa_m} \text{vol}(D_j) \int_0^t \Phi^{(m)}(z_i, t; z_j, \tau) \frac{\partial}{\partial \tau} \Tilde{\sigma}^{(j)}(\tau) \, d\tau = \overline{a} \, \delta^{\beta - h} \, f(t) \, \mathbcal{P}_B\ \mathbcal{Q}_i \cdot \overline{\mathbcal{Q}_i^\textit{Tr}},
\end{align}
where we have defined \( \overline{a} := \frac{\gamma_m}{\overline{\gamma}_{p_i}} \frac{k \cdot \overline{\varepsilon}_p}{2 \pi \kappa_m} \), which is independent of \( \delta \). Since \( \mathbcal{P}_{D_i} = \delta^{3 - h} \mathbcal{P}_B \), we rewrite the algebraic system in the scaled domain \( B \), where \( \Tilde{\mathbcal{Q}}_i \) satisfies the following algebraic system as discussed in Lemma \ref{lemma-maxwell}
\begin{align}\label{in-11}
    \Tilde{\mathbcal{Q}}_i + \eta \sum_{j \neq i}^M \bm{\Upsilon}^{(k)}(z_i, z_j) \, \mathbcal{P}_{D_j} \cdot \Tilde{\mathbcal{Q}}_j = \mathrm{E}^\textit{in}(z_i).
\end{align}
We now assume that the shape of the each nanoparticles $D_j$ for $j=1,2,\ldots,M$ is same, which implies that $b_i = b_j$ for $i,j=1,2,\ldots, M.$ Let us define $\overline{b}:= \overline{b}_j$ be the scaled value of $b_j$ i.e. $\overline{b}=\frac{\overline{\alpha}_j}{\kappa_m} \text{vol}(B_j).$
\\
The next critical step involves establishing a comparison between equation (\ref{matrix}) and the following Lippmann-Schwinger equation
\begin{align}\label{time-effective-equation}
    \bm{\mathrm{Y}}(x, t) + \overline{b} \int_0^t \int_{\mathbf{\Omega}} \Phi^{(m)}(x, t; y, \tau) \frac{\partial}{\partial \tau} \bm{\mathrm{Y}}(y, \tau) \, dy \, d\tau = \mathbcal{F}(x, t), \quad \text{for} \; (x, t) \in \mathbf{\Omega} \times (0, T),
\end{align}
where we have defined $ \displaystyle\mathbcal{F}(x, t) := \overline{a} \, f(t) \, \mathbcal{A}_B\cdot\mathbcal{P}_{B}^{-1}\cdot \mathbcal{A}_B \cdot \bm{E}_f\cdot \overline{\bm{E}}^\textit{Tr}_f$. The initial condition for \( \bm{\mathrm{Y}} \) is implicitly zero in this equation. Here, \( \Tilde{\bm{E}}_f \) is the solution to the following effective Lippmann-Schwinger equation
\begin{align}
    \bm{E}_f(x) + \Big( \mathbb{M}_{\bm{\Omega}}^{(\kappa)} - k^2 \mathbb{N}_{\bm{\Omega}}^{(\kappa)} \Big) \big[ \mathbcal{A}_B \cdot \bm{E}_f \big](x) = \mathrm{E}^\textit{in}(x), \quad \text{for} \; x \in \bm{\Omega},
\end{align}
where \( \mathbcal{A}_B \) is the effective polarization matrix corresponding to the algebraic system defined in equation (\ref{in-11}). 
\\
Let us now define
\begin{align}
    \mathbcal{V}(x,t) := \begin{cases}
                       \bm{\mathrm{Y}}(x,t) & \text{if}\; (x,t) \in \mathbf{\Omega}\times(0,T) \\ \displaystyle
                       \mathbcal{F}(\mathrm{x},\mathrm{t}) - \overline{b} \int_0^t\int_{\mathbf{\Omega}}\Phi^{(m)}(x,t;y,\tau)\ \frac{\partial}{\partial\tau}\bm{\mathrm{Y}}(y,\tau)\ dyd\tau & \text{for}\; (x,t) \in \mathbb R^3\setminus\mathbf{\Omega}\times(0,T).
                 \end{cases}
\end{align}
We set $\mathbcal{W}(x,t) := \mathbcal{F}(x,t) - \mathbcal{V}(x,t).$ Therefore, we have the following result.
\begin{lemma}
    If $\bm{\mathrm{Y}}$ is the solution of the Lippmann-Schwinger equation (\ref{time-effective-equation}), then $\mathbcal{W}$ satisfies the following problem
    \begin{align}\label{effective-parabolic}
        \begin{cases}
            (\kappa_m\partial_t - \Delta)\mathbcal{W} + \overline{b}\rchi_{\mathbf{\Omega}}\partial_t\mathbcal{W} = \overline{b}\rchi_{\mathbf{\Omega}}\mathbcal{F}(x,t) & \text{if}\; (x,t) \in \mathbb R^3 \times (0,T) \\
            \mathbcal{W}(x,0) = 0 & \text{for}\; x \in \mathbb R^3, \\
            |\mathbcal{W}(x,t)|  \leq  C_0\ e^{A|x|^2} & \text{as}\ |x| \to +\infty,
        \end{cases}
    \end{align}
for some positive constant $C_0$ and $A \le \frac{1}{4T}.$ The above differential equation has a unique solution in $\mathrm{H}^{r}_{0,\sigma}\big(0,\mathrm{T}; \mathrm{H}^1(\mathbb R^3)\big)$ for $\mathbcal{F} \in \mathrm{H}^{r-\frac{1}{2}}_{0,\sigma}\big(0,\mathrm{T}; \mathrm{L}^2(\mathbb R^3)\big)$. 
\\
We can show this result in a manner similar to the approach discussed in Section \ref{apriori}. The only difference is the presence of a perturbation in the first-order term, which will not affect the overall proof of well-posedness. 
\end{lemma}


\subsection{Solvability and Regularity of the Effective Integral Equation for the Parabolic Problem} 

\begin{lemma}\label{prop1}    
The Lippmann-Schwinger equation
    \begin{equation}
        \bm{\mathrm{Y}} (\mathrm{x},\mathrm{t}) 
        + \overline{b}\int_0^t\int_{\mathbf{\Omega}}\Phi^{(m)}(x,t;y,\tau)\ \frac{\partial}{\partial\tau}\bm{\mathrm{Y}}(y,\tau)\ dyd\tau 
        = \mathbcal{F}(x,t), \; (x,t) \in \mathbf{\Omega} \times (0, \mathrm{T}),
    \end{equation}
has a unique solution in 
    $\mathrm{H}^{r}_{0,\sigma}\big(0,\mathrm{T};\mathrm{L}^2(\mathbf{\Omega})\big)$ for $f \in \mathrm{H}^{r+\frac{1}{2}}_{0,\sigma}\big(0,\mathrm{T};\mathrm{L}^2(\mathbf{\Omega})\big)$ and it satisfies \\ $\Vert \bm{\mathrm{Y}}\Vert_{\mathrm{H}^{r}_{0,\sigma}\big(0,\mathrm{T};\mathrm{L}^2(\mathbf{\Omega})\big)} \lesssim \Vert f\Vert_{\mathrm{H}^{r+\frac{1}{2}}_{0,\sigma}\big(0,\mathrm{T};\mathrm{L}^2(\mathbf{\Omega})\big)}.$
\end{lemma}   

\begin{proof}  
We begin by considering the Lippmann-Schwinger equation in the Fourier-Laplace domain, where the transform parameter is set as \(\bm{s} \in \mathbb{C}_+ := \Big\{\bm{s}\in \mathbb C : \Re(\bm{s}) > 0 \Big\}\). This choice is motivated by the fact that the heat equation, when transformed into the Laplace domain, assumes the form \(\Big(-\Delta + \left(\bm{s}^{\frac{1}{2}}\right)^2\Big) \hat{u} = \hat{\mathbcal{F}}(x, \bm{s})\). Consequently, it is necessary to consider \(\bm{s} \in \mathbb{C} \setminus (-\infty, 0]\), which ensures that \(\bm{s}^{\frac{1}{2}} \in \mathbb{C}_+\) and \(\omega = \Re (\bm{s}^{\frac{1}{2}}) > 0\).
\\
Next, we introduce the Newtonian heat potential operator, which is defined as
\[
    \bm{\mathrm{V}}^t_\mathbf{\Omega}[f](x,t) := \int_0^t \int_{\mathbf{\Omega}} \Phi^{(m)}(x,t;y,\tau) \, f(y,\tau) \, dy \, d\tau.
\]
Here, we note that the Laplace-Fourier transform \(\mathrm{G}^{(\bm{s})}\) of the fundamental solution \(\Phi^{(m)}(x,t;y,\tau)\) for the heat equation corresponds to the fundamental solution of the differential equation \( -\Delta u + \bm{s} u = 0\), expressed as
\[
    \mathrm{G}^{(\bm{s})}(x,y) = \frac{1}{4\pi|x-y|} e^{-\sqrt{\bm{s}} \, |x-y|}.
\]
Based on this, the corresponding Newtonian potential is defined as
\[
    \bm{\mathrm{V}}_{\bm{s},\mathbf{\Omega}}[\hat{f}](x) := \int_{\Omega} \mathrm{G}^{(\bm{s})}(x,y) \, \hat{f}(y) \, dy,
\]
where \(\displaystyle\widehat{u}:= \widehat{u}(x,\mathbf{s})=\int_0^\infty u(x,t)e^{-\mathbf{s}t}dt.\) Furthermore, it is known that for \(\hat{h}(y) \in L^2(\bm{\Omega})\), extended by 0, the Newtonian potential \(z := \bm{\mathrm{V}}_{\bm{s},\mathbf{\Omega}}[\hat{h}]\) satisfies the following equation
    \begin{equation}\label{4.5}
        -\Delta z + \bm{s} z = \hat{h} \quad \text{in} \ \mathbb{R}^3.
    \end{equation}
Thus, in the Laplace-Fourier domain, the time-domain Lippmann-Schwinger equation (\ref{time-effective-equation}) takes the form
    \begin{align}\label{lax}
        \hat{\bm{\mathrm{Y}}} + \overline{b}\ \mathbf{s}\; \bm{\mathrm{V}}_{\bm{s},\mathbf{\Omega}}\big(\hat{\bm{\mathrm{Y}}}\big) = \hat{\mathbcal{F}}(x,\bm{s}) \quad \text{in} \; \mathbf{\Omega}.
    \end{align}
Our goal is to establish the well-posedness of this problem using the approach outlined in \cite{le-monk}. Specifically, we will formulate a variational framework for this problem and demonstrate its well-posedness using the Lax-Milgram Lemma.
\\
To proceed, we multiply equation (\ref{lax}) by \(\widehat{g} \in L^2(\mathbf{\Omega})\) and integrate over \(\mathbf{\Omega}\), obtaining the following variational form for \(\widehat{\bm{\mathrm{Y}}} \in L^2(\mathbf{\Omega})\):
\[
    \mathbb{A}\big(\hat{\bm{\mathrm{Y}}}, \widehat{g}\big) = \mathbb{B}(\widehat{g}) \quad \text{in} \; L^2(\mathbf{\Omega}),
\]
where
\[
    \mathbb{A}\big(\hat{\bm{\mathrm{Y}}}, \widehat{g}\big) := \int_{\mathbf{\Omega}} \hat{\bm{\mathrm{Y}}} + \mathbf{s}\ \overline{b}\ \bm{\mathrm{V}}_{\bm{s},\mathbf{\Omega}}\big(\hat{\bm{\mathrm{Y}}}\big)\Big) \, \overline{\widehat{g}} \, dy,
\]
and
\[
    \mathbb{B}(\widehat{g}) := \int_{\mathbf{\Omega}} \widehat{\mathbcal{F}} \, \overline{\widehat{g}} \, dy.
\]
To prove the coercivity of this variational form, we choose \(\widehat{g} = \bm{s}^\frac{1}{2}\widehat{\bm{\mathrm{Y}}}\), leading to the following:
\[
    \mathbb{A}\big(\hat{\bm{\mathrm{Y}}}, \bm{s}^\frac{1}{2}\widehat{\bm{\mathrm{Y}}}\big) = \overline{\bm{s}}^\frac{1}{2}\int_{\mathbf{\Omega}} |\hat{\bm{\mathrm{Y}}}|^2 \, dy + \mathbf{s}\overline{\bm{s}}^\frac{1}{2}\ \overline{b}  \int_{\mathbf{\Omega}} \bm{\mathrm{V}}_{\bm{s},\mathbf{\Omega}}\big(\hat{\bm{\mathrm{Y}}}\big) \, \overline{\widehat{\bm{\mathrm{Y}}}} \, dy.
\]
Taking the real part of this equation, we have
\[
    \Re\Big(\mathbb{A}\big(\hat{\bm{\mathrm{Y}}}, \bm{s}^\frac{1}{2}\widehat{\bm{\mathrm{Y}}}\big)\Big) = \omega\int_{\mathbf{\Omega}} |\hat{\bm{\mathrm{Y}}}|^2 \, dy + \Re\Big(\mathbf{s}\overline{\bm{s}}^\frac{1}{2}\ \overline{b} \int_{\mathbf{\Omega}} \bm{\mathrm{V}}_{\bm{s},\mathbf{\Omega}}\big(\hat{\bm{\mathrm{Y}}}\big) \, \overline{\widehat{\bm{\mathrm{Y}}}} \, dy\Big).
\]
From this, using equation (\ref{4.5}), we deduce that
\[
    \int_{\mathbb{R}^3} \bm{\mathrm{V}}_{\bm{s},\mathbf{\Omega}}\big(\hat{\bm{\mathrm{Y}}}\big) \, \overline{\widehat{\bm{\mathrm{Y}}}} \, dy = \int_{\mathbb{R}^3} z \, \overline{\Big(-\Delta z + \bm{\mathrm{s}} z\Big)} \, dy = \int_{\mathbb{R}^3} |\nabla z|^2 + \overline{\bm{\mathrm{s}}} |z|^2 \, dy.
\]
Thus, we find that
\[
    \Re\Big(\mathbf{s}\overline{\bm{s}}^\frac{1}{2}\ \overline{b} \int_{\mathbb{R}^3} \bm{\mathrm{V}}_{\bm{s},\mathbf{\Omega}}\big(\hat{\bm{\mathrm{Y}}}\big) \, \overline{\widehat{\bm{\mathrm{Y}}}} \, dy\Big) = \Re\Big(\mathbf{s}\overline{\bm{s}}^\frac{1}{2}\ \overline{b} \int_{\mathbb{R}^3} |\nabla z|^2 \, dy + |\bm{\mathrm{s}}|^2\overline{\bm{s}}^\frac{1}{2}\ \overline{b} \int_{\mathbb{R}^3} |z|^2 \, dy\Big) \geq 0.
\]
Thus, we conclude that
    \begin{align}\label{t1}
          \mathbb{A}\big(\hat{\bm{\mathrm{Y}}}, \bm{s}^\frac{1}{2}\widehat{\bm{\mathrm{Y}}}\big) \geq \omega\Vert \hat{\bm{\mathrm{Y}}} \Vert_{L^2(\mathbf{\Omega})}^2.
    \end{align}
Moreover, we have
    \begin{align} \label{t2}
          |\mathbb{B}(\widehat{\bm{\mathrm{Y}}})| = \Big| \int_{\mathbf{\Omega}} \widehat{\mathbcal{F}} \, \overline{\mathbf{s}^\frac{1}{2} \widehat{\bm{\mathrm{Y}}}} \, dy \Big| = \Big|\overline{\mathbf{s}}^\frac{1}{2}\int_{\mathbf{\Omega}}  \widehat{\mathbcal{F}} \, \hat{\bm{\mathrm{Y}}} \, dx\Big| \leq  |\mathbf{s}|^\frac{1}{2}\Vert \widehat{f}\Vert_{L^2(\mathbf{\Omega})} \Vert \hat{\bm{\mathrm{Y}}} \Vert_{L^2(\mathbf{\Omega})}.
    \end{align}
Thus, combining (\ref{t1}) and (\ref{t2}), it follows that
    \begin{align}
          \Big\Vert \Big(\mathbf{I} + \mathbf{s} \bm{\mathrm{V}}_{\bm{s},\mathbf{\Omega}}\Big)^{-1} \Big\Vert_{L^2(\mathbf{\Omega}) \to L^2(\mathbf{\Omega})} \leq \frac{\bm{s}^\frac{1}{2}}{\omega},
    \end{align}
where, we observe it satisfies the form (\ref{1.1}) and hence it allows us to apply Lemma \ref{lubich}. Therefore, we define the corresponding time-domain operator
\[
    \bm{\mathcal{A}}_\mathbf{\Omega} := \mathbf{I} + \bm{\mathrm{V}}^t_\mathbf{\Omega} \partial_t.
\]
Following the notation of Lubich et al. \cite{lubich, lubich1} and Lemma \ref{lubich}, and utilizing the Laplace-Fourier techniques presented in \cite{Arpan-Sini-SIMA, sini-haibing, le-monk}, we can show that
\[
    \bm{\mathcal{A}}_{\mathbf{\Omega}}^{-1} : H^{r+\frac{1}{2}}_{0,\sigma}(0,T; L^2(\mathbf{\Omega})) \to H^r_{0,\sigma}(0,T; L^2(\mathbf{\Omega}))
\]
is a bounded operator.
\\
Thus, we conclude that equation (\ref{time-effective-equation}) has a unique solution in \(H^r_{0,\sigma}(0,T; L^2(\mathbf{\Omega}))\) for \(f \in H^{r+\frac{1}{2}}_{0,\sigma}(0,T; L^2(\mathbf{\Omega}))\), completing the proof.
\end{proof}           
We now present the following corollary, derived as a direct consequence of the previous lemma. This corollary will play a crucial role in the subsequent sections. Specifically, the corollary asserts the following
\begin{corollary}\label{cor}             
Consider the Lippmann-Schwinger equation (\ref{time-effective-equation}). Then, we have \(\partial_t^2 \bm{\mathrm{Y}} \in \mathrm{L}^\infty\big(0, \mathrm{T}; L^\infty(\mathbf{\Omega})\big)\) and \(\partial_{\mathrm{x}_i} \partial_t \bm{\mathrm{Y}} \in \mathrm{L}^\infty(0, \mathrm{T}; L^p(\mathbf{\Omega})\big)\) for any $p>3.$
\end{corollary}      

\begin{proof}          
We begin by recalling the Lippmann-Schwinger equation (\ref{time-effective-equation})
    \begin{equation}\label{effective-equation}
        \bm{\mathrm{Y}} (\mathrm{x}, \mathrm{t}) + \overline{b} \int_0^t \int_{\mathbf{\Omega}} \Phi^{(m)}(\mathrm{x}, \mathrm{t}; \mathrm{y}, \tau) \ \frac{\partial}{\partial \tau} \bm{\mathrm{Y}}(\mathrm{y}, \tau) \, d\mathrm{y} d\tau = \mathbcal{F}(\mathrm{x}, \mathrm{t}), \quad (\mathrm{x}, \mathrm{t}) \in \mathbf{\Omega} \times (0, \mathrm{T}), \nonumber
    \end{equation}
which can be rewritten as
    \begin{align}
        \bm{\mathrm{Y}} (\mathrm{x}, \mathrm{t}) + \overline{b}\mathbf{\bm{V}}^t_{\bm{\Omega}} \left[ \partial_t \bm{\mathrm{Y}} \right] (\mathrm{x}, \mathrm{t}) = \mathbcal{F}(\mathrm{x}, \mathrm{t}), \quad (\mathrm{x}, \mathrm{t}) \in \mathbf{\Omega} \times (0, \mathrm{T}).
    \end{align}
By applying Lemma \ref{prop1}, we observe that for \(\mathbcal{F} \in \mathrm{H}^{\frac{11}{2}}_{0,\sigma}(0, \mathrm{T}; \mathrm{L}^2(\mathbf{\Omega}))\), it follows that \(\bm{\mathrm{Y}} \in \mathrm{H}^5_{0,\sigma}(0, \mathrm{T}; \mathrm{L}^2(\mathbf{\Omega}))\) (i.e., \(p = 5\)), which further implies \(\partial_t \bm{\mathrm{Y}} \in \mathrm{H}^4_{0,\sigma}(0, \mathrm{T}; \mathrm{L}^2(\mathbf{\Omega}))\).
\\
Next, using Proposition \ref{prop2}, we conclude that \(\mathbf{\bm{V}}^t_{\bm{\Omega}} \left[ \partial_t \bm{\mathrm{Y}} \right] \in \mathrm{H}^{\frac{7}{2}}_{0,\sigma}(0, \mathrm{T}; \mathrm{H}^2(\mathbf{\Omega}))\). By the continuous Sobolev embedding \(H^2(\mathbf{\Omega}) \hookrightarrow L^\infty(\mathbf{\Omega})\), this implies \(\mathbf{\bm{V}}^t_{\bm{\Omega}} \left[ \partial_t \bm{\mathrm{Y}} \right] \in \mathrm{H}^3_{0,\sigma}(0, \mathrm{T}; L^\infty(\mathbf{\Omega}))\).
\\
Since \(\mathbcal{F} \in \mathrm{H}^3_{0,\sigma}(0, \mathrm{T}; L^\infty(\mathbf{\Omega}))\), it follows that \(\bm{\mathrm{Y}} \in \mathrm{H}^{\frac{7}{2}}_{0,\sigma}(0, \mathrm{T}; L^\infty(\mathbf{\Omega}))\). Consequently, we deduce that \(\partial_t^2 \bm{\mathrm{Y}} \in \mathrm{H}^{\frac{3}{2}}_{0,\sigma}(0, \mathrm{T}; L^\infty(\mathbf{\Omega}))\), and by the Sobolev embedding \(H^r(0, T) \xhookrightarrow{} \mathcal{C}(0, T)\) for \(r > \frac{1}{2}\), we obtain \(\partial_t^2 \bm{\mathrm{Y}} \in \mathrm{L}^\infty(0, \mathrm{T}; L^\infty(\mathbf{\Omega}))\).
\\
We now refer to classical singularity estimates of the fundamental solution \(\Phi^{(m)}(\mathrm{x}, \mathrm{t}; \mathrm{y}, \tau)\), as outlined in \cite[Chapter 1]{friedman} and \cite[Chapter 9]{kress}. A critical inequality in this context is \(\mathrm{s}^\mathrm{r} e^{-\mathrm{s}} \leq \mathrm{r}^\mathrm{r} e^{-\mathrm{r}}\), where \(0 < \mathrm{s}, \mathrm{r} < \infty\) and \(\mathrm{s} = \kappa_m |\mathrm{x} - \mathrm{y}|^2 / 4(t - \tau)\). Utilizing this, we derive the following singularity estimates
    \begin{equation}\label{singularities}
        \begin{cases}
            |\Phi(\mathrm{x}, \mathrm{t}; \mathrm{y}, \tau)| \lesssim \frac{\kappa_m^\mathrm{r}}{(t - \tau)^\mathrm{r}} \frac{1}{|\mathrm{x} - \mathrm{y}|^{3 - 2\mathrm{r}}}, & \quad \mathrm{r} < \frac{3}{2}, \\[10pt]
            |\partial_{\mathrm{x}_i} \Phi(\mathrm{x}, \mathrm{t}; \mathrm{y}, \tau)| \lesssim \frac{\kappa_m^\mathrm{r}}{(t - \tau)^\mathrm{r}} \frac{1}{|\mathrm{x} - \mathrm{y}|^{4 - 2\mathrm{r}}}, & \quad \mathrm{r} < \frac{5}{2}, \ i = 1, 2,
        \end{cases}
    \end{equation}
for \(0 \leq \tau \leq t \leq T\) and \(\mathrm{x}, \mathrm{y} \in \mathbb{R}^3\) with \(\mathrm{x} \neq \mathrm{y}\).
\\
We first observe that
    \begin{align}\label{observe}
         \left| \partial_{\mathrm{x}_i}\partial_t \bm{\mathrm{Y}}(\mathrm{x}, \mathrm{t}) \right| 
         \lesssim \int_0^t \int_{\mathbf{\Omega}} |\partial_{\mathrm{x}_i} \Phi^{(m)}(\mathrm{x}, \mathrm{t}; \mathrm{y}, \tau)| \, d\mathrm{y} d\tau \cdot \big\Vert \frac{\partial^2}{\partial \tau^2} \bm{\mathrm{Y}} \big\Vert_{\mathrm{L}^\infty(0, \mathrm{T}; L^\infty(\mathbf{\Omega}))} + |\partial_{\mathrm{x}_i}\partial_t\mathbcal{F}(x,t)| .
    \end{align}
Following the analysis in \cite[Section 5.1]{cao-sini}, we observe that \(\Tilde{\bm{E}}_f\) belongs to \(W^{1,p}(\bm{\Omega})\) for any \(p > 3\). Additionally, we consider \(f(t)\) within the function space \(\mathrm{H}^{\frac{5}{2}}_{0,\sigma}(0, \mathrm{T})\), which ensures that \(\mathbcal{F} \in \mathrm{H}^{\frac{5}{2}}_{0,\sigma}\big(0, \mathrm{T}; W^{1,p}(\bm{\Omega})\big)\) for any \(p > 3\). 
\\
Utilizing the singularity estimate for \(\partial_{\mathrm{x}_i} \Phi^{(m)}(\mathrm{x}, \mathrm{t}; \mathrm{y}, \tau)\) with \(r > \frac{1}{2}\), we derive
\[
\int_0^t \int_{\mathbf{\Omega}} |\partial_{\mathrm{x}_i} \Phi^{(m)}(\mathrm{x}, \mathrm{t}; \mathrm{y}, \tau)| \, d\mathrm{y} \, d\tau = \mathcal{O}(1),
\]
which implies that \(\partial_{\mathrm{x}_i} \partial_t \bm{\mathrm{Y}} \in \mathrm{H}^{\frac{3}{2}}_{0,\sigma}\big(0, \mathrm{T}; L^p(\mathbf{\Omega})\big)\) for any \(p > 3\). Specifically, we conclude that \(\partial_{\mathrm{x}_i} \partial_t \bm{\mathrm{Y}} \in \mathrm{L}^\infty\big(0, \mathrm{T}; L^p(\mathbf{\Omega})\big)\) for any \(p > 3\).
\\
This completes the proof.
\end{proof}
To complete the proof of Theorem \ref{non-periodic}, we address the effective electric field that arises in the effective parabolic problem (\ref{effective-parabolic}). In the subsequent sections, we analyze the associated electromagnetic problem and present the relevant results.


    \subsection{Proof of Proposition \ref{proposition-maxwell-1}: The Effective Electromagnetic Problem}    

In this section, we present the asymptotic analysis of the solution to the electromagnetic scattering problem (\ref{Maxwell-model}) for a cluster of nanoparticles \( D_j \) for \( j=1,2,\ldots,M \), in the limit as \( \delta \to 0 \), where the plasmonic nanoparticles are distributed within the region of interest \( \bm{\Omega} \).
\\
We start by introducing the Lippmann-Schwinger (LS) equation, which provides the solution to the electromagnetic scattering problem (\ref{Maxwell-model}) for \( i=1,2,\ldots,M \). The equation is formulated as
    \begin{align} \label{Maxwell-ls}
          \mathrm{E}_i(\mathrm{x}) + \eta \sum_{i=1}^M \mathbb{M}_{D_i}^{(\kappa)}\big[\mathrm{E}_i\big](\mathrm{x}) - k^2  \eta \sum_{i=1}^M \mathbb{N}_{D_i}^{(\kappa)}\big[\mathrm{E}_i\big](\mathrm{x}) = \mathrm{E}^\textit{in}(\mathrm{x}), \quad \text{for} \ x \in D := \bigcup_{i=1}^M D_i,
    \end{align}
where the magnetization operator and the Newtonian operator are given by:
    \begin{equation}\nonumber 
          \mathbb{M}_{D_i}^{(\kappa)}\big[\mathrm{E}\big](\mathrm{x}) = \nabla \int_{D_i} \nabla \mathcal{G}^{(\kappa)}(\mathrm{x},\mathrm{y}) \cdot \mathrm{E}(\mathrm{y}) \, d\mathrm{y}, \quad 
          \mathbb{N}_{D_i}^{(\kappa)}\big[\mathrm{E}\big](\mathrm{x}) = \int_{D_i} \mathcal{G}^{(\kappa)}(\mathrm{x},\mathrm{y}) \, \mathrm{E}(\mathrm{y}) \, d\mathrm{y}.
    \end{equation}
Here, \( \mathcal{G}^{(\kappa)}(\mathrm{x}, \mathrm{y}) \) represents the Green's function corresponding to the Helmholtz operator, and the contrast coefficient \( \eta \) is defined as:
    \begin{equation} \label{contrast}
          \eta := \varepsilon_p - \varepsilon_m,
    \end{equation}
where \( \varepsilon_p \) and \( \varepsilon_m \) are the permittivities of the plasmonic material and the surrounding medium, respectively.
\\
Next, we study the LS equation in \( \mathbb{L}^2(D_i) := (L^2(D_i))^3 \). To do this, we project the equation (\ref{Maxwell-ls}), onto the decomposition of \( \mathbb{L}^2(D_i) \), into three sub-spaces
    \begin{align}\nonumber
           \mathbb{L}^{2}(D_i) = \mathbb{H}_{0}(\textit{div}\ 0, D_i) \oplus \mathbb{H}_{0}(\textit{curl}\ 0, D_i) \oplus \nabla \mathbb{H}_{\textit{arm}},
    \end{align}
where \( \mathbb{H}_{0}(\textit{div}\ 0, D_i) \) and \( \mathbb{H}_{0}(\textit{curl}\ 0, D_i) \) are the sub-spaces corresponding to divergence-free and curl-free fields, respectively, and \( \nabla \mathbb{H}_{\textit{arm}} \) represents the gradient of harmonic functions.
\\
Finally, we are interested in analyzing the behavior of the system in the regime where:
    \begin{align}
           M \sim d^{-3}, \quad \text{and} \quad d \sim \delta^\lambda, \quad \text{for some non-negative parameter} \ \lambda.
    \end{align}
This scaling regime characterizes the relationship between the number of nanoparticles \( M \), the inter-particle distance \( d \), and the small parameter \( \delta \).


    \subsubsection{Construction of the Algebraic System for the Electromagnetic Problem}  

\begin{lemma} \label{lemma-maxwell}  
The vectors \(\displaystyle\mathbcal{Q}_i := \int_{D_i} \overset{3}{\mathbb{P}}(E_i)(y)\, dy\) for \(i=1,2,\ldots,M\) satisfy the following linear algebraic system
    \begin{align}\nonumber
          \mathbcal{Q}_i - \eta \sum_{j \neq i}^M \mathbcal{P}_{D_i} \cdot \bm{\Upsilon}^{(k)}(z_i, z_j)\ \mathbcal{Q}_j 
          = \mathbcal{P}_{D_i} \cdot \mathrm{E}^\textit{in}(z_i) + \mathcal{O}\left(\delta^{\min\{4-h, 7-2h-4\lambda\}}\right), \; \text{for} \ j = 1, 2, \ldots, M,
    \end{align}
where \(\mathbcal{P}_{D_i}\) is the polarization matrix, defined as:
    \begin{align}\nonumber
          \mathbcal{P}_{D_i} = \delta^3 \sum_{n} \frac{1}{1 + \eta \lambda_n^{(3)}} \left( \int_{B_i} e^{(3)}_n \otimes \int_{B_i} e^{(3)}_n \right),
    \end{align}
with \(e^{(3)}_n(\cdot)\) denoting the eigenfunctions associated $\nabla \mathbb{H}_{\textit{arm}}$-space in the domain \(B_i\).
\end{lemma} 
\noindent
In addition, the algebraic system described above is invertible under the following condition
    \begin{align}\nonumber
          |\eta| \max_{j=1,2,\ldots,M} \Vert \mathbcal{P}_{D_i} \Vert \ d^{-3} < 1.
    \end{align}
\begin{proof} 
We begin by recalling the Lippmann-Schwinger equation for the electric field \( \mathrm{E}_i(\mathrm{x}) \) within the nanoparticle's cluster
    \begin{align}\label{algebra-lippmann}
         \mathrm{E}_i(\mathrm{x}) + \eta \, \mathbb{M}_{D_i}^{(k)}\big[\mathrm{E}_i\big](\mathrm{x}) - k^2\eta \, \mathbb{N}_{D_i}^{(\kappa)}\big[\mathrm{E}_i\big](\mathrm{x}) = \mathrm{E}^\textit{in}(\mathrm{x}) - \eta \sum_{j\ne i}^M\left(\mathbb{M}_{D_j}^{(\kappa)} - k^2\mathbb{N}_{D_j}^{(\kappa)}\right)\big[\mathrm{E}_j\big](\mathrm{x}),\; x\in D_i
    \end{align}
where \(\eta = \varepsilon_p - \varepsilon_m\). We express the operators \(\mathbb{M}_{D_j}^{(\kappa)}\) and \(\mathbb{N}_{D_j}^{(\kappa)}\) using the dyadic Green's function \(\bm{\Upsilon}^{(k)}(x,y)\), given by
    \begin{align}\nonumber
        \bm{\Upsilon}^{(k)}(x,y) = \underset{x}{\text{Hess}}\,\mathcal{G}^{(k)}(x,y) + k^2 \mathcal{G}^{(k)}(x,y) \mathbb{I},
    \end{align}
where \(\mathcal{G}^{(k)}(x,y)\) is the Green's function of the Helmholtz equation. This allows us to write
    \begin{align}\nonumber
         \left(-\mathbb{M}_{D_j}^{(\kappa)} +k^2 \mathbb{N}_{D_j}^{(\kappa)}\right)\big[\mathrm{E}_j\big](\mathrm{x}) = \int_{D_j} \bm{\Upsilon}^{(k)}(x,y)\, \mathrm{E}_j(y)\, dy,\; \text{for}\ x\notin D_j.
    \end{align}
Let us define \(\mathbb{L}\) as
    \begin{align}\nonumber
        \mathbb{L} := \left( \mathbb{I} + \eta\, \mathbb{M}_{D_i}^{(-k)} - k^2\eta \mathbb{N}_{D_i}^{(-k)} \right)^{-1}(\mathbb{I}) := \mathbb{T}_{-k}^{-1}(\mathbb{I}),
    \end{align}
to simplify the Lippmann-Schwinger equation. Next, we decompose the electric field \(\mathrm{E}_j\) into its projections \(\mathrm{E}_j = \overset{1}{\mathbb{P}}(E_j) + \overset{3}{\mathbb{P}}(E_j),\)
where \(\overset{1}{\mathbb{P}}(E_j)\) is the component in \(\mathbb{H}_0(\textit{div}\ 0, D_j)\), and \(\overset{3}{\mathbb{P}}(E_j)\) is the component in \(\nabla \mathbb{H}_{\textit{arm}}\). Observe that \(\overset{2}{\mathbb{P}}(E_j) = 0\), as
\[
\int_{D_j} E_j \cdot e_n^{(2)} = \frac{1}{k^2} \int_{D_j} \delta_j^{-1} \, \text{curl} \, \text{curl} \, E_j \cdot e_n^{(2)} = 0,
\]
by integration by parts, using the fact that \(\text{curl} \, e_n^{(2)} = 0\) and \(e_n^{(2)} \times \nu = 0\) on \(\partial D_j\).
\\
Integrating both sides of the Lippmann-Schwinger equation (\ref{algebra-lippmann}) over the domain \(D_i\), we obtain
    \begin{align}
          \int_{D_i}\overset{3}{\mathbb{P}}(E_i)(y)\ dy 
          - \eta\sum_{j\ne i}^M\int_{D_i}\mathbb{L}(x) \cdot&\nonumber\int_{D_j} \bm{\Upsilon}^{(k)}(x,y)\ \overset{3}{\mathbb{P}}(E_j)(y) dydx 
          = \int_{D_i}\mathbb{L}(x)\cdot \mathrm{E}^\textit{in}(x) dx 
          \\ \nonumber&+ \eta\sum_{j\ne i}^M\int_{D_i}\mathbb{L}(x) \cdot\int_{D_j} \bm{\Upsilon}^{(k)}(x,y)\ \overset{1}{\mathbb{P}}(E_j)(y) dydx
    \end{align}
Since \(\mathbb{M}_{D_j}^{(\kappa)}\) vanishes on \(\mathbb{H}_0(\textit{div}\ 0, D_j)\), the term involving \(\overset{1}{\mathbb{P}}(E_j)\) simplifies to
\[
\int_{D_i} \mathbb{L}(x) \cdot \int_{D_j} \bm{\Upsilon}^{(k)}(x,y) \, \overset{1}{\mathbb{P}}(E_j)(y)\, dy\, dx = -k^2 \int_{D_i} \mathbb{L}(x) \cdot\int_{D_j} \mathcal{G}^{(k)}(x,y) (y)\, dy.
\]
Thus, the equation becomes
    \begin{align}\label{becomes}
          \int_{D_i}\overset{3}{\mathbb{P}}(E_i)(y)\ dy 
          - \eta\sum_{j\ne i}^M\int_{D_i}\mathbb{L}(x) \cdot&\nonumber\int_{D_j} \bm{\Upsilon}^{(k)}(x,y)\ \overset{3}{\mathbb{P}}(E_j)(y) dydx 
          = \int_{D_i}\mathbb{L}(x)\cdot \mathrm{E}^\textit{in}(x) dx 
          \\ &+ k^2\eta\sum_{j\ne i}^M\int_{D_i}\mathbb{L}(x) \cdot\int_{D_j} \mathcal{G}^{(k)}(x,y)\ \overset{1}{\mathbb{P}}(E_j)(y) dydx
    \end{align}
To analyze this system, we perform a Taylor expansion of the Green’s function \(\bm{\Upsilon}^{(k)}(x,y)\) and \(\mathcal{G}^{(k)}(x,y)\) around the points \(z_i \in D_i\) and \(z_j \in D_j\), with \(i \neq j\). For \(\bm{\Upsilon}^{(k)}(x,y)\), we have
\[
\bm{\Upsilon}^{(k)}(x,y) = \bm{\Upsilon}^{(k)}(z_i,z_j) + \int_0^1 \underset{x}{\nabla} \bm{\Upsilon}^{(k)}(z_i + \theta(x - z_i), z_j) \cdot (x - z_i) \, d\theta + \int_0^1 \underset{y}{\nabla} \bm{\Upsilon}^{(k)}(z_i, z_j + \theta(y - z_j)) \cdot (y - z_j) \, d\theta.
\]
Similarly, the Taylor expansion for \(\mathcal{G}^{(k)}(x,y)\) is
\[
\mathcal{G}^{(k)}(x,y) = \mathcal{G}^{(k)}(x, z_j) + \underset{y}{\nabla} \mathcal{G}^{(k)}(x, z_j) \cdot (y - z_j) + \frac{1}{2} \int_0^1 (y - z_j)^{\perp} \cdot \underset{y}{\textit{Hess}} \mathcal{G}^{(k)}(z_i, z_j + \theta(x - z_j)) \cdot (y - z_j)\, d\theta.
\]
Using the fact that $\displaystyle\int_{D_j}\overset{1}{\mathbb{P}}(E_j)(y) dy = 0$, we can deduce the following expression from (\ref{becomes})
    \begin{align}
          \nonumber&\int_{D_i}\overset{3}{\mathbb{P}}(E_i)(y)\, dy 
          - \eta\sum_{j\neq i}^{M}\int_{D_i}\mathbb{L}(x) \cdot \bm{\Upsilon}^{(k)}(z_i, z_j) \int_{D_j}\overset{3}{\mathbb{P}}(E_j)(y)\, dy 
          \\\nonumber&= \int_{D_i}\mathbb{L}(x) \cdot \mathrm{E}^{\textit{in}}(x)\, dx 
          + \eta \sum_{j\neq i}^{M} \int_{D_i} \mathbb{L}(x) \cdot \int_0^1 \underset{x}{\nabla} \bm{\Upsilon}^{(k)}(z_i + \theta(x - z_i), z_j) \cdot (x - z_i)\, d\theta \int_{D_j} \overset{3}{\mathbb{P}}(E_j)(y)\, dy\, dx 
          \\ \nonumber& + \eta \sum_{j\neq i}^{M} \int_{D_i} \mathbb{L}(x) \cdot \int_{D_j} \int_0^1 \underset{y}{\nabla} \bm{\Upsilon}^{(k)}(z_i, z_j + \theta(x - z_j)) \cdot (y - z_j) \overset{3}{\mathbb{P}}(E_j)(y)\, d\theta\, dy\, dx 
          \\ \nonumber& + k^2  \eta \sum_{j\neq i}^{M} \int_{D_i} \mathbb{L}(x) \cdot \int_{D_j} \underset{x}{\nabla} \mathcal{G}^{(k)}(x, z_j) \cdot (y - z_j) \overset{1}{\mathbb{P}}(E_j)(y)\, dy\, dx 
          \\ \nonumber& + k^2 \eta \sum_{j\neq i}^{M} \int_{D_i} \mathbb{L}(x) \cdot \int_{D_j} \frac{1}{2} \int_0^1 (y - z_j)^\perp \cdot \underset{y}{\nabla} \mathcal{G}^{(k)}(z_i, z_j + \theta(x - z_j)) \cdot (y - z_j) \overset{1}{\mathbb{P}}(E_j)(y)\, d\theta\, dy\, dx.
    \end{align}
This can be rewritten by focusing on the leading order term
    \begin{align}\label{al}
          \int_{D_i}\overset{3}{\mathbb{P}}(E_i)(y)\, dy 
          - \eta \sum_{j\ne i}^{M} \int_{D_i} \mathbb{L}(x) \cdot \bm{\Upsilon}^{(k)}(z_i, z_j) \int_{D_j} \overset{3}{\mathbb{P}}(E_j)(y)\, dy 
          = \int_{D_i}\mathbb{L}(x)\cdot \mathrm{E}^{\textit{in}}(x)\, dx + \textit{Err}_{\mathbb{L},D_i},
    \end{align}
where
    \begin{align}
          \textit{Err}_{\mathbb{L},D_i} := \textit{Err}_{\mathbb{L},1,D_i} + \textit{Err}_{\mathbb{L},2,D_i} + \textit{Err}_{\mathbb{L},3,D_i} + \textit{Err}_{\mathbb{L},4,D_i}.
    \end{align}
The following estimates for the error terms are then obtained. First, we have
    \begin{align}\label{al1}
           \big| \textit{Err}_{\mathbb{L},1,D_i} \big| 
           \nonumber&:= \Big|\eta \sum_{j \neq i}^{M} \int_{D_i} \mathbb{L}(x) \cdot \int_0^1 \nabla_x \bm{\Upsilon}^{(k)}(z_i + \theta(x - z_i), z_j) \cdot (x - z_i)\, d\theta \int_{D_j} \overset{3}{\mathbb{P}}(E_j)(y)\, dy\, dx \Big| \\
           \nonumber&\lesssim \Vert \mathbb{L} \Vert_{L^2(D_i)} \sum_{j \ne i}^{M} \Big\Vert \int_0^1 \nabla_x \bm{\Upsilon}^{(k)}(z_i + \theta(x - z_i), z_j) \cdot (x - z_i)\, d\theta \Big\Vert_{L^2(D_i \times D_j)} \Vert \overset{3}{\mathbb{P}}(E_j)(y) \Vert_{L^2(D_j)} \\
           &\lesssim \delta^4 \Vert \mathbb{L} \Vert_{L^2(D_i)}\sum_{j \ne i}^{M} d_{ij}^{-4} \Vert \overset{3}{\mathbb{P}}(E_j) \Vert_{L^2(D_j)} = \mathcal{O}\Big(\Vert \mathbb{L} \Vert_{L^2(D_i)}\delta^4 d^{-4} \underset{j}{\max} \Vert \overset{3}{\mathbb{P}}(E_j) \Vert_{L^2(D_j)} \Big).
    \end{align}
Similarly, we obtain the following for the error term
    \begin{align}\label{al2}
           \big| \textit{Err}_{\mathbb{L},2,D_i} \big| 
           \nonumber&:= \Big|\eta \sum_{j \ne i}^{M} \int_{D_i} \mathbb{L}(x) \cdot \int_{D_j} \int_0^1 \nabla_y \bm{\Upsilon}^{(k)}(z_i, z_j + \theta(x - z_j)) \cdot (y - z_j) \overset{3}{\mathbb{P}}(E_j)(y)\, d\theta\, dy\, dx \Big| \\
           &= \mathcal{O}\Big(\Vert \mathbb{L} \Vert_{L^2(D_i)}\delta^4 d^{-4} \underset{j}{\max} \Vert \overset{3}{\mathbb{P}}(E_j) \Vert_{L^2(D_j)} \Big).
    \end{align}
Next, we deduce
    \begin{align}\label{al3}
           \big| \textit{Err}_{\mathbb{L},3,D_i} \big| 
           \nonumber&:= \Big| k^2 \eta \sum_{j \ne i}^{M} \int_{D_i} \mathbb{L}(x) \cdot \int_{D_j} \nabla_x \mathcal{G}^{(k)}(x, z_j) \cdot (y - z_j) \overset{1}{\mathbb{P}}(E_j)(y)\, dy\, dx \Big| \\
           \nonumber&\lesssim \Vert \mathbb{L} \Vert_{L^2(D_i)} \sum_{j \ne i}^{M} \Big\Vert \int_{D_j} \nabla_x \mathcal{G}^{(k)}(x, z_j) \cdot (y - z_j) \overset{1}{\mathbb{P}}(E_j)(y)\, dy \Big\Vert_{L^2(D_i)} \\
           &\lesssim \delta^4\Vert \mathbb{L} \Vert_{L^2(D_i)} \sum_{j \ne i}^{M} d_{ij}^{-2} \Vert \overset{1}{\mathbb{P}}(E_j) \Vert_{L^2(D_j)} = \mathcal{O}\Big(\Vert \mathbb{L} \Vert_{L^2(D_i)}\delta^4 d^{-3} \underset{j}{\max} \Vert \overset{1}{\mathbb{P}}(E_j) \Vert_{L^2(D_j)} \Big),
    \end{align}
and finally, we have
    \begin{align}\label{al4}
           \nonumber&\big| \textit{Err}_{\mathbb{L},4,D_i} \big| 
           \\ \nonumber&:= \Big| k^2  \eta \sum_{j \ne i}^{M} \int_{D_i} \mathbb{L}(x) \cdot \int_{D_j} \frac{1}{2} \int_0^1 (y - z_j)^\perp \cdot \underset{x}{\textit{Hess}} \mathcal{G}^{(k)}(z_i, z_j + \theta(x - z_j)) \cdot (y - z_j) \overset{1}{\mathbb{P}}(E_j)(y)\, d\theta\, dy\, dx \Big| \\
           \nonumber&\lesssim \Vert \mathbb{L} \Vert_{L^2(D_i)} \sum_{j \ne i}^{M} \Big\Vert \int_{D_j} \frac{1}{2} \int_0^1 (y - z_j)^\perp \cdot \underset{x}{\textit{Hess}} \mathcal{G}^{(k)}(z_i, z_j + \theta(x - z_j)) \cdot (y - z_j) \overset{1}{\mathbb{P}}(E_j)(y)\, d\theta\, dy \Big\Vert_{L^2(D_i)} \\
           &\lesssim \delta^5\Vert \mathbb{L} \Vert_{L^2(D_i)} \sum_{j \ne i}^{M} d_{ij}^{-3} \Vert \overset{1}{\mathbb{P}}(E_j) \Vert_{L^2(D_j)} 
           = \mathcal{O}\Big(\Vert \mathbb{L} \Vert_{L^2(D_i)}\delta^5 |\log(d)| d^{-3} \underset{j}{\max} \Vert \overset{1}{\mathbb{P}}(E_j) \Vert_{L^2(D_j)} \Big).
    \end{align}
Therefore, plugging the estimates (\ref{al1}), (\ref{al2}), (\ref{al3}), and (\ref{al4}) in (\ref{al}), we arrive at
    \begin{align}
          \int_{D_i}\overset{3}{\mathbb{P}}(E_i)(y)\ dy &\nonumber+ \eta\sum_{j\ne i}^M\int_{D_i}\mathbb{L}(x)dx \cdot \bm{\Upsilon}^{(k)}(z_i,z_j)\int_{D_j} \overset{3}{\mathbb{P}}(E_j)(y) dy 
          = \int_{D_i}\mathbb{L}(x)\cdot \mathrm{E}^\textit{in}(x) dx
          \\ \nonumber&+  \mathcal{O}\Big(\Vert \mathbb{L} \Vert_{L^2(D_i)}\delta^4 d^{-4}\underset{j}{\max} \Vert \overset{3}{\mathbb{P}}(E_j)\Vert_{L^2(D_j)}\Big)
          +  \mathcal{O}\Big(\Vert \mathbb{L} \Vert_{L^2(D_i)}\delta^4 d^{-3}\underset{j}{\max} \Vert \overset{1}{\mathbb{P}}(E_j)\Vert_{L^2(D_j)}\Big) 
          \\ &+ \mathcal{O}\Big(\Vert \mathbb{L} \Vert_{L^2(D_i)}\delta^5|\log(d)| d^{-3}\underset{j}{\max} \Vert \overset{1}{\mathbb{P}}(E_j)\Vert_{L^2(D_j)}\Big).
    \end{align}
We now introduce the notation $\displaystyle\mathbcal{P}_{D_i} := \int_{D_i}\mathbb{L}(x)dx.$ By applying Taylor's series expansion for $z_i \in D_i,$ we obtain the following expression
    \begin{align}
          \int_{D_i}\mathbb{L}(x)\cdot \mathrm{E}^\textit{in}(x) dx 
          \nonumber&= \mathrm{E}^\textit{in}(z_i)\cdot \mathbcal{P}_{D_i} 
          + \int_{D_i}\mathbb{L}(x) \cdot \int_0^1 \underset{x}{\nabla} \mathrm{E}^\textit{in}(z_i +\theta (x-z_i))\cdot (x-z_i)d\theta
         \\ &= \mathrm{E}^\textit{in}(z_i)\cdot \mathbcal{P}_{D_i} + \mathcal{O}\Big(\delta^\frac{5}{2}\Vert \mathbb{L} \Vert_{L^2(D_i)}\Big).
    \end{align}
Using the expression derived above, we can further deduce
    \begin{align}\label{first-al-1}
          \int_{D_i}\overset{3}{\mathbb{P}}(E_i)(y)\ dy 
          &\nonumber- \eta\sum_{j\ne i}^M \mathbcal{P}_{D_i} \cdot \bm{\Upsilon}^{(k)}(z_i,z_j)\int_{D_j} \overset{3}{\mathbb{P}}(E_j)(y) dy
          = \mathbcal{P}_{D_i}\cdot \mathrm{E}^\textit{in}(z_i)
          \\ \nonumber&+  \mathcal{O}\Big(\Vert \mathbb{L} \Vert_{L^2(D_i)}\delta^4 d^{-4}\underset{j}{\max} \Vert \overset{3}{\mathbb{P}}(E_j)\Vert_{L^2(D_j)}\Big)
          +  \mathcal{O}\Big(\Vert \mathbb{L} \Vert_{L^2(D_i)}\delta^4 d^{-2}\underset{j}{\max} \Vert \overset{1}{\mathbb{P}}(E_j)\Vert_{L^2(D_j)}\Big) 
          \\ &+ \mathcal{O}\Big(\Vert \mathbb{L} \Vert_{L^2(D_i)}\delta^5|\log(d)| d^{-3}\underset{j}{\max} \Vert \overset{1}{\mathbb{P}}(E_j)\Vert_{L^2(D_j)}\Big) 
          + \mathcal{O}\Big(\delta^\frac{5}{2}\Vert \mathbb{L} \Vert_{L^2(D_i)}\Big).
    \end{align}
From this point onward, we will focus on estimating the term $\mathbcal{P}_{D_i}$ as well as determining $\Vert \mathbb{L} \Vert_{L^2(D_i)}.$ Initially, we note that
    \begin{align}
           \mathbcal{P}_{D_i} 
           := \delta^3\int_{B_i}\Tilde{\mathbb{L}}(\xi)d\xi 
           = \delta^3\sum_{n}\Big\langle \Tilde{\mathbb{L}}; e^{(3)}_n\Big\rangle \otimes \int_{B_i} e^{(3)}_n(\xi)d\xi.
    \end{align}
First, we see that
\begin{align}
    \int_{B_i} \Tilde{e}^{(3)}_n(\xi)d\xi = \int_{B_i} \mathrm{I}\cdot \Tilde{e}^{(3)}_n(\xi)d\xi = \int_{B_i} \mathbb{T}_{-k\delta}(\tilde{\mathbb{L}})(\xi)\cdot \Tilde{e}^{(3)}_n(\xi)d\xi = \int_{B_i} \tilde{\mathbb{L}}(\xi)\cdot \mathbb{T}_{k\delta}(\Tilde{e}^{(3)}_n)(\xi)d\xi
\end{align}
Moreover, due to the definition of $\mathbb{T}_{k\delta} := \mathbb{I} + \eta\, \mathbb{M}_{B_i}^{(k\delta)} + k^2\eta \mathbb{N}_{B_i}^{(k\delta)}$ and the fact that the  Magnetization operator $\mathbb{M}^{(0)}_{\mathrm{B}_i}: \nabla \mathbb{H}_{\textit{arm}}\rightarrow \nabla \mathbb{H}_{\textit{arm}}$ induces a complete orthonormal basis namely $\big(\lambda^{(3)}_{\mathrm{n}},\Tilde{e}^{(3)}_{\mathrm{n}}\big)_{\mathrm{n} \in \mathbb{N}}$, we observe that
   \begin{align}
          \Big\langle \Tilde{\mathbb{L}}; e^{(3)}_n\Big\rangle = \frac{1}{1+\eta \lambda_n^{(3)}}\int_{B_i} e^{(3)}_n(\xi)d\xi + \frac{\eta}{1+\eta \lambda_n^{(3)}}\Big\langle \tilde{\mathbb{L}}, \widetilde{\textit{err}}_{1,\textit{al}}(e^{(3)}_n)\Big\rangle + \frac{\eta}{1+\eta \lambda_n^{(3)}}\Big\langle \tilde{\mathbb{L}}, \widetilde{\textit{err}}_{2,\textit{al}}(e^{(3)}_n)\Big\rangle, \nonumber
   \end{align}
from this, we can deduce that
   \begin{align}
          \mathbcal{P}_{D_i} 
          &\nonumber= \delta^3\sum_{n}\frac{1}{1+\eta \lambda_n^{(3)}}\int_{B_i} e^{(3)}_n(\xi)d\xi\otimes \int_{B_i} e^{(3)}_n(\xi)d\xi 
          \\ &+ \underbrace{\delta^3\sum_{n}\frac{\eta}{1+\eta \lambda_n^{(3)}}\Big\langle \tilde{\mathbb{L}}, \widetilde{\textit{err}}_{1,\textit{al}}(e^{(3)}_n)\Big\rangle \otimes \int_{B_i} e^{(3)}_n(\xi)d\xi}_{:=\textit{err}_{1,i}}
         + \underbrace{\delta^3\sum_{n}\frac{\eta}{1+\eta \lambda_n^{(3)}}\Big\langle \tilde{\mathbb{L}}, \widetilde{\textit{err}}_{2,\textit{al}}(e^{(3)}_n)\Big\rangle \otimes \int_{B_i} e^{(3)}_n(\xi)d\xi}_{:= \textit{err}_{2,i}},
   \end{align}
where we recall that
    \begin{align}
          \widetilde{\textit{err}}_{1,\textit{al}}\big[e_n^{(3)}\big] \nonumber&= \frac{\omega^2}{2} \delta^2 \mathbb{N}_{B_i}^{(\mathrm{0})}\big[e_n^{(3)}\big](\mathrm{x}) + \frac{i\omega^3\beta^2_\mathrm{m}}{12\pi}\delta^3\int_{B_i} e_n^{(3)}(\xi)d\xi - \frac{\omega^2}{2} \delta^2\int_{B_i} \mathbb{G}^{(\mathrm{0})}(\varsigma,\xi)\frac{\mathrm{A}(\varsigma,\xi)\cdot e_n^{(3)}}{\Vert \varsigma-\xi\Vert^2}d\xi
          \\ &- \frac{1}{4\pi}\sum_{\mathrm{j}\ge 3} \frac{(i\omega\beta^\frac{1}{2}_\mathrm{m})^{\mathrm{j}+1}}{(\mathrm{j}+1)!} \delta^{j+1}\int_{B_i}\underset{\varsigma}{\textit{Hess}}(\Vert \varsigma-\xi\Vert^\mathrm{j})\cdot e_n^{(3)}(\xi)d\xi,
    \end{align}
and
    \begin{align}
          \widetilde{\textit{err}}_{2,\textit{al}}\big[e_n^{(3)}\big] = \delta^2 \mathbb{N}_{B_i}^{(\mathrm{0})}\big[e_n^{(3)}\big](\mathrm{x}) + \frac{i\omega\beta^\frac{1}{2}_\mathrm{m}}{4\pi}\delta^3\int_{B_i} e_n^{(3)}(\xi)d\xi 
          - \frac{1}{4\pi}\sum_{\mathrm{j}\ge 1} \frac{(i\omega\beta^\frac{1}{2}_\mathrm{m})^{\mathrm{j}+1}}{(\mathrm{j}+1)!} \delta^{j+1}\int_{B_i}\Vert \varsigma-\xi\Vert^\mathrm{j}\ e_n^{(3)}(\xi)d\xi.
    \end{align}
It can be observed in both expressions for $\widetilde{\textit{err}}_{1,\textit{al}}\big[e_n^{(3)}\big]$ or $\widetilde{\textit{err}}_{2,\textit{al}}\big[e_n^{(3)}\big]$, the presence of factors such as $\delta^2$ and $\delta^3$, hence, we see that the dominating term is $\delta^2 \mathbb{N}_{B_i}^{(\mathrm{0})}\big[e_n^{(3)}\big].$ We begin by analyzing the term \(\textit{err}_{1,i}\):
     \begin{align}
         |\textit{err}_{1,i}| 
         &\nonumber:= \delta^3\Big|\sum_{n}\frac{\eta}{1+\eta \lambda_n^{(3)}}\Big\langle \tilde{\mathbb{L}}, \widetilde{\textit{err}}_{1,\textit{al}}(e^{(3)}_n)\Big\rangle \otimes \int_{B_i} e^{(3)}_n(\xi)d\xi\Big|  
         \\ &\nonumber\simeq  \delta^5\Big|\sum_{n}\frac{\eta}{1+\eta \lambda_n^{(3)}}\Big\langle \tilde{\mathbb{L}}, \mathbb{N}_{B_i}^{(\mathrm{0})}\big[e_n^{(3)}\big]\Big\rangle \otimes \int_{B_i} e^{(3)}_n(\xi)d\xi\Big| + \delta^6\Big|\sum_{n}\frac{\eta}{1+\eta \lambda_n^{(3)}}\Big\langle \tilde{\mathbb{L}}, \int_{B_i} e_n^{(3)}(\xi)d\xi\Big\rangle \otimes \int_{B_i} e^{(3)}_n(\xi)d\xi\Big| 
         \\ \nonumber &\lesssim \delta^{5-h} \Vert \Tilde{\mathbb{L}} \Vert_{L^2(B_i)},
   \end{align}
A similar estimate holds for the term \(\textit{err}_{2,i}\). Based on these previous estimates for \(\textit{err}_{1,i}\) and \(\textit{err}_{2,i}\), we can conclude
    \begin{align}\label{pol-esti}
          \mathbcal{P}_{D_i} 
          &= \delta^3\sum_{n}\frac{1}{1+\eta \lambda_n^{(3)}}\int_{B_i} e^{(3)}_n(\xi)d\xi\otimes \int_{B_i} e^{(3)}_n(\xi)d\xi + \mathcal{O}\Big(\delta^{5-h} \Vert \Tilde{\mathbb{L}} \Vert_{L^2(B_i)}\Big).
   \end{align}
Next, we estimate the term \(\|\tilde{\mathbb{L}}\|_{L^2(B_i)}\):
    \begin{align}
          \Vert \Tilde{\mathbb{L}} \Vert_{L^2(B_i)}^2 
           \nonumber&= \sum_n \big|\langle \Tilde{\mathbb{L}}, e_n^{(1)}\rangle\big|^2 + \sum_n \big|\langle \Tilde{\mathbb{L}}, e_n^{(3)}\rangle\big|^2
           \\ &\nonumber\lesssim \sum_n \frac{1}{|1+\eta \lambda_n^{(3)}|^2}|\langle \mathbb{I},e_n^{(3)}\rangle|^2 + \sum_n\frac{|\eta|^2}{|1+\eta \lambda_n^{(3)}|^2}\Big|\langle \tilde{\mathbb{L}}, \widetilde{\textit{err}}_{1,\textit{al}}(e^{(3)}_n)\rangle\Big|^2 + \sum_n\frac{|\eta|^2}{|1+\eta \lambda_n^{(3)}|^2}\Big|\langle \tilde{\mathbb{L}}, \widetilde{\textit{err}}_{2,\textit{al}}(e^{(3)}_n)\rangle\Big|^2.
    \end{align}
Thus, we obtain
    \begin{align}
          \Vert \Tilde{\mathbb{L}} \Vert_{L^2(B_i)}^2 \lesssim \mathcal{O}(\delta^{-2h}) + \mathcal{O}(\delta^{4-2h}\Vert \Tilde{\mathbb{L}} \Vert_{L^2(B_i)}^2).
    \end{align}
By choosing,
$$h < 2$$
we can further deduce that
    \begin{align}
          \Vert \Tilde{\mathbb{L}} \Vert_{L^2(B_i)} \lesssim \delta^{-h}.
    \end{align}
Consequently, using the above estimate and plugging it in (\ref{pol-esti}), we have    
    \begin{align}
          \mathbcal{P}_{D_i} 
          = \delta^3\sum_{n}\frac{1}{1+\eta \lambda_n^{(3)}}\int_{B_i} e^{(3)}_n(\xi)d\xi\otimes \int_{B_i} e^{(3)}_n(\xi)d\xi 
          + \mathcal{O}\big(\delta^{5-2h}\big).
    \end{align}
Again, substituting the estimate for \(\|\tilde{\mathbb{L}}\|_{L^2(B_i)}\) into equation \((\ref{first-al-1})\), we derive
    \begin{align}\label{first-al}
          \int_{D_i}\overset{3}{\mathbb{P}}(E_i)(y)\ dy 
          &\nonumber- \eta\sum_{j\ne i}^M \mathbcal{P}_{D_i} \cdot \bm{\Upsilon}^{(k)}(z_i,z_j)\int_{D_j} \overset{3}{\mathbb{P}}(E_j)(y) dy
          = \mathbcal{P}_{D_i}\cdot \mathrm{E}^\textit{in}(z_i) + \mathcal{O}\Big(\delta^{4-h}\Big)
          \\ \nonumber &+  \mathcal{O}\Big(\delta^{\frac{11}{2}-h} d^{-4}\underset{j}{\max} \Vert \overset{3}{\mathbb{P}}(E_j)\Vert_{L^2(D_j)}\Big)
          +  \mathcal{O}\Big(\delta^{\frac{11}{2}-h} d^{-3}\underset{j}{\max} \Vert \overset{1}{\mathbb{P}}(E_j)\Vert_{L^2(D_j)}\Big) 
          \\ &+ \mathcal{O}\Big(\delta^{\frac{13}{2}-h}|\log(d)| d^{-3}\underset{j}{\max} \Vert \overset{1}{\mathbb{P}}(E_j)\Vert_{L^2(D_j)}\Big).
    \end{align}
The above algebraic system is invertible under the following condition, as discussed in \cite[Section 6.1]{cao-ahcene-sini-jlms}:
    \begin{align}\nonumber
         \eta \underset{j=1,2,\ldots,M}{\max}\Vert \mathbcal{P}_{D_i} \Vert_{L^\infty(\bm{\Omega})}d^{-3}<1.
    \end{align}    
To finalize the proof of Lemma \(\ref{lemma-maxwell}\), we require the following proposition, which states that:


   \subsubsection{\texorpdfstring{A-Priori Estimates in $\mathbb{L}^4(D)$}{A-Priori Estimates}}\label{a-priori}     

\paragraph{Proof of Proposition \ref{proposition-maxwell-1}.}\label{proposition-maxwell}      
Consider the electromagnetic scattering problem (\ref{Maxwell-model}) for cluster of plasmonic nanoparticles $D_i$ for $i=1,2,\ldots,M.$ Then, for $h,\lambda$ satisfying
    \begin{align}\nonumber
          3-3\lambda-h \ge 0 \quad \text{and} \quad \frac{9}{5}< h <2,
    \end{align}
we have the following a-priori estimates
    \begin{align}\nonumber
           \underset{i}{\max}\big\Vert \mathrm{E}_i\big\Vert_{\mathbb{L}^2(D_i)} \lesssim \delta^{\frac{3}{2}-h}.
    \end{align}
In addition to that, we have
    \begin{align}\nonumber
         \big\Vert E_i\big\Vert_{\mathbb{L}^\mathrm{4}(D_i)} \sim \delta^{\frac{3}{4}-\mathrm{h}}.
    \end{align}

\begin{proof} 
The proof of the Proposition is structured into two parts. We begin by deriving an a priori estimate in the space \(\mathbb{L}^2(D)\).
\subsubsection*{Part 1: Derivation of the \(\mathbb{L}^2(D)\) Estimate} We derive the proof by projecting the Lippmann-Schwinger equation (\ref{Maxwell-ls}) onto each of the three sub-spaces outlined in (\ref{decomposition-introduction}).

         
     \subparagraph{1. Projection onto $\mathbb{H}_{0}(\textit{div}\ 0,D_i)$.}
We first consider the Lippmann-Schwinger equation (\ref{Maxwell-ls})
\begin{align}\label{maxwell-hdiv}
           \mathrm{E}_i(\mathrm{x}) + \eta \;\mathbb{M}_{D_i}^{(k)}\big[\mathrm{E}_i\big](\mathrm{x}) 
           = \mathrm{E}^\textit{in}(\mathrm{x}) 
           - \eta \sum_{j\ne m}^M\Big(\mathbb{M}_{D_j}^{(\kappa)}\big[\mathrm{E}_j\big](\mathrm{x})  
           - k^2\mathbb{N}_{D_j}^{(\kappa)}\Big)\big[\mathrm{E}_i\big](\mathrm{x}) 
           + k^2\eta \; \mathbb{N}_{D_i}^{(\kappa)}\big[\mathrm{E}_i\big](\mathrm{x})
    \end{align}
Then, we note that $\mathbb{M}^{(\kappa)}[\mathrm{E}_i]$ vanishes in the subspace $\mathbb{H}_{0}(\textit{div}\ 0,D_i)$. Consequently, projecting the above equation onto $\mathbb{H}_{0}(\textit{div}\ 0,D_i)$ yields the following equation for $x\in D_i$
    \begin{align}
          \nonumber \big\langle \mathrm{E}_i, e_n^{(1)}\big\rangle 
          = \big\langle \mathrm{E}^\textit{in}, e_n^{(1)}\big\rangle 
          + k^2\eta\big\langle \mathbb{N}^{(\kappa)}_{D_i}\big[\mathrm{E}_i\big](\mathrm{x}), e_n^{(1)}\big\rangle 
          + k^2\eta\sum_{j\ne i}^M\big\langle \mathbb{N}^{(\kappa)}_{D_j}\big[\mathrm{E}_j\big](\mathrm{x}), e_n^{(1)}\big\rangle,
    \end{align}
which simplifies to
\begin{align}\label{5.8}
          \big\langle \mathrm{E}_i, e_n^{(1)}\big\rangle 
          = \big\langle \mathrm{E}^\textit{in}, e_n^{(1)}\big\rangle 
          + k^2\eta\big\langle \mathbb{N}^{(\kappa)}_{D_i}\big[\mathrm{E}_i\big](\mathrm{x}), e_n^{(1)}\big\rangle 
          + k^2\eta\ \textit{Err}^{(1)}_{n,D_j},
    \end{align}
where we define the error term as after expanding it using Taylor's series expansion
    \begin{align}\label{errndi}
          \textit{Err}^{(1)}_{n,D_j} 
          \nonumber&:= \sum_{j\ne i}^M\Big\langle \mathbb{G}^{(k)}(z_i,z_j)\int_{D_j}\mathrm{E}_j(y)dy; e_{n}^{(1)}\Big\rangle 
          \\ \nonumber&+ \sum_{j\ne i}^M\Big\langle \int_0^1 \underset{x}{\nabla}\mathbb{G}^{(k)}(z_i+\theta(x-z_i),z_j)\cdot (x-z_i)d\theta\int_{D_j}\mathrm{E}_j(y)dy; e_{n}^{(1)}\Big\rangle 
          \\ &+ \sum_{j\ne i}^M\Big\langle \int_{D_j}\int_0^1 \underset{y}{\nabla}\mathbb{G}^{(k)}(z_i,z_j+\theta(y-z_j))\cdot (y-z_j)\mathrm{E}_j(y)d\theta dy; e_{n}^{(1)}\Big\rangle.
    \end{align}
Upon re-scaling to the domains $B_i$ and $B_j$ in equation (\ref{5.8}), we obtain
    \begin{align}
         \big\langle \Tilde{\mathrm{E}}_i, \Tilde{e}_n^{(1)}\big\rangle 
         &\nonumber= \big\langle \Tilde{\mathrm{E}}^\textit{in}, \Tilde{e}_n^{(1)}\big\rangle 
         + k^2\eta\ \delta^2 \big\langle \mathbb{N}_{B_i}^{(\kappa)}\big[\Tilde{\mathrm{E}}_i\big](\mathrm{x}), \tilde{e}_n^{(1)}\big\rangle
         + k^2\eta\ \widetilde{\textit{Err}}^{(1)}_{n,B_j}.
    \end{align}
Taking the squared modulus and summing over $n$, we derive
    \begin{align}\nonumber
         \Vert \overset{1}{\mathbb{P}}(\Tilde{\mathrm{E}}_i)\Vert^2_{L^2(B_i)} 
         \lesssim  \Vert \overset{1}{\mathbb{P}}(\Tilde{\mathrm{E}}^\textit{in})\Vert^2_{L^2(B_i)} 
         + \delta^4  \Vert \mathbb{N}_{B_i}^{(\kappa)}\big[\Tilde{\mathrm{E}}_i\big]\Vert^2_{L^2(B_i)} 
         + \sum_n |\widetilde{\textit{Err}}^{(1)}_{n,B_j}|^2.
    \end{align}
Using the continuity of the Newtonian operator, we then infer
    \begin{align}\label{errhdiv}
          \Vert \overset{1}{\mathbb{P}}(\Tilde{\mathrm{E}}_i)\Vert^2_{L^2(B_i)} \lesssim  \Vert \overset{1}{\mathbb{P}}(\Tilde{\mathrm{E}}^\textit{in})\Vert^2_{L^2(B_i)} 
          + \delta^4\ \Vert \Tilde{\mathrm{E}}_i\Vert^2_{L^2(B_i)}
          + \sum_n |\widetilde{\textit{Err}}^{(1)}_{n,B_j}|^2.
    \end{align}
Next, we focus on estimating the term $\sum_n |\widetilde{\textit{Err}}^{(1)}_{n,B_j}|^2$. Scaling (\ref{errndi}) to $B_j$, we get
    \begin{align}
          \widetilde{\textit{Err}}_{n,B_j} 
         &\nonumber= \delta^3\sum_{j\ne i}^M\Big\langle \mathbb{G}^{(k)}(z_i,z_j)\int_{B_j}\Tilde{\mathrm{E}}_j(\xi)d\xi; \Tilde{e}_{n}^{(1)}\Big\rangle 
          + \delta^4\sum_{j\ne i}^M\Big\langle \int_0^1 \underset{\varsigma}{\nabla}\mathbb{G}^{(k)}(z_i+\theta\varsigma,z_j)\cdot \varsigma d\theta\int_{B_j}\Tilde{\mathrm{E}}_j(\xi)d\xi; \Tilde{e}_{n}^{(1)}\Big\rangle 
          \\ \nonumber&+ \delta^4\sum_{j\ne i}^M\Big\langle \int_{B_j}\int_0^1 \underset{\xi}{\nabla}\mathbb{G}^{(k)}(z_i,z_j+\theta\xi)\cdot \xi\Tilde{\mathrm{E}}_j(\xi)d\theta d\xi; \Tilde{e}_{n}^{(1)}\Big\rangle.
    \end{align}
Taking the squared modulus and summing over $n$, we obtain the bound
    \begin{align}
          \sum_n |\widetilde{\textit{Err}}^{(1)}_{n,B_j}|^2
          &\nonumber\lesssim M\delta^6\Bigg( \sum_{j\ne i}^M \Big\Vert \mathbb{G}^{(k)}(z_i,z_j)\int_{B_j}\Tilde{\mathrm{E}}_j(\xi)d\xi \Big\Vert^2_{L^2(B_i)}  \\ \nonumber&+ \delta^2 \sum_{j\ne i}^M \Big\Vert\underset{\varsigma}{\nabla}\mathbb{G}^{(k)}(z_i+\theta\varsigma,z_j)\cdot \varsigma\int_{B_j}\Tilde{\mathrm{E}}_j(\xi)d\xi \Big\Vert^2_{L^2(B_i)}
          + \delta^2 \sum_{j\ne i}^M \Big\Vert\underset{\xi}{\nabla}\mathbb{G}^{(k)}(z_i,z_j+\theta\xi)\cdot \xi\Tilde{\mathrm{E}}_j(\xi) \Big\Vert^2_{L^2(B_i)}\Bigg).
\end{align}
Given that $\mathbb{G}^{(k)}(z_i,z_j) = \mathcal{O}\left(\frac{1}{|z_i-z_j|}\right)$, the resulting estimate is as follows
    \begin{align}\label{hdiverr1}
          \sum_n |\widetilde{\textit{Err}}^{(1)}_{n,B_j}|^2
          &\nonumber \lesssim M\delta^6\ d^{-3} \underset{j}{\max}\Vert \Tilde{\mathrm{E}}_j\Vert_{L^2(B_j)} 
          + M\delta^8\ d^{-4} \underset{j}{\max}\Vert\Tilde{\mathrm{E}}_j\Vert_{L^2(B_j)} 
          \\ &\lesssim \delta^{6-6\lambda}\ \underset{j}{\max}\Vert \Tilde{\mathrm{E}}_j\Vert_{L^2(B_j)}  
          + \delta^{8-7\lambda}\ \underset{j}{\max}\Vert \Tilde{\mathrm{E}}_j\Vert_{L^2(B_j)}.
    \end{align}
Finally, using the above estimate (\ref{hdiverr1}) in (\ref{errhdiv}), we obtain
    \begin{align}\label{main-hdiv}
          \Vert \overset{1}{\mathbb{P}}(\Tilde{\mathrm{E}}_i)\Vert^2_{L^2(B_i)} \lesssim  \Vert \overset{1}{\mathbb{P}}(\Tilde{\mathrm{E}}^\textit{in})\Vert^2_{L^2(B_i)} + \delta^{\max\{ 4, 6-6\lambda, 8-7\lambda\}}\ \underset{j}{\max}\Vert \Tilde{\mathrm{E}}_j\Vert_{L^2(B_j)}.
    \end{align}
    
             
    \paragraph{\textit{2. Projection on $\mathbb{H}_{0}(\textit{curl}\ 0,D_i).$}}
Using integration by parts and the Maxwell model (\ref{Maxwell-model}), we first find that
    \begin{align}\nonumber
          \int_{D_i} E_i \cdot e_n^{(2)} 
          &= \frac{1}{k^2} \int_{D_j} \delta_j^{-1} \, \text{curl} \, \text{curl} \, E_j \cdot e_n^{(2)} \nonumber \\
         &= \frac{1}{k^2} \int_{D_j} \delta_i^{-1} \, \text{curl} \, E_i \cdot \text{curl} \big(e_n^{(2)}\big) \, dx + \frac{1}{k^2} \int_{\partial D_i} \delta_i^{-1} \, \big(\text{curl} \, E_i \times e_n^{(2)}\big) \cdot \nu \, dS, \nonumber
    \end{align}
which, due to the fact that \(\text{curl} \, e_n^{(2)} = 0\) and \(e_n^{(2)} \times \nu= 0\) on \(\partial D_i\), implies that
    \begin{align}\label{main-hcurl}
         \overset{2}{\mathbb{P}}(\Tilde{\mathrm{E}}_i) = 0, \quad \text{for} \; i = 1,2,\ldots, M.
    \end{align}
    
              
    \paragraph{\textit{3. Projection on $\nabla \mathbb{H}_{\textit{arm}}.$}} 
We note that the Magnetization operator can be decomposed as follows
    \begin{align}
          \mathbb{M}^{(\kappa)}\big[\mathrm{E}\big](\mathrm{x}) 
          \nonumber&=  \mathbb{M}^{(\mathrm{0})}\big[\mathrm{E}\big](\mathrm{x}) 
          + \frac{k^2}{2} \mathbb{N}^{(\mathrm{0})}\big[\mathrm{E}\big](\mathrm{x}) 
          +  \frac{ik^3}{12\pi}\int_{D} \mathrm{E}(\mathrm{y})d\mathrm{y} 
          - \frac{k^2}{2} \int_D \mathbb{G}^{(\mathrm{0})}(\mathrm{x},\mathrm{y})\frac{\mathrm{A}(\mathrm{x},\mathrm{y})\cdot\mathrm{E}(\mathrm{y})}{\Vert \mathrm{x}-\mathrm{y}\Vert^2}d\mathrm{y}
         \\ \nonumber&- \frac{1}{4\pi}\sum_{\mathrm{j}\ge 3} \frac{(ik)^{\mathrm{j}+1}}{(\mathrm{j}+1)!} \int_D\underset{\mathrm{x}}{\textit{Hess}}(\Vert \mathrm{x}-\mathrm{y}\Vert^\mathrm{j})\cdot \mathrm{E}(\mathrm{y})d\mathrm{y}, 
    \end{align}
where $\mathrm{A}(\mathrm{x},\mathrm{y}):= (\mathrm{x}-\mathrm{y})\otimes (\mathrm{x}-\mathrm{y})$. Using this decomposition, we can rewrite the Lippmann-Schwinger equation (\ref{Maxwell-ls}) as follows
    \begin{align}\label{maxwell-ls2}
           \mathrm{E}_i(\mathrm{x}) + \eta \;\mathbb{M}_{D_i}^{(0)}\big[\mathrm{E}_i\big](\mathrm{x}) 
           = \mathrm{E}^\textit{in}(\mathrm{x}) 
           + \eta \sum_{j\ne m}^M\Big(-\mathbb{M}_{D_j}^{(\kappa)}\big[\mathrm{E}_j\big](\mathrm{x})  
           + k^2\mathbb{N}_{D_j}^{(\kappa)}\Big)\big[\mathrm{E}_i\big](\mathrm{x}) 
           + k^2\eta \; \mathbb{N}_{D_j}^{(\kappa)}\big[\mathrm{E}_i\big](\mathrm{x})
           - \textit{Err}_{1,D_i},
    \end{align}
where,
    \begin{align}
          \textit{Err}_{1,D_i} &\nonumber:= \frac{k^2}{2} \mathbb{N}_{D_i}^{(\mathrm{0})}\big[\mathrm{E}\big](\mathrm{x}) +  \frac{i k^3}{12\pi}\int_{D_i} \mathrm{E}(\mathrm{y})d\mathrm{y} \textcolor{blue}{-} \frac{k^2}{2} \int_{D_i} \mathbb{G}^{(\mathrm{0})}(\mathrm{x},\mathrm{y})\frac{\mathrm{A}(\mathrm{x},\mathrm{y})\cdot\mathrm{E}(\mathrm{y})}{\Vert \mathrm{x}-\mathrm{y}\Vert^2}d\mathrm{y}
          \\&- \frac{1}{4\pi}\sum_{\mathrm{j}\ge 3} \frac{(i k)^{\mathrm{j}+1}}{(\mathrm{j}+1)!} \int_{D_i}\underset{\mathrm{x}}{\textit{Hess}}(\Vert \mathrm{x}-\mathrm{y}\Vert^\mathrm{j})\cdot \mathrm{E}(\mathrm{y})d\mathrm{y}.
    \end{align}
Next, we scale the domain $D_i$ to $B_i$, leading to the following expression for the error term
    \begin{align}\label{expression1}
          \widetilde{\textit{Err}}_{1,B_i} 
          \nonumber&= \frac{k^2}{2} \delta^2 \mathbb{N}_{B_i}^{(\mathrm{0})}\big[\Tilde{\mathrm{E}}_i\big](\mathrm{x}) 
          + \frac{ik^3}{12\pi}\delta^3\int_{B_i} \Tilde{\mathrm{E}}_i(\xi)d\xi 
          - \frac{k^2}{2} \delta^2\int_{B_i} \mathbb{G}^{(\mathrm{0})}(\varsigma,\xi)\frac{\mathrm{A}(\varsigma,\xi)\cdot\Tilde{\mathrm{E}}_i(\xi)}{\Vert \varsigma-\xi\Vert^2}d\xi
          \\ &- \frac{1}{4\pi}\sum_{\mathrm{j}\ge 3} \frac{(ik)^{\mathrm{j}+1}}{(\mathrm{j}+1)!} \delta^{j+1}\int_{B_i}\underset{\varsigma}{\textit{Hess}}(\Vert \varsigma-\xi\Vert^\mathrm{j})\cdot \Tilde{\mathrm{E}}_i(\xi)d\xi.
    \end{align}
We also recall that for $x\in D_i$
\begin{align}
          \Big(-\mathbb{M}_{D_j}^{(\kappa)}+ k^2\mathbb{N}_{D_j}^{(\kappa)}\Big)\big[\mathrm{E}_j\big](\mathrm{x}) = \int_{D_j} \bm{\Upsilon}^{(k)}(x,y)\ E_j(y) dy,
    \end{align}
where, $\bm{\Upsilon}^{(k)}(x,y) := \underset{x}{\textit{Hess}}\mathcal{G}^{(k)}(x,y) + k^2\mathcal{G}^{(k)}(x,y)\mathbb{I}$ represents the corresponding dyadic Green’s function, and $\mathcal{G}^{(k)}(x,y)$ is the Green’s function for the Helmholtz operator. Consequently, we can deduce the following relation
    \begin{align}
          E_i + \eta\mathbb{M}_{D_i}^{(0)}\big[E_i\big](x) 
          &\nonumber= E^\textit{in} 
          + \eta\sum_{j\ne i}^M\bm{\Upsilon}^{(k)}(z_i,z_j) \int_{D_j} E_j(y) dy 
          \\ \nonumber&+\eta\sum_{j\ne i}^M\int_0^1\underset{x}{\nabla}\bm{\Upsilon}^{(k)}(z_i+\theta(x-z_i),z_j)\cdot (x-z_i)d\theta\int_{D_j} E_j(y) dy
          \\ \nonumber&+ \eta\sum_{j\ne i}^M \int_{D_j} \int_0^1\underset{y}{\nabla}\bm{\Upsilon}^{(k)}(z_i,z_j+\theta(x-z_j))\cdot (y-z_j)E_j(y) d\theta dy
          + k^2\eta \; \mathbb{N}_{D_i}^{(\kappa)}\big[E_i\big](\mathrm{x})
           \\ \nonumber &- \textit{Err}_{1,D_i}.
    \end{align}    
Next, we take the projection of equation (\ref{maxwell-ls2}) onto $\nabla \mathbb{H}_{\textit{arm}}$, and after appropriate scaling, we obtain
    \begin{align}
          \big\langle\tilde{\mathrm{E}}_i,\Tilde{e}_n^{(3)}\big\rangle &\nonumber+ \eta \;\big\langle\mathbb{M}_{B_i}^{(0)}\big[\Tilde{\mathrm{E}}_i\big],\tilde{e}_n^{(3)}\big\rangle 
          =  \big\langle\tilde{\mathrm{E}}^\textit{in},\Tilde{e}_n^{(3)}\big\rangle
          + \eta \;\delta^3 \sum_{j\ne m}^M\big\langle\bm{\Upsilon}^{(k)}(z_i,z_j)\cdot\int_{B_j}\Tilde{\mathrm{E}}_j(\xi)d\xi,\tilde{e}_n^{(3)}\big\rangle 
          \\ \nonumber&+ \eta \;\delta^4 \sum_{j\ne m}^M\big\langle\int_0^1 \underset{x}{\nabla}\bm{\Upsilon}^{(k)}(z_i+\theta\delta\varsigma,z_j)\cdot \varsigma d\theta\int_{B_j}\Tilde{\mathrm{E}}_j(\xi)d\xi,\tilde{e}_n^{(3)}\big\rangle
          \\ &\nonumber+ \eta \;\delta^4 \sum_{j\ne m}^M\big\langle\int_{B_j}\int_0^1 \underset{y}{\nabla}\bm{\Upsilon}^{(k)}(z_m,z_j+\theta\delta\xi)\cdot \xi\Tilde{\mathrm{E}}_j)(\xi)\ d\theta d\xi,\tilde{e}_n^{(3)}\big\rangle
          + k^2\eta \; \delta^2\big\langle \mathbb{N}_{B_i}^{(\kappa)}\big[\overset{3}{\mathbb{P}}(\Tilde{\mathrm{E}}_i)\big](\mathrm{x}), \tilde{e}_n^{(3)}\big\rangle 
          \\ &- \big\langle\widetilde{\textit{Err}}_{1,B_i}, \tilde{e}_n^{(3)}\big\rangle.
    \end{align}
Furthermore, by solving the dispersion equation $f_{n_0}(\omega,\gamma):= 1+\eta\lambda_n^{(3)}=0$, we establish a property based on the choice of the incident frequency and the Lorentz model, which yields the following condition for $\mathrm{h}>0$
    \begin{align}\nonumber
        \big|1 + \eta\lambda^{(3)}_{n}\big| \sim 
           \begin{cases}
                \delta^\mathrm{h} & n = n_0 \\
                1 & n \ne n_0.
           \end{cases}   
    \end{align}
Based on the discussion in \cite[Subsection 2.1]{Arpan-Sini}, we utilize the property that
    \begin{align}\nonumber
          \Big\Vert \big(\mathbb{I}+\eta\mathbb{M}_{B_i}^{(0)}\Big)^{-1}\Big\Vert_{\mathbcal{L}\big(L^2(B_i);L^2(B_i)\big)} \lesssim \delta^{-h}.
    \end{align}
The sub-space $\nabla \mathbb{H}_{\textit{arm}}$ is an invariant subspace of the operator $\mathbb{M}_{B_i}^{(0)}$, where it induces a complete orthonormal basis $\big(\lambda^{(3)}_{\mathrm{n}},\Tilde{\mathrm{e}}^{(3)}_{\mathrm{n}}\big)_{\mathrm{n} \in \mathbb{N}}.$ After taking the squared modulus and summing over $n$, we deduce that
    \begin{align}\label{main-gradh}
          \big\Vert \overset{3}{\mathbb{P}}(\Tilde{\mathrm{E}}_i)\big\Vert^2_{\mathbb{L}^2(\mathrm{B}_i)} 
          &\nonumber\lesssim \delta^{-2h}\Bigg(\big\Vert \overset{3}{\mathbb{P}}(\Tilde{\mathrm{E}}^\textit{in})\big\Vert^2_{\mathbb{L}^2(\mathrm{B}_i)} 
          + \delta^6 \sum_{j\ne i}^M\Big\Vert\bm{\Upsilon}^{(k)}(z_i,z_j)\cdot\int_{B_j}\Tilde{\mathrm{E}}_j(\xi)d\xi\Big\Vert^2_{L^2(B_i)} 
          \\ \nonumber&+ \delta^6 \sum_{j\ne m}^M\big|\bm{\Upsilon}^{(k)}(z_m,z_j)\cdot\int_{B_j}\Tilde{\mathrm{E}}_j)(\xi)d\xi\big| \sum_{r>j,r\ne j}^M\big|\bm{\Upsilon}^{(k)}(z_i,z_r)\cdot\int_{B_j}\Tilde{\mathrm{E}}_j(\xi)d\xi\big|
          \\ \nonumber&+ M\delta^8 \sum_{j\ne i}^M\Big\Vert\int_0^1 \underset{\varsigma}{\nabla}\bm{\Upsilon}^{(k)}(z_i+\theta\delta\varsigma,z_j)d\theta\cdot \varsigma\int_{B_j}\Tilde{\mathrm{E}}_j(\xi)\  d\xi\Big\Vert_{L^2(B_i)}^2 
          \\ \nonumber&+ M\delta^8 \sum_{j\ne i}^M\Big\Vert\int_{B_j}\int_0^1 \underset{\xi}{\nabla}\bm{\Upsilon}^{(k)}(z_i,z_j+\theta\delta\xi)\cdot \xi\Tilde{\mathrm{E}}_j(\xi)\ d\theta d\xi\Big\Vert_{L^2(B_i)}^2 
          + \delta^4 \Vert \mathbb{N}_{B_i}^{(\kappa)}\big[\Tilde{\mathrm{E}}_i)\big]\Vert^2_{L^2(B_i)}
          \\ &+ \Vert \widetilde{\textit{Err}}_{1,B_i}\Vert^2_{L^2(B_i)}\Bigg).
    \end{align}
We can express the contributions from various terms. For instance, using the fact that $\bm{\Upsilon}(z_i,z_j) \simeq d_{ij}^{-3},$ we deduce
    \begin{align}\label{combine-1}
          \delta^6 \sum_{j\ne i}^M\Big\Vert\bm{\Upsilon}^{(k)}(z_i,z_j)\cdot\int_{B_j}\Tilde{\mathrm{E}}_j(\xi)d\xi\Big\Vert^2_{L^2(B_i)} 
          \lesssim \sum_{j\ne i}^M d_{ij}^{-6} \Vert \Tilde{E}_j\Vert_{L^2(B_j)} \lesssim d^{-6}\underset{j}{\max}\Vert\Tilde{E}_j\Vert_{L^2(B_j)}.
    \end{align}
Moreover, for the mixed term, we have
    \begin{align}\label{combine-2}
           &\nonumber\delta^6 \sum_{j\ne m}^M\big|\bm{\Upsilon}^{(k)}(z_m,z_j)\cdot\int_{B_j}\Tilde{\mathrm{E}}_j(\xi)d\xi\big| \sum_{r>j,r\ne j}^M\big|\bm{\Upsilon}^{(k)}(z_i,z_r)\cdot\int_{B_j}\Tilde{\mathrm{E}}_j(\xi)d\xi\big|
           \\ &\le \delta^6\sum_{j\ne i}^M d_{ij}^{-6} \Vert \Tilde{E}_j\Vert_{L^2(B_j)} \sum_{j\ne i}^M d_{ij}^{-6} \Vert \Tilde{E}_j\Vert_{L^2(B_j)} \lesssim \delta^6 d^{-6}|\log(\delta)|^2 \underset{j}{\max}\Vert\Tilde{E}_j)\Vert_{L^2(B_j)}.
     \end{align}
We now, estimate the following term.
\begin{align}\label{term-1-mul}
    & \nonumber M\delta^{8-2h} \sum_{j\ne i}^M\Big\Vert\int_0^1 \underset{\varsigma}{\nabla}\bm{\Upsilon}^{(k)}(z_i+\theta\delta\varsigma,z_j)d\theta\cdot \varsigma\int_{B_j}\Tilde{\mathrm{E}}_j(\xi)\  d\xi\Big\Vert_{L^2(B_i)}^2 
    \\ &\lesssim M\delta^{8-2h} \sum_{j\ne i}^Md_{ij}^{-8}\Big\Vert\Tilde{\mathrm{E}}_j\Big\Vert_{L^2(B_i)}^2 \lesssim \delta^{8-11\lambda-2h}\underset{j}{\max}\big\Vert \Tilde{\mathrm{E}}_j\big\Vert^2_{\mathbb{L}^2(\mathrm{B}_j)}.
\end{align}
In a similar way, we have
\begin{align}\label{term-2-mul}
    M\delta^8 \sum_{j\ne i}^M\Big\Vert\int_{B_j}\int_0^1 \underset{\xi}{\nabla}\bm{\Upsilon}^{(k)}(z_i,z_j+\theta\delta\xi)\cdot \xi\Tilde{\mathrm{E}}_j(\xi)\ d\theta d\xi\Big\Vert_{L^2(B_i)}^2 \lesssim \delta^{8-11\lambda-2h}\underset{j}{\max}\big\Vert \Tilde{\mathrm{E}}_j\big\Vert^2_{\mathbb{L}^2(\mathrm{B}_j)}.
\end{align}
Considering the expression (\ref{expression1}), we observe that the dominating term is $\frac{k^2}{2} \delta^2 \mathbb{N}_{B_i}^{(\mathrm{0})}\big[\Tilde{\mathrm{E}}_i\big](\mathrm{x}).$ Therefore, neglecting the other higher order terms, we take the $\mathrm{L}^2(B_i)$-norm, and utilizing the continuity of the Newtonian operator, we obtain
    \begin{align}\label{gradherr2}
           \Vert \widetilde{\textit{Err}}_{1,B_i}\Vert^2_{L^2(B_i)} \lesssim \delta^4\big\Vert \Tilde{\mathrm{E}}_j\big\Vert^2_{\mathbb{L}^2(\mathrm{B}_j)}.
    \end{align}
It is noteworthy that $\nabla\bm{\Upsilon}^{(\mathrm{0})}(z_m,z_j)\big| \simeq \frac{1}{d^4_{mj}}.$ 
Subsequently, by leveraging the continuity of the Newtonian operator and combining the estimates from (\ref{combine-1}), (\ref{combine-2}), (\ref{term-1-mul}), (\ref{term-2-mul}), (\ref{gradherr2}), and we substitute these findings into (\ref{main-gradh}), leading us to the conclusion that
    \begin{align}\label{main-algebra}
          \big\Vert \overset{3}{\mathbb{P}}(\Tilde{\mathrm{E}}_i)\big\Vert^2_{\mathbb{L}^2(\mathrm{B}_i)} 
          &\nonumber\lesssim \delta^{-2h}\big\Vert \overset{3}{\mathbb{P}}(\Tilde{\mathrm{E}}^\textit{in})\big\Vert^2_{\mathbb{L}^2(\mathrm{B}_i)} 
          + \delta^{8-11\lambda-2h}\underset{j}{\max}\big\Vert \Tilde{\mathrm{E}}_j\big\Vert^2_{\mathbb{L}^2(\mathrm{B}_j)}
          \\  &+ \delta^6 \underset{j}{\max}\Vert\Tilde{E}_j\Vert_{L^2(B_j)} + \delta^6 d^{-6}|\log(\delta)|^2 \underset{j}{\max}\Vert\Tilde{E}_j\Vert^2_{L^2(B_j)} + \delta^{4-2h} \Vert \Tilde{\mathrm{E}}_j\big\Vert^2_{\mathbb{L}^2(\mathrm{B}_j)}.
    \end{align}
Now, we use Parseval's identity to estimate $\Tilde{E}_i$, i.e. we write
$\big\Vert \Tilde{E}_i\big\Vert^2_{\mathbb{L}^2(\mathrm{B}_i)} = \sum\limits_{\mathrm{j}=1}^3 \big\Vert \overset{\mathrm{j}}{\mathbb{P}}(\Tilde{E}_i)\big\Vert^2_{\mathbb{L}^2(\mathrm{B}_i)}.$ Subsequently, due to the estimate (\ref{main-hdiv}), (\ref{main-hcurl}) and (\ref{main-algebra}), we derive that
    \begin{align}
          \underset{i}{\max}\Vert\Tilde{E}_i\Vert^2_{L^2(B_i)} 
           &\nonumber\lesssim \delta^{-2h}\big\Vert \Tilde{\mathrm{E}}^\textit{in}\big\Vert^2_{\mathbb{L}^2(\mathrm{B}_i)} 
           + \delta^{\{\max(4, 4-2h, 8-11\lambda-2h, 6-6\lambda, 8-7\lambda \}} \underset{j}{\max}\Vert\Tilde{E}_j\Vert^2_{L^2(B_j)}
          \\  \nonumber&+ \delta^6 \underset{j}{\max}\Vert\Tilde{E}_j\Vert_{L^2(B_j)} 
          + \delta^6 d^{-6}|\log(\delta)|^2 \underset{j}{\max}\Vert\Tilde{E}_j\Vert^2_{L^2(B_j)}
    \end{align}

\noindent
Thus, according to the invertibility condition of the linear algebraic system, it follows that $\lambda \le 1-\frac{h}{3}.$ Looking at the above expression, we examine the case where
    \begin{align}\nonumber
          \frac{9}{5}<h<2,
    \end{align}
and therefore, we derive
    \begin{align}\nonumber
         \underset{i}{\max}\big\Vert \Tilde{\mathrm{E}}_i\big\Vert^2_{\mathbb{L}^2(\mathrm{B}_i)} \lesssim \delta^{-2h}\big\Vert \Tilde{\mathrm{E}}^\textit{in}\big\Vert^2_{\mathbb{L}^2(\mathrm{B}_i)},\; \text{which implies,}\; \underset{i}{\max}\big\Vert \mathrm{E}_i\big\Vert^2_{\mathbb{L}^2(D_i)} \lesssim \delta^{3-2h}.
    \end{align}
This concludes the proof of the initial estimate stated in Proposition \ref{proposition-maxwell-1}.\\
\\
\subsubsection*{Part 2: Proof of the Estimate for $\big\Vert \mathrm{E}\big\Vert_{\mathbb{L}^\mathrm{4}(D)}$ in Proposition \ref{proposition-maxwell-1}}
We start by recalling the Lippmann-Schwinger (LS) equation, which provides the solution to the electromagnetic scattering problem (\ref{Maxwell-model}) for \( i=1,2,\ldots,M \)
    \begin{align} \label{Maxwell-ls-new}
          \mathrm{E}_i(\mathrm{x}) + \eta \sum_{i=1}^M \mathbb{M}_{D_i}^{(k)}\big[\mathrm{E}_i\big](\mathrm{x}) - k^2  \eta \sum_{i=1}^M \mathbb{N}_{D_i}^{(k)}\big[\mathrm{E}_i\big](\mathrm{x}) = \mathrm{E}^\textit{in}(\mathrm{x}), \quad \text{for} \ x \in D := \bigcup_{i=1}^M D_i,
    \end{align} 
Now, with integration by parts and as $\nabla\cdot \mathrm{E}_i=0$, we show that $\mathbb{M}_{D_i}^{(k)}\big[\mathrm{E}_i\big] = \nabla \mathcal{S}_{\partial D_i}^{(k)}\big[\nu\cdot\mathrm{E}_i\big]$, where $\mathcal{S}_{\partial D_i}^{(k)}$ is the Single-Layer operator defined from $\mathbb L^2(\partial D_i)$ to $\mathbb H^{\frac{3}{2}}(D_i)$, by
    \begin{align}\nonumber
           \mathcal{S}_{\partial D_i}^{(k)}\big[f\big](\cdot) := \int_{\partial D_i} \mathbb{G}^{(k)}(\cdot,\mathrm{y})\;f(\mathrm{y})d\sigma_\mathrm{y}.
    \end{align}
Therefore, with the help of the above identity, we rewrite the equation (\ref{Maxwell-ls-new}) for $x\in D_i$ as follows
    \begin{align}\label{equn-single}
        \mathrm{E}_i(\mathrm{x}) + \eta \nabla \mathcal{S}_{\partial D_i}^{(k)}\big[\nu\cdot\mathrm{E}_i\big](\mathrm{x}) + \eta \sum_{j\ne i}^M \nabla \mathcal{S}_{\partial D_j}^{(k)}\big[\nu\cdot\mathrm{E}_j\big](\mathrm{x}) - k^2  \eta \mathbb{N}_{D_i}^{(k)}\big[\mathrm{E}_i\big](\mathrm{x}) 
         - k^2  \eta \sum_{j\ne i}^M \mathbb{N}_{D_j}^{(k)}\big[\mathrm{E}_j\big](\mathrm{x}) 
        = \mathrm{E}^\textit{in}(\mathrm{x}).
    \end{align}
Scaling to $B$ and taking $\textit{curl}$ in above equation, we arrive at
    \begin{align}
         \textit{curl}(\Tilde{E}_i) = \textit{curl}(\Tilde{E}^\textit{in}) + k^2\eta\delta^2\textit{curl}\ \mathbb{N}_{B_i}^{(k\delta)}\big[\Tilde{\mathrm{E}}_i\big] 
         + \widetilde{\textit{Err}}_{\textit{f},1},
    \end{align}
where we define
    \begin{align}
          \widetilde{\textit{Err}}_{\textit{f},1} \nonumber&:= \delta^3\sum_{j\ne i}^M\mathbb{G}^{(k)}(z_i,z_j)\int_{B_j}\textit{curl}(\Tilde{\mathrm{E}}_j)(\xi)d\xi 
          + \delta^4\sum_{j\ne i}^M \int_0^1 \underset{\varsigma}{\nabla}\mathbb{G}^{(k)}(z_i+\theta\varsigma,z_j)\cdot \varsigma d\theta\int_{B_j}\textit{curl}(\Tilde{\mathrm{E}}_j)(\xi)d\xi
          \\ &+ \delta^4\sum_{j\ne i}^M \int_{B_j}\int_0^1 \underset{\xi}{\nabla}\mathbb{G}^{(k)}(z_i,z_j+\theta\xi)\cdot \xi\textit{curl}(\Tilde{\mathrm{E}}_j)(\xi)d\theta d\xi.
    \end{align}
Taking the $L^2(B_i)$-norm to both hand side of the above equation yields to
   \begin{align}
         \Vert \textit{curl}(\Tilde{E}_i) \Vert_{L^2(B_i)} 
         \nonumber&\lesssim \Vert \textit{curl}(\Tilde{E}^\textit{in}) \Vert_{L^2(B_i)} + \delta^2\Vert \textit{curl}\ \mathbb{N}_{B_i}^{(k\delta)}\big[\Tilde{\mathrm{E}}_i\big]\Vert_{L^2(B_i)} + \Vert \widetilde{\textit{Err}}_{\textit{f},1} \Vert_{L^2(B_i)}.
   \end{align}
Next, we use the continuity of Newtonian operator, the fact that $|\mathbb{G}^{(k)}(z_i,z_j)| \sim d_{ij}^{-1}$ and $\nabla\mathbb{G}^{(k)}(z_i,z_j) \sim d_{ij}^{-2}$ to deduce the following  
\begin{align}
    \Big(1-\delta^{\max\{3-3\lambda, 4-3\lambda\}}\Big) \Vert \textit{curl}(\Tilde{E}_i) \Vert_{L^2(B_i)} \lesssim \Vert \textit{curl}(\Tilde{E}^\textit{in}) \Vert_{L^2(B_i)} + \delta^2\Vert \Tilde{E}_i \Vert_{L^2(B_i)}.
\end{align}
Due to the a-priori estimate $\Vert \Tilde{E}_i \Vert_{L^2(B_i)} \sim \delta^{-h},$ and based on the chosen regime
\begin{align}\nonumber
    \lambda \le 1-\frac{h}{3}, \; \text{with}\; \frac{9}{5}<h<2,
\end{align}
it implies that
\begin{align}\label{curl}
    \Vert \textit{curl}(\Tilde{E}_i) \Vert_{L^2(B_i)} \lesssim 1.
\end{align}
We also have the following estimate
\begin{align}\label{nu}
    \big\Vert \nu\cdot\Tilde{\mathrm{E}}_i\big\Vert_{H^{-\frac{1}{2}}(\partial B_i)} \nonumber &\lesssim \big\Vert \Tilde{\mathrm{E}}_i\big\Vert_{H(\textit{curl},B_i)}
    \nonumber\\ &\lesssim \Big(\big\Vert \Tilde{\mathrm{E}}\big\Vert_{\mathbb{L}^2(\mathrm{B})}^2 + \big\Vert \textit{curl}\;\Tilde{\mathrm{E}}_i\big\Vert_{\mathbb{L}^2(B_i)}^2\Big)^{\frac{1}{2}} \sim \delta^{-\mathrm{h}}.
\end{align}
To write the closed system of equations, we need to take the normal trace of the equation (\ref{equn-single}) in $\mathrm{B}$ and  use the jump relation of the single-layer operator to get
\begin{align}\label{afore}
     \big(1+\frac{\eta}{2}\big)\nu\cdot\Tilde{\mathrm{E}}_i = -\eta\mathcal{K}_{\partial B_i}^*(\nu\cdot\Tilde{\mathrm{E}}_i) + k^2\eta\delta^2 \ \nu\cdot \mathbb{N}_{B_i}^{(k\delta)}\big[\Tilde{\mathrm{E}}_i\big] 
      + \eta\widetilde{\textit{Err}}_{\textit{f},3}   + k^2\eta\ \nu\cdot\widetilde{\textit{Err}}_{\textit{f},2} +  \nu \cdot \Tilde{\mathrm{E}}^\textit{in},
\end{align}
where, $\mathcal{K}_{\partial B_i}^*$ is the Neumann-Poincar\'e operator defined by
\begin{align}\nonumber
    \mathcal{K}^*_{\partial B_i}\big[f\big](x) := \text{p.v.}\int_{\partial B_i} \partial_{\nu_\mathrm{x}}\mathbb{G}^{(\kappa)}(\cdot,\mathrm{y})\;f(\mathrm{y})d\sigma_\mathrm{y}.
\end{align}
We also define
    \begin{align}
          \widetilde{\textit{Err}}_{\textit{f},3} 
          \nonumber&:= - \sum_{j\ne i}^M \partial_\nu \mathcal{S}_{\partial D_j}^{(k)}\big[\nu\cdot\mathrm{E}_j\big]
          \\ \nonumber&= -\eta \delta^2\Bigg(\sum_{j\ne i}^M\partial_\nu\mathbb{G}^{(k)}(z_i,z_j)\int_{\partial B_j}\nu\cdot\Tilde{\mathrm{E}}_j(\xi)d\sigma_\xi
          + \delta^4\sum_{j\ne i}^M \int_0^1 \underset{\varsigma}{\nabla}\partial_\nu\mathbb{G}^{(k)}(z_i+\theta\varsigma,z_j)\cdot \varsigma d\theta\int_{\partial B_j}\nu\cdot\Tilde{\mathrm{E}}_j(\xi)d\sigma_\xi
          \\ &+ \delta^4\sum_{j\ne i}^M \int_{\partial B_j}\int_0^1 \underset{\xi}{\nabla}\partial_\nu\mathbb{G}^{(k)}(z_i,z_j+\theta\xi)\cdot \xi\nu\cdot\Tilde{\mathrm{E}}_j(\xi)d\theta d\sigma_\xi\Bigg)
    \end{align}
and
    \begin{align}
          \widetilde{\textit{Err}}_{\textit{f},2} 
          \nonumber&:= \delta^3\sum_{j\ne i}^M\mathbb{G}^{(k)}(z_i,z_j)\int_{B_j}\Tilde{\mathrm{E}}_j(\xi)d\xi 
          + \delta^4\sum_{j\ne i}^M \int_0^1 \underset{\varsigma}{\nabla}\mathbb{G}^{(k)}(z_i+\theta\varsigma,z_j)\cdot \varsigma d\theta\int_{B_j}\Tilde{\mathrm{E}}_j(\xi)d\xi
          \\ &+ \delta^4\sum_{j\ne i}^M \int_{B_j}\int_0^1 \underset{\xi}{\nabla}\mathbb{G}^{(k)}(z_i,z_j+\theta\xi)\cdot \xi\Tilde{\mathrm{E}}_j(\xi)d\theta d\xi
    \end{align}
Consequently, taking the $H^{\frac{1}{2}}(\partial B_i)$-norm on both sides of the aforementioned equation (\ref{afore}), we see that
    \begin{align}
          \Vert \nu\cdot\Tilde{E}_i \Vert_{H^{\frac{1}{2}}(\partial B_i)} 
          &\nonumber\lesssim \Big|\frac{1}{1+\eta}\Big| \Vert \nu \cdot \Tilde{\mathrm{E}}^\textit{in} \Vert_{H^{\frac{1}{2}}(\partial B_i)} + \Big|\frac{\eta}{1+\eta}\Big| \Vert \mathcal{K}_{\partial B_i}^*(\nu\cdot\Tilde{\mathrm{E}}_i) \Vert_{H^{\frac{1}{2}}(\partial B_i)} 
          + \delta^2\Big|\frac{\eta}{1+\eta}\Big| \Vert \nu\cdot \mathbb{N}_{B_i}^{(k\delta)}\big[\Tilde{\mathrm{E}}_i\big] \Vert_{H^{\frac{1}{2}}(\partial B_i)} 
          \\ &+ \Big|\frac{\eta}{1+\eta}\Big| \Vert \widetilde{\textit{Err}}_{\textit{f},3} \Vert_{H^{\frac{1}{2}}(\partial B_i)} 
          + \Big|\frac{\eta}{1+\eta}\Big| \Vert \widetilde{\textit{Err}}_{\textit{f},2} \Vert_{H^{\frac{1}{2}}(\partial B_i)}.
    \end{align}
We know that for the case when $\partial B_i$ is $\mathcal{C}^2$-regular, the Neumann-Poincar\'e operator $\mathcal{K}_{\partial B_i}^*$ is continuous from $H^{-\frac{1}{2}}(\partial B_i) \to H^{\frac{1}{2}}(\partial B_i).$ Then, using the continuity of the Newtonian operator and due to the estimate (\ref{nu}), we arrive at the following reduced expression
    \begin{align}
           \Vert \nu\cdot\Tilde{E}_i \Vert_{H^{\frac{1}{2}}(\partial B_i)} \lesssim 1 + \delta^{-h} + \delta^{2-h} + \Vert \widetilde{\textit{Err}}_{\textit{f},3} \Vert_{H^{\frac{1}{2}}(\partial B_i)} + \Vert \widetilde{\textit{Err}}_{\textit{f},3} \Vert_{H^{\frac{1}{2}}(\partial B_i)}.
    \end{align}
Then, we use the fact that $|\partial_\nu\mathbb{G}^{(k)}(z_i,z_j)| \sim d_{ij}^{-2}$ and $\nabla\partial_\nu\mathbb{G}^{(k)}(z_i,z_j) \sim d_{ij}^{-3}$ to deduce the following
    \begin{align}\nonumber
          \Big(1- \delta^{2-3\lambda}\Big) \Vert \nu\cdot\Tilde{E}_i \Vert_{H^{\frac{1}{2}}(\partial B_i)} \lesssim \delta^{-h} + \delta^{3-3\lambda-h}.
    \end{align}
Therefore, based on the chosen regime
    \begin{align}\nonumber
          \lambda \le 1-\frac{h}{3}, \; \text{with}\; \frac{9}{5}<h<2,
    \end{align} 
it follows from the previous estimate that
    \begin{align}
          \Vert \nu\cdot\Tilde{E}_i \Vert_{H^{\frac{1}{2}}(\partial B_i)} \lesssim \delta^{-h}.
    \end{align}
Then, based on the inequality as shown in \cite[ineq. 1.4]{ineq}, we deduce
\begin{align}\label{2.18}
    \big\Vert \Tilde{\mathrm{E}}_i\big\Vert_{H^{1}(B)} \lesssim \underbrace{\big\Vert \Tilde{\mathrm{E}}_i\big\Vert_{L^{2}(B)}}_{\sim \; \delta^{-\mathrm{h}}} + \underbrace{\big\Vert \textit{curl}(\tilde{\mathrm{E}}_i)\big\Vert_{L^2(B)}}_{\sim \; 1} + \underbrace{\big\Vert \text{div}\;\tilde{\mathrm{E}}\big\Vert_{L^2(B)}}_{= \; 0} + \big\Vert \nu\cdot\Tilde{\mathrm{E}}\big\Vert_{H^{\frac{1}{2}}(\partial B)} \sim \delta^{-\mathrm{h}}.
\end{align}
We have the following estimate based on Gagliardo-Nirenberg inequality, estimate $\Vert \Tilde{E}_i \Vert_{L^2(B_i)} \sim \delta^{-h},$ and using (\ref{2.18})
\begin{align}\nonumber
\big\Vert \Tilde{\mathrm{E}}_i\big\Vert_{L^\mathrm{4}(\mathrm{B})} \nonumber &\lesssim \big\Vert \Tilde{\mathrm{E}}_i\big\Vert_{L^\mathrm{2}(B_i)}^{\frac{1}{2}}\big\Vert \Tilde{\mathrm{E}}_i\big\Vert_{H^1(B)}^\frac{1}{2}
\\ &\lesssim \delta^{-\frac{\mathrm{h}}{2}} \cdot \delta^{-\frac{\mathrm{h}}{2}} \sim \delta^{-\mathrm{h}}.
\end{align}
So, using the aforementioned estimate and scaling it back to $D_i$, we arrive at
\begin{align}\nonumber
    \big\Vert E_i\big\Vert_{L^\mathrm{4}(D_i)} \sim \delta^{\frac{3}{4}-\mathrm{h}}.
\end{align}
This completes the proof of Proposition \ref{proposition-maxwell-1}.
\end{proof} 
\noindent
To complete the proof of the Lemma \ref{lemma-maxwell}, we first recall the derived linear algebraic system given by (\ref{first-al})
    \begin{align}
          \int_{D_i}\overset{3}{\mathbb{P}}(E_i)(y)\ dy 
          &\nonumber- \eta\sum_{j\ne i}^M \mathbcal{P}_{D_i} \cdot \bm{\Upsilon}^{(k)}(z_i,z_j)\int_{D_j} \overset{3}{\mathbb{P}}(E_j)(y) dy
          = \mathbcal{P}_{D_i}\cdot \mathrm{E}^\textit{in}(z_i) + \mathcal{O}\Big(\delta^{4-h}\Big)
          \\ &+  \mathcal{O}\Big(\delta^{\frac{11}{2}-h} d^{-4}\underset{j}{\max} \Vert \overset{3}{\mathbb{P}}(E_j)\Vert_{L^2(D_j)}\Big)
          +  \mathcal{O}\Big(\delta^{\frac{11}{2}-h} d^{-2}\underset{j}{\max} \Vert \overset{1}{\mathbb{P}}(E_j)\Vert_{L^2(D_j)}\Big).
    \end{align}
From the a priori estimates, we have
    \begin{align}\nonumber
          \underset{i}{\max}\big\Vert \mathrm{E}_i\big\Vert_{\mathbb{L}^2(D_i)} \lesssim \delta^{\frac{3}{2}-h}.
    \end{align}
Consequently, we obtain the following linear algebraic system 
\begin{align}\nonumber
          \mathbcal{Q}_i 
          -\eta\sum_{j\ne i}^M \mathbcal{P}_{D_i} \cdot \bm{\Upsilon}^{(k)}(z_i,z_j)\cdot\mathbcal{Q}_j 
          = \mathbcal{P}_{D_i}\cdot \mathrm{E}^\textit{in}(z_i) + \mathcal{O}\Big(\delta^{\min\{4-h;7-2h-4\lambda\}}\Big),\; \texttt{for}\ i=1,2,\ldots,M.
    \end{align}
Here, we define the polarization matrix $\mathbcal{P}_{D_i}$ as
    \begin{align}\nonumber
          \mathbcal{P}_{D_i} 
          = \delta^3\sum_{n}\frac{1}{1+\eta \lambda_n^{(3)}}\int_{B_i} e^{(3)}_n(\xi)d\xi\otimes \int_{B_i} e^{(3)}_n(\xi)d\xi.
    \end{align}
\noindent
The above algebraic system is invertible under the following condition, as discussed in \cite[Section 6.1]{cao-ahcene-sini-jlms}:
    \begin{align}\label{inv-cond}
         |\eta| \underset{j=1,2,\ldots,M}{\max}\Vert \mathbcal{P}_{D_i} \Vert\ d^{-3}<1.
    \end{align} 
This condition ensures the uniqueness and existence of solutions to the system, thereby completing the proof of the lemma.
\end{proof} 
Neglecting the error order term, let us now rewrite the above algebraic system, which will be useful in defining its continuous effective equation in the next section, as follows:
    \begin{align}\label{general-al-equn}
          \Tilde{\mathbcal{Q}}_i 
          - \eta\sum_{j\ne i}^M\bm{\Upsilon}^{(k)}(z_i,z_j)\cdot \mathbcal{P}_{D_j} \cdot \Tilde{\mathbcal{Q}}_j 
          = \mathrm{E}^\textit{in}(z_i),
    \end{align}
where, we set $ \Tilde{\mathbcal{Q}}_i:= \mathbcal{P}^{-1}_{D_i}\cdot\mathbcal{Q}_i.$ Then, due to the properties (\ref{frequency})–(\ref{condition3D}), \( \mathbcal{P}_{D_i} \) is the polarization matrix given by
\begin{align}\label{def-pol}
    \mathbcal{P}_{D_i} =  \frac{\delta^3}{1 + \eta \lambda_{n_0}^{(3)}}\sum_{m=1}^{m_{n_0}} \int_{B_i} e^{(3)}_{m,n_0}(B) \otimes \int_{B_i} e^{(3)}_{m,n_0}(B) + \mathcal{O}(\delta^3),
\end{align}
where \( \Tilde{e}^{(3)}_{m,n_0} \) representing the scalled eigenfunctions associated with the space \( \nabla \mathbb{H}_{\textit{arm}} \) on the domain \( B_i \) and \(\displaystyle e^{(3)}_{m,n_0}(B) = \frac{1}{\Vert \Tilde{e}_{m,n_0}^{(3)} \Vert_{L^2(B_i)}} \Tilde{e}_{m,n_0}^{(3)} \) representing the corresponding normalized functions. Subsequently, using the same properties (\ref{frequency})–(\ref{condition3D}), we can deduce that $\mathbcal{P}_{D_i} = \delta^{3-h}\mathbcal{P}_B$, where, we have $\mathbcal{P}_B$ as
\begin{align}
    \mathbcal{P}_B := C_B \sum_{m=1}^{m_{n_0}} \int_{B_i} e^{(3)}_{m,n_0}(B) \otimes \int_{B_i} e^{(3)}_{m,n_0}(B),\; \text{with}\ C_B = \Bigg(\lambda_{n_0}^{(3)}\cdot \dfrac{k_p^2k_{n_0}\sqrt{5}}{k_{n_0}^4 + (\zeta_{n_0}k_{n_0})^2}\Bigg)^{-1}.
\end{align} 
Let us now discuss about its corresponding effective equation in the next section.

    
    \subsubsection{The Effective Medium Theory for the Electromagnetic Problem}\label{effective-max-section} 
   
Now, we distribute the plasmonic nanoparticle periodically inside a given smooth domain $\mathbf{\Omega}$ with a period given by $\delta^{1-\frac{h}{3}}$, which means that the distance between close nanoparticle is of the order $\delta^{1-\frac{h}{3}}$. For simplicity, we take the nanoparticles having all the same contrasts. In this case, we show that the algebric system derived in (\ref{general-al-equn})-(\ref{def-pol}) along with the invertibility condition (\ref{inv-cond}) "converges" to the solution of the following kind of Lippmann-Schwinger equation
    \begin{equation}\label{eff-ls}
         \textit{\bm{\mathrm{E}}}_f (\mathrm{x}) 
         + \nabla\int_{\mathbf{\Omega}}\nabla\mathbcal{G}^{(k)}(x,y)\cdot \mathbcal{A}_B\cdot \textit{\bm{\mathrm{E}}}_f (y)dy 
         - k^2 \int_{\mathbf{\Omega}}\mathbcal{G}^{(k)}(x,y)\cdot \mathbcal{A}_B\cdot \textit{\bm{\mathrm{E}}}_f (y)dy
         = \mathrm{E}^\textit{in}(x).
    \end{equation}
Next, using the fact that
$$\bm{\Omega}= \Omega_i\cup \Big(\bigcup_{j\ne i}^{[d^{-3}]} \Omega_j\Big)\cup \Big(\bm{\Omega}\setminus\bigcup_{j=1}^{[d^{-3}]}\Omega_j\Big)\; \text{and}\; \textit{Vol}(\Omega_j) = d^3,$$
we express equation (\ref{eff-ls}) in its discretized form at $x=z_i$ as:
    \begin{align}
        \textit{\bm{\mathrm{E}}}_f (z_i) 
        \nonumber&+ \nabla\int_{\Omega_i}\nabla\mathbcal{G}^{(0)}(z_i,y)\cdot \mathcal{A}_B \cdot\textit{\bm{\mathrm{E}}}_f (y)\ dy 
        - \delta^{3-h}\sum_{j\ne i}^{[d^{-3}]}\mathbcal{A}_B\cdot\bm{\Upsilon}^{(k)}(z_i,z_j)\cdot \textit{\bm{\mathrm{E}}}_f (z_j) 
        = \mathrm{E}^\textit{in}(z_i)
        \\ \nonumber&+ \int_{\mathbf{\Omega}\setminus\bigcup\limits_{j=1}^{[d^{-3}]}\Omega_j}\bm{\Upsilon}^{(k)}(z_i,y)\cdot \mathbcal{A}_B\cdot \textit{\bm{\mathrm{E}}}_f (y)
         - \nabla\int_{\Omega_i}\nabla\mathbcal{G}^{(0)}(z_i,y)\cdot \mathcal{A}_B \cdot\big(\textit{\bm{\mathrm{E}}}_f (y)-\textit{\bm{\mathrm{E}}}_f (z_i)\big)dy
        \\ \nonumber&- \sum_{j\ne i}^{[d^{-3}]}\int_{\Omega_j}\Big(\bm{\Upsilon}^{(k)}(z_i,z_j)\cdot\mathcal{A}_B \cdot\textit{\bm{\mathrm{E}}}_f (z_j)-\bm{\Upsilon}^{(k)}(z_i,y)\cdot \mathcal{A}_B \cdot \textit{\bm{\mathrm{E}}}_f (y)\Big)dy 
        \\ &+ \int_{\Omega_i}\Big(\bm{\Upsilon}^{(k)}(z_i,y)-\bm{\Upsilon}^{(0)}(z_i,y)\Big)\cdot \mathcal{A}_B\cdot\textit{\bm{\mathrm{E}}}_f (y)dy,        
    \end{align}
which we rewrite as follows
    \begin{align}
        \nonumber&\widehat{\textit{\bm{\mathrm{E}}}}_f (z_i) 
        - \delta^{3-h}\sum_{j\ne i}^{[d^{-3}]}\bm{\Upsilon}^{(k)}(z_i,z_j)\cdot \mathbcal{P}_B\cdot \widehat{\textit{\bm{\mathrm{E}}}}_f (z_j) 
        = \mathrm{E}^\textit{in}(z_i)
        + \underbrace{\int_{\mathbf{\Omega}\setminus\bigcup\limits_{j=1}^{[d^{-3}]}\Omega_j}\bm{\Upsilon}^{(k)}(z_i,y)\cdot \mathbcal{P}_B\cdot\widehat{\textit{\bm{\mathrm{E}}}}_f (y)}_{:=\textit{Err}^{(1)}_i}
        \\ \nonumber&- \underbrace{\nabla\int_{\Omega_i}\nabla\mathbcal{G}^{(0)}(z_i,y)\cdot \mathbcal{P}_B\cdot\big(\widehat{\textit{\bm{\mathrm{E}}}}_f (y)-\widehat{\textit{\bm{\mathrm{E}}}}_f (z_i)\big)dy}_{:= \textit{Err}^{(2)}_i }
        \\ \nonumber&+ \underbrace{\sum_{j\ne i}^{[d^{-3}]}\int_{\Omega_j}\Big(\bm{\Upsilon}^{(k)}(z_i,z_j)\cdot\mathbcal{P}_B \cdot\widehat{\textit{\bm{\mathrm{E}}}}_f (z_j)-\bm{\Upsilon}^{(k)}(z_i,y)\cdot \mathbcal{P}_B \cdot \widehat{\textit{\bm{\mathrm{E}}}}_f (y)\Big)dy}_{:= \textit{Err}^{(3)}_i} 
        \\ &+ \underbrace{\int_{\Omega_i}\Big(\bm{\Upsilon}^{(k)}(z_i,y)-\bm{\Upsilon}^{(0)}(z_i,y)\Big)\cdot \mathbcal{P}_B \cdot \widehat{\textit{\bm{\mathrm{E}}}}_f (y)dy}_{\textit{Err}^{(4)}_i},
    \end{align}
where we impose $\displaystyle\mathbcal{P}_B := \Big(\mathbb{I} + \mathbcal{A}_B\cdot\nabla\int_{\Omega_i}\nabla\mathbcal{G}^{(0)}(z_i,y)dy\Big)^{-1} \cdot \mathbcal{A}_B$, and we set $\widehat{\textit{\bm{\mathrm{E}}}}_f := \mathbcal{P}_B^{-1}\cdot \mathbcal{A}_B\cdot\textit{\bm{\mathrm{E}}}_f.$ Therefore, we have
\begin{equation}\label{Pol-Matrix}
\mathbcal{A}_B=\Big(\mathbb{I}-\displaystyle\mathbcal{P}_B \cdot\nabla\int_{B}\nabla\mathbcal{G}^{(0)}(z_i,y)dy\Big)^{-1} \cdot \displaystyle\mathbcal{P}_B.
\end{equation}
We now proceed to estimate the fourth term of the above expression:
\begin{align}
    \textit{Err}^{(3)}_i
    &\nonumber:= \sum_{j\ne i}^{[d^{-3}]}\int_{\Omega_j}\Big(\bm{\Upsilon}^{(k)}(z_i,z_j)\cdot\mathbcal{P}_B \cdot\widehat{\textit{\bm{\mathrm{E}}}}_f (z_j)-\bm{\Upsilon}^{(k)}(z_i,y)\cdot \mathbcal{P}_B \cdot \widehat{\textit{\bm{\mathrm{E}}}}_f (y)\Big)dy
    \\ &\nonumber\le \sum_{j\ne i}^{[d^{-3}]}\int_{\Omega_i}\bm{\Upsilon}^{(k)}(z_i,z_j)\cdot \mathbcal{A}_B\cdot\big(\textit{\bm{\mathrm{E}}}_f (z_j) - \textit{\bm{\mathrm{E}}}_f (y)\big)
    + \sum_{j\ne i}^{[d^{-3}]}\int_{\Omega_i}\Big(\bm{\Upsilon}^{(k)}(z_i,z_j)-\bm{\Upsilon}^{(k)}(z_i,y)\Big)\cdot\mathbcal{A}_B\cdot\textit{\bm{\mathrm{E}}}_f (y)
    \\ &\nonumber\le \underbrace{\sum_{j\ne i}^{[d^{-3}]}\bm{\Upsilon}^{(k)}(z_i,z_j)\cdot\mathbcal{A}_B\cdot\int_{\Omega_i}|z_j-y|^\alpha\big[\textit{\bm{\mathrm{E}}}_f\big]_{C^{0,\alpha}(\overline{\bm{\Omega}})}}_{:= \textit{Err}_\textit{ef,i,1} }
    + \underbrace{\sum_{j\ne i}^{[d^{-3}]}\int_{\Omega_j}\Big(\bm{\Upsilon}^{(k)}(z_i,z_j)-\bm{\Upsilon}^{(k)}(z_i,y)\Big)\cdot \mathbcal{A}_B\cdot\textit{\bm{\mathrm{E}}}_f (y)}_{:= \textit{Err}_\textit{ef,i,2}}.
\end{align}
By applying Lemma \ref{counting} and leveraging the result from \cite[Section 5.1]{cao-sini}, along with the \(C^{0,\alpha}\)-regularity of \(\bm{E}_f\) for \(0 < \alpha < 1\), we can estimate the first error term as follows
\begin{align}
    |\textit{Err}_\textit{ef,i,1}| 
    &\nonumber= |\mathbcal{A}_B| \big[\textit{\bm{\mathrm{E}}}_f\big]_{C^{0,\alpha}(\overline{\bm{\Omega}})}\textit{Vol}(\Omega_j) d^\alpha \Big|\sum_{j\ne i}^{[d^{-3}]}\bm{\Upsilon}^{(k)}(z_i,z_j)\Big|
    \\ &\nonumber\lesssim |\mathbcal{A}_B| \big[\textit{\bm{\mathrm{E}}}_f\big]_{C^{0,\alpha}(\overline{\bm{\Omega}})}\textit{Vol}(\Omega_j) d^\alpha d^{-3}|\log(d)|.
\end{align}
Taking the square of this expression and summing over $i=1\ \text{to}\ [d^{-3}]$ yields the final estimation
\begin{align}\label{r1}
  \sum_{i=1}^{[d^{-3}]} |\textit{Err}_\textit{ef,i,1}|^2 
  \lesssim |\mathbcal{A}_B|^2 \big[\textit{\bm{\mathrm{E}}}_f\big]^2_{C^{0,\alpha}(\overline{\bm{\Omega}})} d^{2\alpha-3}|\log(d)|^2. 
\end{align}
Next, we proceed as follows
\begin{align}
    |\textit{Err}_\textit{ef,i,2}| 
    \nonumber&:= \Big|\sum_{j\ne i}^{[d^{-3}]}\int_{\Omega_i}\Big(\bm{\Upsilon}^{(k)}(z_i,z_j)-\bm{\Upsilon}^{(k)}(z_i,y)\Big)\cdot \mathbcal{A}_B\cdot\textit{\bm{\mathrm{E}}}_f (y)\Big|
    \\ &\nonumber\lesssim \Vert \textit{\bm{\mathrm{E}}}_f\Vert_{L^\infty(\bm{\Omega})} |\mathbcal{A}_B| \sum_{j\ne i}^{[d^{-3}]}\int_{\Omega_i}\Big|\int_0^1 \underset{y}{\nabla}\bm{\Upsilon}^{(k)}(z_i,z_j+\theta(y-z_j))\cdot(y-z_j)d\theta\Big|dy
    \\ \nonumber &\lesssim \Vert \textit{\bm{\mathrm{E}}}_f\Vert_{L^\infty(\bm{\Omega})} |\mathbcal{A}_B| \textit{Vol}(\Omega_j)d \sum_{j\ne i}^{[d^{-3}]} d_{ij}^{-4}
    \\ \nonumber &= \Vert \textit{\bm{\mathrm{E}}}_f\Vert_{L^\infty(\bm{\Omega})} |\mathbcal{A}_B| \textit{Vol}(\Omega_j)d \Bigg(\sum_{j\ne i}^{\textit{A}_1} d_{ij}^{-4} + \sum_{j\ne i}^{\textit{A}_2} d_{ij}^{-4}\Bigg),
\end{align}
where, $\textit{A}_1$ represents the numbers of subdomains $\Omega_j$ such that $\Omega_j \cap \mathcal{B}^0_j \ne {\varphi}$, with $\mathcal{B}^0_j$ being a small neighborhood of $\mathcal{B}_j$, for which there exist $r \in ]0,1[$ such that
\begin{align}
    \textit{Vol}(\mathcal{B}^0_j) = \mathcal{O}(d^{3r}),\; \text{and}\; \textit{diam}(\mathcal{B}^0_j) = \mathcal{O}(d^r),
\end{align}
and we consider $\textit{A}_1$ is the number of sub-domains $\Omega_j$ such that $\Omega_j \cap \mathcal{B}^0_j = {\varphi}.$ Furthermore, the number of subdomains $\textit{A}_1$ satisfies $\textit{A}_1 = \mathcal{O}(d^{3r-3}),$ as shown in \cite[Section 4]{cao-sini}. Squaring the above expression and summing over $i=1\ \text{to}\ [d^{-3}]$ yields
\begin{align}
    \sum_{i=1}^{[d^{-3}]} |\textit{Err}_\textit{ef,i,2}|^2
    \nonumber&\lesssim \Vert \textit{\bm{\mathrm{E}}}_f\Vert^2_{L^\infty(\bm{\Omega})} |\mathbcal{A}_B|^2 \textit{Vol}(\Omega_j)^2d^2 \Big(\textit{A}_1\sum_{j\ne i}^{\textit{A}_1}\sum_{j\ne i}^{[d^{-3}]}d_{ij}^{-8} + d^{-8r-3}\Big)
    \\ &\lesssim \Vert\textit{\bm{\mathrm{E}}}_f\Vert^2_{L^\infty(\bm{\Omega})} |\mathbcal{A}_B|^2 \textit{Vol}(\Omega_j)^2d^2\Big(\textit{A}_1^2d^{-8} + d^{-8r-3}\Big)
\end{align}
Now, choosing $r = \frac{11}{14}$ results in 
\begin{align}\label{r2}
    \sum_{i=1}^{[d^{-3}]} |\textit{Err}_\textit{ef,i,2}|^2
    \lesssim\Vert\textit{\bm{\mathrm{E}}}_f\Vert^2_{L^\infty(\bm{\Omega})} |\mathbcal{A}_B|^2 d^{-\frac{9}{7}}.
\end{align}
Subsequently, combining the estimates (\ref{r1}) and (\ref{r2}), we arrive at
\begin{align}
    \sum_{i=1}^{[d^{-3}]} |\textit{Err}_i^{(3)}|^2
    \lesssim \Vert\textit{\bm{\mathrm{E}}}_f\Vert^2_{L^\infty(\bm{\Omega})} |\mathbcal{A}_B|^2 d^{-\frac{9}{7}} + |\mathbcal{A}_B|^2 \big[\textit{\bm{\mathrm{E}}}_f\big]^2_{C^{0,\alpha}(\overline{\bm{\Omega}})} d^{2\alpha-3}|\log(d)|^2.
\end{align}
To estimate the term $\textit{Err}^{(1)}_i$, we utilize the following counting lemma, as described in \cite[Lemma 6.2]{cao-sini}.
\begin{lemma}
For any fixed $z_i, i=1,2,\ldots,[d^{-3}]$, the following counting estimate holds
    \begin{align}
        \sum_{i=1}^{[d^{-3}]}\Big|\int_{\mathbf{\Omega}\setminus\bigcup\limits_{j=1}^{[d^{-3}]}\Omega_j}\bm{\Upsilon}^{(k)}(z_i,y)dy\Big|^2 = \mathcal{O}(d^{-1}).
    \end{align}
\end{lemma}
\noindent
Therefore, we have
    \begin{align}
          |\textit{Err}^{(1)}_i| 
          \nonumber&\lesssim \Big|\int_{\mathbf{\Omega}\setminus\bigcup\limits_{j=1}^{[d^{-3}]}\Omega_j}\bm{\Upsilon}^{(k)}(z_i,y)\cdot \mathbcal{P}_B \cdot\widehat{\textit{\bm{\mathrm{E}}}}_f (y)\Big| 
          \\ \nonumber&\lesssim \Vert\textit{\bm{\mathrm{E}}}_f\Vert_{L^\infty(\bm{\Omega})} |\mathbcal{A}_B| \int_{\mathbf{\Omega}\setminus\bigcup\limits_{j=1}^{[d^{-3}]}\Omega_j}\bm{\Upsilon}^{(k)}(z_i,y)dy.
    \end{align}
Consequently, taking square, summing up to $[d^{-3}]$, and using the previous lemma, we deduce
    \begin{align}\label{r3}
          \sum_{i=1}^{[d^{-3}]}|\textit{Err}^{(1)}_i|^2 
          \nonumber&\lesssim \Vert\textit{\bm{\mathrm{E}}}_f\Vert^2_{L^\infty(\bm{\Omega})} |\mathbcal{A}_B|^2 \sum_{i=1}^{[d^{-3}]}\Big|\int_{\mathbf{\Omega}\setminus\bigcup\limits_{j=1}^{[d^{-3}]}\Omega_j}\bm{\Upsilon}^{(k)}(z_i,y)dy\Big|^2
           \\ &\lesssim \Vert\textit{\bm{\mathrm{E}}}_f\Vert^2_{L^\infty(\bm{\Omega})} |\mathbcal{A}_B|^2 d^{-1}.
    \end{align}
We now recall that
\begin{align}
        |\textit{Err}^{(4)}_i| 
        \nonumber&:= \Bigg|\int_{\Omega_i}\Big(\bm{\Upsilon}^{(k)}(z_i,y)-\bm{\Upsilon}^{(0)}(z_i,y)\Big)\cdot \mathcal{P}_B\cdot\hat{\textit{\bm{\mathrm{E}}}}_f (y)dy\Bigg|
        \\ \nonumber&\lesssim \Vert \textit{\bm{\mathrm{E}}}_f\Vert_{L^\infty(\overline{\bm{\Omega}})} |\mathbcal{A}_B| \int_{D_j}\frac{1}{|z_i-y|}\ dy
        \\ &\lesssim \Vert \textit{\bm{\mathrm{E}}}_f\Vert_{L^\infty(\overline{\bm{\Omega}})} |\mathbcal{A}_B|\ d^2,
\end{align}
which further implies after taking the square modulus and summing from $1$ to $[d^{-3}]$ as
    \begin{align}\label{r5}
           \sum_{i=1}^{[d^{-3}]} |\textit{Err}^{(5)}_i|^2 \lesssim |\mathbcal{A}_B|^2\ \Vert \textit{\bm{\mathrm{E}}}_f\Vert_{L^\infty(\overline{\bm{\Omega}})}\ d.
    \end{align}
Next, in order to estimate the following term, we use $C^{0,\alpha}$-regularity of $\bm{E}_f$ for $0<\alpha<1$. Consequently, we derive
    \begin{align}
         |\textit{Err}^{(2)}_i| 
         \nonumber&:= \Big|\nabla\int_{\Omega_i}\nabla\mathbcal{G}^{(0)}(z_i,y)\cdot \mathcal{P}_B \cdot\big(\widehat{\textit{\bm{\mathrm{E}}}}_f (y)-\widehat{\textit{\bm{\mathrm{E}}}}_f (z_i)\big)dy\Big|
         \\ &\lesssim |\mathcal{A}_B| \big[\textit{\bm{\mathrm{E}}}_f\big]_{C^{0,\alpha}} \int_{\Omega_i} \frac{1}{|z_i-y|^{3-\alpha}} dy,
    \end{align}
and by a scaling we derive the following estimate
    \begin{align}\label{r4}
         |\textit{Err}^{(2)}_i| \lesssim |\mathcal{A}_B| \big[\textit{\bm{\mathrm{E}}}_f\big]_{C^{0,\alpha}} \textit{Vol}(\Omega_i)d^{\alpha-3}.
    \end{align}
Similar to the previous estimates, we square the above expression and summing from $1$ to $[d^{-3}]$, we obtain
   \begin{align}
         \sum_{i=1}^{[d^{-3}]} |\textit{Err}^{(2)}_i|^2 \lesssim |\mathbcal{A}_B|^2 \big[\textit{\bm{\mathrm{E}}}_f\big]^2_{C^{0,\alpha}(\overline{\bm{\Omega}})} d^{2\alpha-3}.
   \end{align}
Therefore, we obtain the following linear algebraic system 
\begin{align}
           \Big(\widehat{\textit{\bm{\mathrm{E}}}}_f (z_i) - \Tilde{\mathbcal{Q}}_i\Big)
           &\nonumber- \eta\sum_{j\ne i}^M \bm{\Upsilon}^{(k)}(z_i,z_j)\cdot \mathbcal{P}_{D_j} \cdot \Big(\widehat{\textit{\bm{\mathrm{E}}}}_f (z_i) - \Tilde{\mathbcal{Q}}_j\Big)
          = \mathcal{R}_i,\; \texttt{for}\ j=1,2,\ldots,M,
    \end{align}
where, based on the estimates (\ref{r1}), (\ref{r2}), (\ref{r3}), (\ref{r5}) and (\ref{r4}), we estimated the term $\mathcal{R}_i$ as
    \begin{align}
          \sum_{i=1}^{[d^{-3}]} |\mathcal{R}_i|^2
          &\nonumber= \mathcal{O}\Big(|\mathbcal{A}_B|^2 \big[\textit{\bm{\mathrm{E}}}_f\big]^2_{C^{0,\alpha}(\overline{\bm{\Omega}})} d^{2\alpha-3}|\log(d)|^2 
          + |\mathbcal{A}_B|^2 \big[\textit{\bm{\mathrm{E}}}_f\big]^2_{C^{0,\alpha}(\overline{\bm{\Omega}})} d^{2\alpha-3} 
          + \Vert\textit{\bm{\mathrm{E}}}_f\Vert^2_{L^\infty(\bm{\Omega})} |\mathbcal{A}_B|^2 d^{-\frac{9}{7}} 
          \\ &+ \Vert\textit{\bm{\mathrm{E}}}_f\Vert^2_{L^\infty(\bm{\Omega})} |\mathbcal{A}_B|^2 d^{-1} 
          + |\mathbcal{A}_B|^2\ \Vert \textit{\bm{\mathrm{E}}}_f\Vert_{L^\infty(\overline{\bm{\Omega}})} \ d\Big)
    \end{align}
In conclusion, due to the invertibility condition of the linear algebric system (\ref{general-al-equn}) and by the previous estimate, we deduce that
    \begin{align}
          \sum_{i=1}^{[d^{-3}]}\Vert \widehat{\textit{\bm{\mathrm{E}}}}_f (z_i) - \Tilde{\mathbcal{Q}}_i\Vert_{\ell}^2 &\nonumber= \mathcal{O}\Big(|\mathbcal{A}_B|^2 \big[\textit{\bm{\mathrm{E}}}_f\big]^2_{C^{0,\alpha}(\overline{\bm{\Omega}})} d^{2\alpha-3}|\log(d)|^2 
          + |\mathbcal{A}_B|^2 \big[\textit{\bm{\mathrm{E}}}_f\big]^2_{C^{0,\alpha}(\overline{\bm{\Omega}})} d^{2\alpha-3} 
         + \Vert\textit{\bm{\mathrm{E}}}_f\Vert^2_{L^\infty(\bm{\Omega})} |\mathbcal{A}_B|^2 d^{-\frac{9}{7}} \\\ &+ \Vert\textit{\bm{\mathrm{E}}}_f\Vert^2_{L^\infty(\bm{\Omega})} |\mathbcal{P}_B|^2 d^{-1}+ |\mathbcal{A}_B|^2\ \Vert \textit{\bm{\mathrm{E}}}_f\Vert_{L^\infty(\overline{\bm{\Omega}})} \delta^4\ d^{-3}\Big).
    \end{align}
As, each term of the above expression is square summable, based on the above estimate, we therefore obtain that
    \begin{align}\label{estimate-proof}
          \nonumber& \sum_{i=1}^{[d^{-3}]} \Big\Vert\big|\widehat{\textit{\bm{\mathrm{E}}}}_f (z_i)\big|^2 -\big|\Tilde{\mathbcal{Q}}_i\big|^2\Big\Vert_{\ell^2}^2
          \\ &\lesssim \sup\limits_{B}\Big|\widehat{\textit{\bm{\mathrm{E}}}}_f (z_i) + \Tilde{\mathbcal{Q}}_i\Big|^2\sum_{i=1}^{[d^{-3}]}\Vert \widehat{\textit{\bm{\mathrm{E}}}}_f (z_i) - \Tilde{\mathbcal{Q}}_i\Vert_{\ell^2}^2 \lesssim d^{-\frac{9}{7}}.
    \end{align}


    \subsection{End of the Proof of Theorem \ref{non-periodic}: The Asymptotic Expansions} \label{end}   
    
This section begins by estimating the term $\displaystyle\sum_{\substack{i=1}}^M|\sigma^{(i)}(t)-\bm{\mathrm{Y}}(z_i,t)|^2$, where $(\sigma^{(i)})_{i=1}^M$ represents the vector-valued solution to the Volterra-type system of integral equations (\ref{matrix}) and $\bm{\mathrm{Y}}(z_i,t)$, $i=1,\ldots,M$, denotes the corresponding solution of the Lippmann-Schwinger equation (\ref{effective-equation}) with zero initial conditions. Next, using the fact that
$$\bm{\Omega}= \Omega_i\cup \Big(\bigcup_{j\ne i}^{[d^{-3}]} \Omega_j\Big)\cup \Big(\bm{\Omega}\setminus\bigcup_{j=1}^{[d^{-3}]}\Omega_j\Big)\; \text{and}\; \textit{Vol}(\Omega_j) = d^3,$$
we rewrite the expression (\ref{effective-equation}) in a discretized form at $\mathrm{x}=\mathrm{z}_i$
    \begin{align}
         \bm{\bm{\mathrm{Y}}} (\mathrm{z}_i,\mathrm{t}) + \sum\limits_{\substack{j=1 \\ j\neq i}}^M \overline{b}\ \delta^{3-\beta} \int_0^t\Phi^{(m)}(z_i,t;z_j,\tau)\ \frac{\partial}{\partial\tau}\bm{\mathrm{Y}}(z_j,\tau)\ d\tau = \mathbcal{F}(z_i,t) + E_{(\mathbf{1})} + E_{(\mathbf{2})} + E_{(\mathbf{3})},
    \end{align}
or, equivalently,
    \begin{align}\label{effective-heat-equation}
         \bm{\bm{\mathrm{Y}}} (\mathrm{z}_i,\mathrm{t}) + \sum\limits_{\substack{j=1 \\ j\neq i}}^M \overline{b}\ \delta^{3-\beta} \int_0^t\Phi^{(m)}(z_i,t;z_j,\tau)\ \frac{\partial}{\partial\tau}\bm{\mathrm{Y}}(z_j,\tau)\ d\tau
        \nonumber&= \overline{a} \, f(t) \, \delta^{\beta - h} \, \mathbcal{P}_{B} \cdot \widehat{\bm{E}}_f(z_i)\cdot \overline{\widehat{\bm{E}}}^\textit{Tr}_f(z_i) 
        \\&+ E_{(\mathbf{1})} + E_{(\mathbf{2})} + E_{(\mathbf{3})},
    \end{align}
where, we set $\widehat{\bm{E}}_f:= \mathbcal{P}_B^{-1}\cdot \mathbcal{A}_B\cdot \bm{E}_f$, 
$$E_{(\mathbf{1})}:= -\int_0^t\int_{\mathbf{\Omega}\setminus\bigcup\limits_{j=1}^{[d^{-3}]}\Omega_j}\overline{b}\ \Phi^{(m)}(z_i,t;y,\tau)\ \frac{\partial}{\partial\tau}\bm{\mathrm{Y}}(y,\tau)\ dyd\tau,$$
$$E_{(\mathbf{2})}:= -\int_0^t\int_{\Omega_j}\overline{b}\ \Phi^{(m)}(z_i,t;y,\tau)\ \frac{\partial}{\partial\tau}\bm{\mathrm{Y}}(y,\tau)\ dyd\tau,$$
and
$$E_{(\mathbf{3})}:=-\int_0^t\sum\limits_{\substack{j=1 \\ j\neq i}}^M \int_{\Omega_i}\overline{b}\ \Phi^{(m)}(z_i,t;y,\tau)\ \frac{\partial}{\partial\tau}\bm{\mathrm{Y}}(y,\tau)\ dyd\tau + \sum\limits_{\substack{j=1 \\ j\neq i}}^M \overline{b}\ \delta^{3-\beta} \int_0^t\Phi^{(m)}(z_i,t;z_j,\tau)\ \frac{\partial}{\partial\tau}\bm{\mathrm{Y}}(z_j,\tau)\ d\tau.$$
\noindent
Before estimating the terms introduced earlier, we first recall a useful result, stated as follows:  
\begin{lemma}\cite{habib-sini} \label{counting}  
\textbf{Counting Lemma.}  
For any arbitrary distribution of points \(z_j, j = 1, \ldots, M\), within a bounded domain of \(\mathbb{R}^3\) (with a prescribed minimum distance \(d\) between any two points), the following estimates hold uniformly with respect to \(i\):  
\begin{align}  
    \sum\limits_{\substack{j=1 \\ j \neq i}}^{M} \frac{1}{|z_i - z_j|^k} =  
    \begin{cases}  
        \mathcal{O}(d^{-3})\; &\text{if}\; k < 3, \\  
        \mathcal{O}\Big(d^{-3}(1 + |\log(d)|)\Big)\; &\text{if}\; k = 3, \\  
        \mathcal{O}(d^{-k})\; &\text{if}\; k > 3.  
    \end{cases} \nonumber  
\end{align}  
\end{lemma} 
\noindent
To estimate the term \(E_{(\mathbf{1})}\), we adopt the framework and notations introduced in \cite{sini-haibing, sini-wang-yao}. The analysis proceeds by addressing the following two scenarios:


\begin{figure}
\begin{center}
\begin{tikzpicture}[scale=1.2]

\draw (0,0) ellipse (3cm and 1.5cm);
\node[scale=1] at (0,-1.8) {$\mathbf{\Omega}$};

\draw[<-,red, thick] (1.6, 1.15) -- (2.5, 1.5);
\node[above, scale=0.5] at (2.5,1.5) {$\aleph_{(1)}\ (\text{The region above the blue line})$};

\draw[<-,red, thick] (-1.6, -1.15) -- (-2.5, -1.5);
\node[above, scale=0.5] at (-2.5,-1.75) {$\aleph_{(2)} (\text{The region below the blue line})$};

\draw[<-,red, thick] (0.1,1.2) -- (1, 1.8);
\node[above, scale=0.5] at (1,1.8) {$\mathrm{D}_i$};

\node[above, scale=.5] at (0,0.8) {$z_i$};
\node[above, scale=0.5] at (0,0.9) {$\bullet$};

\coordinate (B) at (0,1); 
\draw (B) circle (0.2);

\draw[dotted, blue, line width=1.2pt] (-3.5,0.5) -- (3.5,0.5);

\clip (0,0) ellipse (3cm and 1.5cm);

\foreach \x in {-3.5,-2.5,...,3.5} {
    \draw (\x,-2) -- (\x,2);
}
\def\xellip(#1){4*sqrt(1 - (#1/2)^2)}

\foreach \y in {-1.75,-1.25,...,1.75} {
    \draw ({\xellip(\y)},\y) -- ({-\xellip(\y)},\y);
}

\begin{scope}
    \clip (2.5,-1.5) rectangle (3.5,1.5);
    \foreach \i in {-1.5,-1.3,...,3.5} {
        \draw[black, thick] (\i,-0.8) -- ++(3.5,1.5);
    }
\end{scope}
\begin{scope}
    \clip (-2.9,-1.5) rectangle (3.5,-1.25);
    \foreach \i in {-2.9,-2.65,...,0} {
        \draw[black, thick] (\i,-1.7) -- ++(3.5,1);
    }
\end{scope}

\begin{scope}
    \clip (-3.9,-1.5) rectangle (-2.5,2.7);
    \foreach \i in {-3.9,-3.8,...,3.9} {
        \draw[black, thick] (\i,-.9) -- ++(3.5,4.9);
    }
\end{scope}

\begin{scope}
    \clip (-3,1) rectangle (5.5,2);
    \foreach \i in {-3,-2.8,...,2} {
        \draw[black, thick] (\i,1.25) -- ++(5.5,1.5);
    }
\end{scope}

\begin{scope}
    \clip (1.5,0.75) rectangle (5.5,1.5);
    \foreach \i in {-3,-2.8,...,2} {
        \draw[black, thick] (\i,0) -- ++(5.5,1.5);
    }
\end{scope}

\begin{scope}
    \clip (-6.5,0.75) rectangle (-1.5,1.5);
    \foreach \i in {-6.5,-6.3,...,2} {
        \draw[black, thick] (\i,0) -- ++(5.5,1.5);
    }
\end{scope}

\begin{scope}
    \clip (-6.5,-2) rectangle (-1.5,-.75);
    \foreach \i in {-6.5,-6.3,...,2} {
        \draw[black, thick] (\i,-.5) -- ++(5.5,-2);
    }
\end{scope}

\begin{scope}
    \clip (1.5,-.75) rectangle (5.5,-5);
    \foreach \i in {-3,-2.8,...,5} {
        \draw[black, thick] (\i,0) -- ++(5.5,-3);
    }
\end{scope}

\draw (0,0) ellipse (3cm and 1.5cm);

\end{tikzpicture}
\end{center}
\caption{A schematic representation of the partitioning of the region \(\mathbf{\Omega} \setminus \bigcup\limits_{j=1}^{[d^{-3}]} \Omega_j\). } \label{pic1}
\end{figure}    


\begin{enumerate}
    \item \textit{Case 1: $z_i$ Far from the Boundary $\partial\mathbf{\Omega}$ }.
    \\
         When the point \( z_i \) is sufficiently distant from the boundary \( \partial\mathbf{\Omega} \), the function \( |z_i - z|^{-1} \) remains bounded in the vicinity of the boundary. Consequently, in this scenario, the error term \( E_{(\mathbf{1})} \) scales as  
         \[
         E_{(\mathbf{1})} = \mathcal{O}\left(\text{vol}\left(\mathbf{\Omega} \setminus \bigcup_{j=1}^{\lfloor d^{-3} \rfloor} \Omega_j\right)\right) = \mathcal{O}(d),
         \]  
         where the volume of the excluded region is directly proportional to \( d \), assuming \( d \ll 1 \).
    \item \textit{Case 2: $z_i$ Near the Boundary $\partial\mathbf{\Omega}$.}     
    \\
         When \( z_i \) approaches the boundary \( \partial\mathbf{\Omega} \) or is near one of the sub-regions \( \Omega_j \), the estimation is divided into two components. Let the contribution from the \( \Omega_j \)-regions near \( z_i \) be denoted as \( \aleph_{(1)} \), and the contribution from the remainder of the domain as \( \aleph_{(2)} \).
         \begin{enumerate}
             \item \textit{Contribution from \( \aleph_{(2)} \):}
              The integral over \( \aleph_{(2)} \) is evaluated similarly to the previous case. Since \( \aleph_{(2)} \subset \mathbf{\Omega} \setminus \bigcup\limits_{j=1}^{\lfloor d^{-3} \rfloor} \Omega_j \), the volume of \( \aleph_{(2)} \), denoted as \( \text{vol}(\aleph_{(2)}) \), scales as \( d \) for \( d \ll 1 \).
              \item \textit{Contribution from \( \aleph_{(1)} \):}
              To estimate the integral over \( \aleph_{(1)} \), we first determine the number of \( \Omega_j \)-regions near \( z_i \). These regions are localized near a small segment of the boundary \( \partial\mathbf{\Omega} \). Assuming \( \partial\mathbf{\Omega} \) is sufficiently smooth, this segment can be approximated as flat and centered around \( z_i \). This flat region is partitioned into concentric square layers centered at \( z_i \), as illustrated in Figure \(\ref{pic1}\). 
              
              Given that the characteristic size of the flat region is of order unity relative to the parameter \( d \ll 1 \), and that the maximum edge length of the squares (or subregions \( \Omega_j \)) is \( d \), the number of layers is at most \( \lfloor d^{-1} \rfloor \). In the \( n^\text{th} \) layer (\( n = 0, \ldots, \lfloor d^{-1} \rfloor \)), the number of squares is at most \( (2n + 1)^2 \).  
              Excluding the innermost layer (\( n=0 \)), the number of squares in the \( n^\text{th} \) layer is bounded by  \[(2n + 1)^2 - (2n - 1)^2,\]  and their minimum distance from \( z_i \) is approximately \( n\left(d - \frac{d^3}{2}\right) \).  
         \end{enumerate}
    This framework provides a systematic approach for estimating the integral contributions from regions near and far from the boundary.     
\end{enumerate}
\noindent
Therefore, we express the following term and, using the singularity estimate with \(\beta > \frac{1}{2}\), we obtain
\[
    \int_0^t |\Phi^{(m)}(x,t;z,\tau)|\ d\tau = \mathcal{O}(|x-z|^{-1}),
\]
where \(\Phi^{(m)}(x,t;z,\tau)\) denotes the fundamental solution with the corresponding singular behavior. We have
\begin{align}
    \nonumber&|E_{(\mathbf{1})}| 
    \\ \nonumber&= \Bigg|\int_0^t\int_{\mathbf{\Omega}\setminus\bigcup\limits_{j=1}^{[d^{-3}]}\Omega_j}\overline{b}\ \Phi^{(m)}(z_i,t;y,\tau)\ \frac{\partial}{\partial\tau}\bm{\mathrm{Y}}(y,\tau)\ dyd\tau\Bigg|
    \\ \nonumber &\le \Bigg|\int_0^t\int_{\aleph_{(1)}}\overline{b}\ \Phi^{(m)}(z_i,t;y,\tau)\ \frac{\partial}{\partial\tau}\bm{\mathrm{Y}}(y,\tau)\ dyd\tau\Bigg| + \Bigg|\int_0^t\int_{\aleph_{(2)}}\overline{b}\ \Phi^{(m)}(z_i,t;y,\tau)\ \frac{\partial}{\partial\tau}\bm{\mathrm{Y}}(y,\tau)\ dyd\tau\Bigg|
    \\ \nonumber &\lesssim \sum\limits_{\substack{j=1}}^{[d^{-1}]}\frac{1}{d_{ij}}\Vert \frac{\partial }{\partial t}{\bm{\mathrm{Y}}\Vert_{\mathrm{L}^\infty\big(0,\mathrm{T};\mathrm{L}^\infty(\mathbf{\Omega})\big)}}\textit{vol}\big(\Omega_j\big)
    + \Vert \frac{\partial}{\partial t}{\bm{\mathrm{Y}}\Vert_{\mathrm{C}^1\big(0,\mathrm{T};\mathrm{L}^\infty(\mathbf{\Omega})\big)}} \Vert\Phi^{(m)}(z_i,t;y,\tau)\Vert_{\mathrm{L}^\infty\big(0,\mathrm{T};\mathrm{L}^\infty(\mathbf{A}_{(2)})\big)}\textit{vol}\big(\aleph_{(2)}\big)
    \\ \nonumber &\lesssim d^3\sum\limits_{\substack{j=1}}^{[d^{-1}]}\frac{1}{d_{ij}} + d
    \lesssim d^3 \sum\limits_{\substack{n=1}}^{[d^{-1}]}\big[(2n+1)^2-(2n-1)^2\big]\frac{1}{n(d-\frac{d^3}{2})} + d
    \lesssim d^3\cdot d^{-2} + d.  
\end{align}
Hence, we obtain 
\begin{align}\label{esti1}
    |E_{(\mathbf{1})}|  =  \mathcal{O}\big(d\big).
\end{align}
Next, Using Lemma \ref{prop1} and Corollary \ref{cor}, we have $\frac{\partial}{\partial t}{\bm{\mathrm{Y}}}\in \mathrm{L}^\infty\big(0,\mathrm{T};\mathrm{L}^\infty(\mathbf{\Omega})\big)$ and due to the singularity estimates introduced in (\ref{singularities}), we deduce that
\begin{align}
    |E_{(\mathbf{2})}| \nonumber&= \mathcal{O}\Big(\overline{b}\;\Vert \frac{\partial}{\partial t}{\bm{\mathrm{Y}}}\Vert_{\mathrm{L}^\infty\big(0,\mathrm{T};\mathrm{L}^\infty(\mathbf{\Omega})\big)}\int_{\Omega_i}|y-z_i|^{2r-3}\;dy\Big).
\end{align}
To analyze the term $\displaystyle\int_{\Omega_i}|y-z_i|^{2r-3}\;dy$, for $0<r<1$, we conclude that $\displaystyle\int_{\Omega_i}|y-z_i|^{2r-3}\;dy = \mathcal{O}(\delta^\frac{2r}{3}).$ Therefore, we deduce that 
\begin{align}\label{esti2}
    |E_{(\mathbf{2})}| = \mathcal{O}\Big(d^{2r}\Big).
\end{align}
We now proceed to estimate the third term, \(E_{(\mathbf{3})}\). Specifically, we have
    \begin{align}
          E_{(\mathbf{3})} 
          \nonumber&:= -\sum\limits_{\substack{j=1 \\ j\neq i}}^{[d^{-3}]}\overline{b}\int_{\Omega_j} \Bigg[\int_0^t \Phi^{(m)}(z_i,t;y,\tau)\ \frac{\partial}{\partial\tau}\bm{\mathrm{Y}}(y,\tau)\ dyd\tau 
          - \int_0^t\Phi^{(m)}(z_i,t;z_j,\tau)\ \frac{\partial}{\partial\tau}\bm{\mathrm{Y}}(z_j,\tau)\ d\tau\Bigg]\;dyd\tau.
         \\ \nonumber &= -\sum\limits_{\substack{j=1 \\ j\neq i}}^{[d^{-3}]}\overline{b}\int_{\Omega_j}\Bigg[ \underbrace{\frac{\partial}{\partial\tau}\bm{\mathrm{Y}}(y,\tau)\int_0^t\Big(\Phi^{(m)}(z_i,t;y,\tau)-\Phi^{(m)}(z_i,t;z_j,\tau)\Big)\;dyd\tau.}_{\textbf{err}_{(1)}}
         \\  &+ \underbrace{\int_0^t\Phi^{(m)}(z_i,t;z_j,\tau) \Big[\frac{\partial}{\partial t}{\bm{\mathrm{Y}} (y, \tau)}-\frac{\partial }{\partial t}{\bm{\mathrm{Y}} (z_j, \tau)}\Big]}_{\textbf{err}_{(2)}}\Bigg]\;dyd\tau.
    \end{align}
Using the singularity estimate of the heat fundamental solution, as outlined in \cite[Chapter 1]{friedman} and \cite[Chapter 9]{kress}, for \(0<\beta <1\), we obtain

\[
    \int_0^t |\nabla_y\Phi^{(m)}(z_i,t;z_j,\tau)|\ d\tau = \mathcal{O}(|x-z|^{2r-4}),
\]we estimate $\text{err}_{(1)}$ as
\begin{align}\label{part11}
    \textbf{err}_{(1)} := \Bigg|&\nonumber\sum\limits_{\substack{j=1 \\ j\neq i}}^{[d^{-3}]}\overline{b}\int_{\Omega_j}\frac{\partial}{\partial\tau}\bm{\mathrm{Y}}(y,\tau)\int_0^t\Big(\Phi^{(m)}(z_i,t;z_j,\tau)-\Phi^{(m)}(z_i,t;z_j,\tau)\Big)\;dyd\tau\Bigg|
    \\ \nonumber &= \Bigg|\frac{1}{4\pi}\sum\limits_{\substack{j=1 \\ j\neq i}}^{[d^{-3}]}\overline{b}\int_{\Omega_j}\frac{\partial}{\partial\tau}\bm{\mathrm{Y}}(y,\tau)\ (y-z_j)\int_0^t\nabla_y\Phi^{(m)}(z_i,t;z^*_j,\tau)\ dyd\tau\Bigg|,\quad \text{with}\; z_j^*\in \Omega_j
    \\ &= \mathcal{O}\Bigg(\overline{b}\sum\limits_{\substack{j=1 \\ j\neq i}}^{[d^{-3}]}\frac{1}{d^{-4+2r}_{ij}}\; \Vert \frac{\partial }{\partial t}{\bm{\mathrm{Y}}\Vert_{\mathrm{L}^\infty\big(0,\mathrm{T};\mathrm{L}^\infty(\mathbf{\Omega})\big)}}\; \int_{\Omega_j} |y-z_j|dy\Bigg) = \mathcal{O}\Bigg(d^4\sum\limits_{\substack{j=1 \\ j\neq i}}^{[d^{-3}]}\frac{1}{d^{-4+2r}_{ij}}\Bigg).
\end{align}
Next, we deduce using Taylor's series expansion that
\begin{align}
   \frac{\partial}{\partial t}{\bm{\mathrm{Y}} (y, \tau)}-\frac{\partial }{\partial t}{\bm{\mathrm{Y}} (z_j, \tau)} = (y-z_j)\;\nabla_{y}\frac{\partial }{\partial t}{\bm{\mathrm{Y}} (z^*_j, \tau)},
\end{align}
where, $z^*\in \Omega_j.$ 
\newline
Therefore, due to Corollary \ref{cor}, we have $\partial_{x_i}\frac{\partial}{\partial t}\bm{\mathrm{Y}} \in L^\infty\big(0,\mathrm{T};\mathrm{L}^p(\mathbf{\Omega})\big)$ for $p>3$ and using the singularity estimate with \(0<r <1\), we obtain

\[
    \int_0^t |\Phi^{(m)}(z_i,t;z_j,\tau)|\ d\tau = \mathcal{O}(|x-z|^{2r-3}),
\]
    \begin{align}
          \textbf{err}_{(2)} 
          &\nonumber:= \Bigg|\sum\limits_{\substack{j=1 \\ j\neq i}}^{[d^{-3}]}\overline{b}\int_0^t\int_{\Omega_j}(y-z_j)\ \Phi^{(m)}(z_i,t;z_j,\tau)\ \frac{\partial}{\partial_x}\frac{\partial}{\partial t}{\bm{\mathrm{Y}} (z_j^*, \tau)}\ dyd\tau \Bigg|
          \\ \nonumber &\lesssim \sum\limits_{\substack{j=1 \\ j\neq i}}^{[d^{-3}]} \Bigg(\int_0^t\int_{\Omega_j}(y-z_j)^q |\Phi^{(m)}(z_i,t;z_j,\tau)|^q\ dyd\tau\Bigg)^\frac{1}{q} \Vert \partial_{x_i}\frac{\partial}{\partial t}\bm{\mathrm{Y}}\Vert_{\mathrm{L}^\infty\big(0,\mathrm{T};\mathrm{L}^p(\mathbf{\Omega})\big)},\; \text{with}\ \frac{1}{p} + \frac{1}{q} =1.
    \end{align}
Now, due to the singularity estimate of $\Phi^{(m)}(z_i, \mathrm{t}; z_j, \tau)$, we have
\[
|\Phi^{(m)}(z_i, \mathrm{t}; z_j, \tau)| \lesssim \frac{1}{(t - \tau)^{qr}} \frac{1}{|z_i - z_j|^{q(3 - 2r)}},\; \text{with}\; z_i \ne z_j.
\]
The above function is integrable in \( \Omega_j \times (0, T) \) only when \( qr < 1 \). Then, by Corollary \ref{cor}, and choosing \( p \) sufficiently large, we take \( q \) close to 1. Considering \( r < 1 \), we then see that
\begin{align}\label{part2}
    \textbf{err}_{(2)}  &\lesssim \sum\limits_{\substack{j=1 \\ j\neq i}}^{[d^{-3}]}\frac{1}{d^{q(-3+2r)}_{ij}}\; \Vert \partial_{x_i}\frac{\partial}{\partial t}\bm{\mathrm{Y}}\Vert_{\mathrm{L}^\infty\big(0,\mathrm{T};\mathrm{L}^p(\mathbf{\Omega})\big)}\; \int_{\Omega_j} |y-z_j|^qdy\Bigg) = \mathcal{O}\Bigg(d^{3+q}\sum\limits_{\substack{j=1 \\ j\neq i}}^{[d^{-3}]}\frac{1}{d^{q(-3+2r)}_{ij}}\Bigg).
\end{align}
Further, based on the estimates (\ref{part11}) and (\ref{part2}), we derive the following estimate
\begin{align}\label{esti3}
    E_{(\mathbf{3})} = \mathcal{O}\Big(\sum\limits_{\substack{j=1 \\ j\neq i}}^{[d^{-3}]}\frac{1}{d^{-4+2r}_{ij}}\Big)d^4 = \mathcal{O}\Big(d\Big).
\end{align}
Comparing the equations (\ref{matrix}) and (\ref{effective-heat-equation}), we arrive at the following expression with $\mathbcal{B}_i(z_i,t) := \bm{\bm{\mathrm{Y}}}(\mathrm{z}_i,\mathrm{t}) -\sigma^{(i)}(t) $
\begin{align}\label{effective-heat-equation-2}
   \mathbcal{B}_i(z_i,t) + \sum\limits_{\substack{j=1 \\ j\neq i}}^M \overline{b}\ \delta^{3-\beta} \int_0^t\Phi^{(m)}(z_i,t;z_j,\tau)\ \frac{\partial}{\partial\tau}\mathbcal{B}_j(z_j,t)\ d\tau 
   = \overline{a}f(t)\delta^{\beta-h}\ \mathbcal{P}_{B}\Big(\big| \widehat{\bm{\textit{E}}}_f(z_i)\big|^2-\big|\Tilde{\mathbcal{Q}}_i\big|^2\Big) 
   + \mathcal{O}(d).
\end{align}
Gathering (\ref{esti1}), (\ref{esti2}) and (\ref{esti3}), we get
\begin{align}
    \sum_{i=1}^{[d^{-3}]} \Big(|E_{(\mathbf{1})}|^2 + |E_{(\mathbf{2})}|^2 + |E_{(\mathbf{3})}|^2\Big) = \mathcal{O}\Big(M\;d^2 + M\;d^{4r}\Big) = \mathcal{O}(d^{-1}).
\end{align}
Thus, applying the above estimate together with the invertibility property and the estimate (\ref{goodesti}) for the algebraic system (\ref{algebraic}), we use the same method to that applied to (\ref{effective-heat-equation-2}) to derive the following estimate as $d\ll 1$
\begin{align}
    \sum_{i=1}^M|\sigma^{(i)}(t)-\bm{\mathrm{Y}}(z_i,t)|^2 
    &= \mathcal{O}(d^{-1}) 
    + \mathcal{O}\Bigg(\delta^{2\beta-2h}\sum_{i=1}^{[d^{-3}]} \Big\Vert\big| \widehat{\bm{\textit{E}}}_f(z_i)\big|^2-\big|\Tilde{\mathbcal{Q}}_i\big|^2\Big\Vert_{\ell^2}^2\Bigg).
\end{align}
Then, based on the estimate (\ref{estimate-proof}) proved in Section (\ref{effective-max-section}), we have the following estimate
    \begin{align}
          \sum_{i=1}^{[d^{-3}]} \Big\Vert\big| \widehat{\bm{\textit{E}}}_f(z_i)\big|^2-\big|\Tilde{\mathbcal{Q}}_i\big|^2\Big\Vert_{\ell^2}^2 \lesssim  d^{-\frac{9}{7}}.
    \end{align}
Consequently, we derive that
    \begin{align}\label{mainesti}
          \sum_{i=1}^M|\sigma^{(i)}(t)-\bm{\mathrm{Y}}(z_i,t)|^2 
          = \mathcal{O}(d^{-1}) 
          + \mathcal{O}\Big(\delta^{2\beta-2h}d^{-\frac{9}{7}}\Big).
    \end{align}
We introduce the unknown variable $\bm{\mathrm{Y}} = \frac{\partial}{\partial t}\mathbf{U}$, where $\mathbf{U}$ satisfies the following Lippmann-Schwinger equation
\begin{equation}\nonumber
 \bm{\mathrm{U}} (\mathrm{x},\mathrm{t}) + \overline{b}\int_0^t\int_{\mathbf{\Omega}}\Phi^{(m)}(x,t;y,\tau)\ \frac{\partial}{\partial\tau}\bm{\mathrm{U}}(y,\tau)\ dyd\tau = \mathbcal{F}(\mathrm{x},\mathrm{t}),\mbox{ for } \mathrm{x} \in \mathbb{R}^3, \mathrm{t} \in (0, \mathrm{T}),
\end{equation}
with zero initial conditions for $\mathbf{U}$ up to the first order and define
Let us now define
\begin{align}
    \mathbcal{W}(x,t) := \begin{cases}
                       \mathbf{U}(x,t)  & \text{if}\; (x,t) \in \mathbf{\Omega}\times(0,T) \\ \displaystyle
                       \mathbcal{F}(\mathrm{x},\mathrm{t}) - \overline{b}\int_0^t\int_{\mathbf{\Omega}}\Phi^{(m)}(x,t;y,\tau)\ \frac{\partial}{\partial\tau}\bm{\mathrm{U}}(y,\tau)\ dyd\tau & \text{for}\; (x,t) \in \mathbb R^3\setminus\mathbf{\Omega}\times(0,T).
                 \end{cases}
\end{align}
From this point onward, our objective is to estimate \(|\mathbcal{W}^\textit{sc}(x,t) - u^\textit{sc}(x,t)|\). To proceed, we assume that the point \(x\) lies outside the region \(\mathbf{\Omega} \cup \{x_0\}\). Under this assumption, we have
\begin{align}
    \mathbcal{W}^\textit{sc}(x,t) 
    \nonumber &= -\overline{b}\int_0^t\int_{\mathbf{\Omega}}\Phi^{(m)}(x,t;y,\tau)\ \frac{\partial}{\partial\tau}\bm{\mathrm{U}}(y,\tau)\ dyd\tau
    \\ \nonumber &= -\overline{b}\int_0^t\int_{\mathbf{\Omega}}\Phi^{(m)}(x,t;y,\tau)\ \bm{\mathrm{Y}}(y,\tau)\ dyd\tau
    \\ \nonumber &=-\sum\limits_{\substack{i=1 }}^{[d^{-3}]}\overline{b}\ \delta^{3-\beta}\;\int_0^t\Phi^{(m)}(x,t;z_i,\tau)\bm{\mathrm{Y}}(z_i,\tau)d\tau
    - \int_0^t\int_{\mathbf{\Omega}\setminus\bigcup\limits_{i=1}^{[d^{-3}]}\Omega_i}\overline{b}\ \Phi^{(m)}(x,t;z_i,\tau)\ \frac{\partial}{\partial\tau}\bm{\mathrm{Y}}(y,\tau)\ dyd\tau
    \\ \nonumber&- \sum\limits_{\substack{i=1 }}^{[d^{-3}]}\overline{b}\int_0^t\int_{\Omega_i} \Big(\Phi^{(m)}(x,t;y,\tau)\ \bm{\mathrm{Y}}(y,\tau) -  \Phi^{(m)}(x,t;z_i,\tau)\bm{\mathrm{Y}}(z_i,\tau)\Big)\;dyd\tau.
\end{align}
Using similar techniques as those employed to estimate \(E_{(\mathbf{1})}\) and \(E_{(\mathbf{3})}\), it can be shown that the second and third terms in the above expression are of the order \(\mathcal{O}(d)\) as \(d \ll 1\). Consequently, we deduce that
    \begin{align}\label{conclude}
         \nonumber&\mathbcal{W}^\textit{sc}(x,t) 
         \\ \nonumber&= -\sum\limits_{\substack{i=1 }}^{[d^{-3}]}\overline{b}\ \delta^{3-\beta}\;\int_0^t\Phi^{(m)}(x,t;z_i,\tau)\bm{\mathrm{Y}}(z_i,\tau)d\tau  + \mathcal{O}(d)
         \\ \nonumber&= -\sum\limits_{\substack{i=1 }}^{[d^{-3}]}\overline{b}\ \delta^{3-\beta}\;\int_0^t\Phi^{(m)}(x,t;z_i,\tau)\sigma^{(i)}(\tau)d\tau 
         + \underbrace{\sum\limits_{\substack{i=1 }}^{[d^{-3}]}\overline{b}\ \delta^{3-\beta}\;\int_0^t\Phi^{(m)}(x,t;z_i,\tau)\ \big(\sigma^{(i)}-\bm{\mathrm{Y}}(z_i,\tau)\big)(\tau)d\tau}_{:= \textbf{err}_{(3)}}
         \\ &+ \mathcal{O}(d).
   \end{align}
We then apply the Cauchy-Schwarz inequality along with the estimate (\ref{mainesti}) to derive the following bound
    \begin{align}
          \textbf{err}_{(3)} 
          \nonumber&:= \sum\limits_{\substack{i=1 }}^{[d^{-3}]}\overline{b}\ \delta^{3-\beta}\;\int_0^t\Phi^{(m)}(x,t;z_i,\tau)\ \big(\sigma^{(i)}-\bm{\mathrm{Y}}(z_i,\tau)\big)(\tau)d\tau 
          \\ &\nonumber= \mathcal{O}\Bigg(\delta^{3-\beta}\ \Big(\sum\limits_{\substack{i=1}}^{[d^{-3}]}\int_0^t|\Phi^{(m)}(x,t;z_i,\tau)|^2d\tau\Big)^\frac{1}{2}\;\Big(\sum\limits_{\substack{i=1}}^{[d^{-3}]}|\sigma^{(i)}-\bm{\mathrm{Y}}(z_i,\tau)|^2\Big)^\frac{1}{2}\Bigg)
          \\ &\nonumber = \mathcal{O}(\delta^{3-\beta}\ d^{-\frac{3}{2}}\ \delta^{2\beta-2h} d^{-\frac{9}{7}}).
    \end{align}
Consequently, due to the chosen regime as mentioned in (\ref{scales-h-beta}) i.e.
\begin{align}\nonumber
    h = \beta\; \text{and}\; d\sim \delta^{1-\frac{\beta}{3}},
\end{align}
we conclude from (\ref{conclude}) that
\begin{align}\nonumber
        u^\textit{sc}(x,t) - \mathbcal{W}^\textit{sc}(x,t) = \mathcal{O}(\delta^{\frac{2}{7}(3-\beta)}\; \text{as}\; \delta\to 0,
    \end{align}
which completes the proof.\qed

\bigskip

\noindent
\textit{\textbf{Acknowledgments.}}
\\
\textit{\textbf{Data Availability Statement.}} Data sharing is not applicable to this article as no datasets were generated or analyzed during the current study.
\bigskip

\noindent
\textit{\textbf{Declarations.}}
\newline
\textit{\textbf{Conflict of interest.}} The authors declare that they have no conflict of interest.

\end{document}